\documentclass{article}
\usepackage{geometry}

\usepackage[utf8]{inputenc}
\usepackage{mathtools}
\usepackage{amsmath}
\usepackage{amsthm}
\usepackage{amssymb}
\usepackage{amsfonts}
\usepackage{upgreek}
\usepackage{centernot}
\usepackage{stmaryrd}
\usepackage{comment}
\usepackage{enumerate}
\usepackage[shortlabels]{enumitem}
\usepackage[linguistics]{forest}
\usepackage{float}
\usepackage{wasysym}
\usepackage{caption}
\usepackage{subcaption}
\usepackage{tikz-qtree}
\usepackage{adjustbox}
\usepackage{graphicx}
\usepackage{color,soul}

\usepackage{epsfig}

\usepackage{accents}

\theoremstyle{plain}
\newtheorem{theorem}{Theorem}[section]
\newtheorem{lemma}[theorem]{Lemma} 
\theoremstyle{definition}
\newtheorem{defn}[theorem]{Definition} 
\newtheorem{prop}[theorem]{Proposition}
\newtheorem{eg}[theorem]{Example}
\newtheorem{rmk}[theorem]{Remark} 

\newtheorem{cor}[theorem]{Corollary}

\def\R{\mathbb{R}}
\def\C{\mathbb{C}}

\def\cc{\mathcal{C}}

\def\ff{\mathcal{F}}

\def\yy{\mathcal{Y}}
\def\pp{\mathcal{P}}

\def\zz{\mathcal{Z}}
\usepackage{tikz-cd}
\tikzset{
  symbol/.style={
    draw=none,
    every to/.append style={
      edge node={node [sloped, allow upside down, auto=false]{$#1$}}}
  }
}
\usepackage{mathdots}
\usepackage{young}

\usepackage{hyperref}
\hypersetup{colorlinks,linkcolor={red},citecolor={blue},urlcolor={red}}  

\def\ss{\mathcal{S}}
\title{Idempotent Completion of Persistence Categories}
\author{John Miller}
\date{\today}
\begin{document}

\maketitle
\begin{abstract}
    This paper studies how persistence categories and triangulated persistence categories behave with respect to taking idempotent completions. In particular we study whether the idempotent completion (i.e. Karoubi envelope) of categories admitting persistence refinement also admits such a refinement. In doing so, we introduce notions of persistence semi-categories and persistent presheaves and explore their properties.
\end{abstract}

\section{Introduction}

\text{Obj}ects in geometry are often studied via algebraic machinery, for example through the construction of categories or homology theories. These methods give a rich source of information allowing us to differentiate geometric objects by differences in these algebraic constructions. However, we end up forgetting a large amount of geometric information in these methods. The theory of persistence is a way to incorporate more of this geometric information into our algebraic machinery. The type of extra information we include is the notion of `size' or `weight'. Persistence homology was introduced by H. Edelsbrunner, D. Letscher and A. Zomorodian in 2000 \cite{ELZ} as a tool in topological data analysis. The notion was generalised to a more abstract mathematical object called a persistence module by  G. Carlsson and A. Zamorodian \cite{ZC} a few years later. Since then, the use of persistence modules has seen numerous applications in various areas of mathematics, in geometry and topology in particular. A persistence module can be viewed as a functor $\pp: (\R,<) \to R-\text{\text{Mod}}$, from the `poset category' of real numbers to the category of $R$-modules. Usually $R$ is chosen to be a field as to work with vector spaces. This allows one to view a persistence module as a `barcode'. Recently, P. Biran, O. Cornea and J. Zhang \cite{BCZ} introduced the idea of a persistence category, and furthermore a triangulated persistence category. A persistence category $\cc$ is a category enriched over persistence modules, that is, for every $A,B\in \text{Obj}(\cc)$ there is a functor $\text{Hom}(A,B)(-):(\R,<) \to R\text{-\text{Mod}}$. Composition in persistence categories respects this structure, that is if $f\in \text{Hom}(A,B)(r)$ and $g\in \text{Hom}(B,C)(s)$ then $g \circ f\in \text{Hom}(A,C)(r+s)$. The main sources of examples of persistence categories are filtered $A_\infty$ categories or filtered dg (differential graded) categories. There are two natural categories associated to a given persistence category, namely, the zero level category $\cc_0$, whose objects  remain the same but \text{Hom}-sets are given by $\text{Hom}_{\cc_0}(A,B)=\text{Hom}_\cc(A,B)(0)$. The other is the limit category $\cc_\infty$, this is a category where we `forget' the persistence structure and that is given by the localisation of $\cc_0$ with respect to a certain subcategory. A category $\cc$ is said to admit a persistence refinement if it can be realised as the limit category of some persistence category. 

This paper explores how categories admitting persistence refinement behave with respect to idempotent completion, this is the completion of a category to a category where every idempotent splits. Idempotent completion of general categories appears in various areas of mathematics. In algebraic geometry, in particular the study of Weil cohomology theories of smooth projective varieties, it is known (an idea conjectured and proved by Grothendieck, though he never published any written work on this) that any such theory factors through a certain category, namely the category of Chow motives. Roughly this is given by the idempotent completion of the category of correspondences of schemes. For an exposition, see \cite{Mi}. The main motivation for considering idempotent completion from the point of view of this paper comes from symplectic geometry, in particular homological mirror symmetry. In 1994 Maxim Kontsevich (\cite{Ko}) conjectured that relations seen between algebraic and symplectic geometry exhibited by `mirror' manifolds can be expressed in terms of an equivalence of triangulated categories. Explicitly, the conjecture states that the split-closed (idempotent completion of the) derived Fukaya category of one manifold is equivalent to the bounded derived category of coherent sheaves of the other, and vice versa. The conjecture has been proved for certain mirror manifolds. For example, Polishchuk and
Zaslow (\cite{PZ}) in 2000 proved it for certain elliptic curves and tori. There are examples of symplectic manifolds whose Fukaya categories are known to already be idempotent complete. In \cite{AS} Auroux and Smith show that the (wrapped) Fukaya categories of a collection of punctured surfaces with at least three punctures are all idempotent complete. Other related works that study idempotents in the Fukaya category include a celebrated result of Abouzaid( \cite{Ab}) who gave a criteria for when the Fukaya category is split-generated (generated as a triangulated category up to splitting of an idempotent) by an object. Biran, Cornea and Zhang's work on defining fragmentation metrics via persistence categories was motivated from phenomena exhibited in the case of the Fukaya category for certain symplectic manifold. Distances between Lagrangians had been studied via the notion of `shadows' of Lagrangian cobordisms by Biran, Cornea and Shelukhin (\cite{BCS}). Their paper in some sense gives a more algebraic formalism of this. In general the Fukaya category need not be already idempotent complete, thus the aim of this paper is to study if the persistence category structure and fragmentation distances survive the passing to idempotent completion.

We will study the following questions; if $\cc$ admits a persistence refinement, does it admit an idempotent completion which has a persistence refinement extending that of $\cc$? Furthermore does its Karoubi envelope, which is the universal idempotent completion, admit a persistence refinement? i.e., given a persistence category $\cc$ we wish to find a persistence category $\hat{\cc}$ whose limit category is idempotent complete and functors $i$ and $\hat{i}$ such that the following commutes (up to some equivalence).

\begin{equation}
     \begin{tikzcd}
         \cc_\infty \ar[r,"i"]&\hat{ \cc}_\infty\\
         \cc\ar[u,"\lim"]\ar[r,"\hat{i}"] & \hat{\cc}\ar[u,"\lim"]
     \end{tikzcd}
 \end{equation}
Moreover, ideally we would construct $\hat{\cc}$ such that its limit category is equivalent to the Karoubi envelope of $\cc_\infty$. The answer to the first question is positive, in fact we show that the Karoubi envelope of $\cc_\infty$ is equivalent to a category admitting persistence refinement.

\begin{theorem}\label{mainth1}
    Given a persistence category $\cc$, there exists an equivalence $\tilde{\zz}:\text{Split}(\cc_\infty) \rightarrow \hat{\cc}_\infty $ to the limit category of a persistence category $\hat{\cc}$ where $\hat{\cc}_\infty$ is idempotent complete.
\end{theorem}

The Karoubi completion of a general category can be realised in two equivalent ways: the first is by a canonical enlargement of $\cc$ to a category of pairs $(A,e)$ where $A\in \text{Obj}(\cc)$ and $e:A \to A$ is an idempotent in $\cc$. A morphism in $\text{Split}(\cc)$ $f:(A,e_A) \to (B,e_B)$ is a morphism $f:A \to B$ in $\cc$ such that $e_B \circ f = f = f \circ e_A$. The identity on $(A,e_A)$ is given by $e_A$.  The canonical inclusion $i: \cc \to \text{Split}(\cc)$ given by $i(A)=(A,id_A)$ is fully faithful. An idempotent $e_A: A \to A$ now splits in $\text{Split}(\cc)$ as 
 $(A,id_A) \xrightarrow{e_A}(A,e_A)\xrightarrow{e_A}(A,id_A)$. Section \ref{sectpcatpair} discusses how to try to construct a persistence category with limit category the Karoubi envelope of $\cc_\infty$ viewed as the category of pairs. It becomes apparent that we run into an issue, as such a persistence refinement cannot have identities. We therefore define the notion of \textbf{persistence semi-categories}, and show that there exists a persistence semi-category $\widehat{Split}_\pp(\cc)$ whose limit category recovers the Karoubi envelope of $\cc_\infty$ viewed as the persistence category of pairs.
 \begin{theorem}\label{mainth2}
     Given a persistence category $\cc$, the associated semi-persistence category $\widehat{\text{Split}}_\pp(\cc)$ has limit category equivalent to the Karoubi envelope of $\cc_\infty$.
 \end{theorem}
 
The other realisation of the Karoubi envelope of a category is via presheaves. One considers the full subcategory of presheaves on $\cc$, $\text{PSh}(\cc)$, spanned objects that are retracts of representable presheaves, i.e., objects of the form $\yy(A)=\text{Hom}_\cc(-,A)$ (see for example \cite{Lu}). In section \ref{sectppsheav} we construct a persistence category of \textbf{persistence presheaves} $\text{PSh}_\pp(\cc)$ associated to a persistence category $\cc$. Then consider a full subcategory $\text{PSh}_\pp(\cc)$ spanned by \text{weighted retracts}, which we denote by $\text{Split}_\pp(\cc)$. This will play the role of $\hat{\cc}$ in Theorem \ref{mainth1}. The main issue with constructing these persistence refinements is due to idempotents in $\cc_\infty$ not in the image of an idempotent in $\cc_0$, but rather the image of a \textbf{weighted idempotent} of non-zero weight. We show that in the case that all idempotents in $\cc_\infty$ are represented by weight zero idempotents in $\cc_0$, then we can find a persistence refinement of the Karoubi completion of $\cc_\infty$ in the usual sense.

\begin{theorem}\label{mainth3}
    Let $\cc$ be a persistence category, if every idempotent of $\cc_\infty$ is represented by an idempotent of weight zero in $\cc_0$ then $\text{Split}(\cc_\infty)$ is equivalent to $\text{Split}_\pp^0(\cc)_\infty$.
\end{theorem}

 The persistence category $\text{Split}_\pp^0(\cc)_\infty$ is the full subcategory of $\text{Split}_\pp(\cc)_\infty$ consisting of weight zero retracts of representable objects. With the assumption of idempotents being represented by weight zero idempotents, we find that $\widehat{\text{Split}}_\pp(\cc)$ is a persistence category in the usual sense, and moreover is equivalent, as a persistence category, to $\text{Split}_\pp^0(\cc)_\infty$.

A triangulated persistence category (TPC) $\cc$ is a persistence category such that $\cc_0$ is a triangulated category in the usual sense, along with some extra criteria discussed in \cite{BCZ} and that we will recall in more detail shortly. If $\cc$ is a TPC then $\cc_\infty$ is naturally triangulated and furthermore to each triangle we can associate a `weight', that is, a map $w:\Delta(\cc_\infty) \to \R_+$ where $\Delta(\cc_\infty)$ on the class of exact triangles. These weights allow one to define a family of (pseudo) metrics, named fragmentation pseudo-metrics on the objects of $\cc$. TPCs were introduced to formalise metrics/weights that appear when considering a filtered version of the Fukaya category of a symplectic manifold.

In section \ref{secttri} we discuss the question of when the limit category of a TPC has and idempotent completion which admits a TPC refinement, i.e., a persistence refinement that is also a TPC. The reason for one to consider this is as follows, the fragmentation metrics associated to a TPC depend on a choice of family of objects of the category. One can think of the distance between two objects to be the best (i.e., most efficient) way of building one object using the other via attachings of extra objects belonging to this family. Often the metrics are degenerate and/or give infinite distances. Ideally the family taken should generate the category (as a triangulated category) as to have a non-infinite distance between objects. A collection of objects generating a triangulated category often needs to be quite large, however, if one allows for `split-generation', the number of objects needed can be much smaller. By split-generation, we mean that every object of the category can be reached by taking mapping cones, up to the splitting of an idempotent, i.e., the object is a retract of an object that can be reached by usual triangulated generation. By considering split-generation, we can work with the fragmentation metrics induced by the TPC structure more tangibly. Similarly to the results of Theorem \ref{mainth3} we show 
\begin{theorem}\label{mainth4}
     Let $\cc$ be a triangulated persistence category, if every idempotent of $\cc_\infty$ is represented by an idempotent of weight zero in $\cc_0$ then $\text{Split}_\pp^0(\cc)$ is a TPC refinement of $\text{Split}(\cc_\infty)$ and is equivalent as a TPC to $\widehat{\text{Split}}_\pp(\cc)$ .
\end{theorem}

In the case of not every idempotent of the limit category admitting a weight zero idempotent representation, then the two candidate persistence refinements of $\text{Split}(\cc_\infty)$ are not TPCs. In particular, both do not have zero level category triangulated.

\subsection{Review of Persistence}
We begin by recalling a few definitions and ideas. We use \cite{PRSZ} as the main reference for the basic definitions of persistence modules. For persistence categories the material follows from \cite{BCZ} A \textbf{persistence module} is an $\R$-indexed collection of $k$-modules, $\{V_r\in \text{Mod}_k\}_{r\in \R}$ along with a family of morphisms $i^V_{r,s}:V_r \to V_s$, which we call the \textbf{persistence structure maps}, for all $s\geq r$ and such that $i^V_{r,r}=1_{V_r}$, $i^V_{r,s}\circ i^V_{t,r}=i^V_{t,s}$. Here $k$ is a commutative ring, usually assumed to be a field. This is the most basic definition of a persistence module, often further conditions are imposed and so we work with a very general form of persistence module. 

We can define persistence modules in the language of category theory, which we will largely be working with. One does so as follows: let $(\R,<)$ be the category whose objects are real numbers and morphisms are given by
\begin{equation}\text{Hom}_{(\R,<)}(r,s)=\begin{cases}
    i_{r,s} & r \leq s\\
    \emptyset & s<r.
\end{cases}\end{equation}
A persistence module can then be viewed as a functor, $V:(\R,<) \to \text{Mod}_k$, via: 
\begin{equation}V_r:=V(r), \hspace{0.3cm} i^V_{r,s}:=V(i_{r,s})\end{equation}
We will view persistence modules through either description interchangeably. A \textbf{persistence category} will be a category $\cc$ such that the morphism sets form persistence modules, and composition respects this structure in the additive sense. Explicitly, if $A,B\in \cc$ then
\begin{equation}\text{Hom}_{\cc}(A,B):(\R,<) \to \text{Mod}_k\end{equation}
and composition satisfies 
\begin{equation}\circ: \text{Hom}_\cc( A,B)(r) \otimes \text{Hom}_\cc(B,C)(s) \to \text{Hom}_\cc(A,C)(r+s)\end{equation}
and the following diagram commutes;
\begin{equation}\begin{tikzcd}
    \text{Hom}_\cc(A,B)(r)\otimes \text{Hom}_\cc(B,C)(s) \ar[r,"\circ"]& \text{Hom}_\cc(A,C)(r+s)\\
     \text{Hom}_\cc(A,B)(r')\otimes \text{Hom}_\cc(B,C)(s')\ar[u,"i^{AB}_{r',r}\otimes i^{BC}_{s',s}"] \ar[r,"\circ"]& \text{Hom}_\cc(A,C)(r'+s')\ar[u,"i^{AC}_{r'+s',r+s}"]
\end{tikzcd}\end{equation}

We will often denote by $\text{Mor}^r_\cc(A,B):=\text{Hom}_\cc(A,B)(r)$, to match the notation used in \cite{BCZ}. Finally, we will assume that all persistence categories carry a (family of) shift functor(s), $\{\ss^a:\cc \to \cc \}_{a\in \R}$ that satisfy the following conditions:
\begin{enumerate}
    \item $\ss^0=1_\cc.$
    \item  $\ss^a \circ \ss^b = \ss^b \circ \ss^a = \ss^{a+b}$.
    \item $\ss^a(A) \cong \ss^b(A)$ via an isomorphism $(\eta_{a,b})_A\in \text{Mor}_\cc^{b-a}(A,B)$.
\end{enumerate}

\begin{rmk}
    We can view the shift functors as a functor $\ss:(\R,+) \to \text{End}(\cc)$ where $(\R,+)$ is the category of real numbers with a unique morphism between any two objects. 
\end{rmk}

\begin{eg}
    Consider the category of filtered chain complexes, $\ff Ch_k$. An object $(C,d)\in \ff Ch_k$ is a collection of chain complexes $(C^{\leq r},d^{\leq r})$ where $d^{\leq r}:C^{\leq r}\to C^{\leq r}$ is a differential and for every $r\leq s$ we have a canonical inclusion chain morphism 
    \begin{equation}i_{r,s}:(C^{\leq r}, d^{\leq r}) \to (C^{\leq s},d^{\leq s})\end{equation}
    which satisfies
    \begin{equation}i_{r,s}\circ d^{\leq r} = d^{\leq s}\circ i_{r,s}.\end{equation}
    A morphism of filtered chain complexes $f\in \text{Mor}^s((C,d_C), (D,d_D))$ is a family of chain maps $f^{\leq r}: (C^{\leq r},d_C^{\leq r}) \to (D^{\leq r+s},d_D^{\leq r+s})$ satisfying 

    \begin{equation}i_{r,s}^D \circ f^{\leq r} = f^{\leq s}\circ i_{r,s}^C.\end{equation}
The inclusion maps give maps
    \begin{equation}\iota_{r,s}:\text{Mor}^r((C,d_C),(D,d_D)) \to \text{Mor}^s((C,d_C),(D,d_D))\end{equation}
    \begin{equation}\iota_{r,s}(f)^{\leq t}=i_{t+r, t+s} \circ f^{\leq t} .\end{equation}
We have functors $\ss^a: \ff Ch_k \to \ff Ch_k$, where $(\ss^aC ,\ss^ad):=\ss^a(C,d)$ is given by  
    \begin{equation}(\ss^{a}C)^{\leq r}= C^{\leq r- a}\end{equation}
    \begin{equation}(\ss^a d)^{\leq r}= d^{\leq r -a}\end{equation}
Passing to homology we have a persistence category $H\ff Ch_k$ with structure maps induced by the inclusions and shift functors induced by $\ss^a$.
    
\end{eg}

Given two persistence categories $\cc$ and $\cc'$, one can define a persistence functor $F:\cc \to \cc'$ to be a functor in the usual sense, but which respects the persistence structures. Explicitly, if $f\in \text{Mor}^r_\cc(A,B)$ then $F(f)\in \text{Mor}^r(F(A),F(B))$ and moreover we want $i_{r,s}^{\cc'}\circ F = F \circ i_{r,s}^\cc$. We can view this as commutativity of the following
\begin{equation}
    \begin{tikzcd}
        \text{Mor}^s_\cc(A,B)\ar[r,"F_{AB}^s"] & \text{Mor}^s_{\cc'}(F(A),F(B))\\
        \text{Mor}^r_\cc(A,B)\ar[u,"i^\cc_{r,s}"]\ar[r,"F_{AB}^r"] & \text{Mor}_{\cc'}^r(F(A),F(B))\ar[u,"i^{\cc'}_{r,s}"]
    \end{tikzcd}
\end{equation}
where $F_{AB}^{(-)}$ is the module map induced by the functor $F$. We can also define the opposite category of a persistence category. If $\cc$ is a persistence category, then $\cc^{op}$ is the category with the same objects, but with morphisms running the wrong way. We set

\begin{equation}
    \text{Mor}^{r}_{\cc^{op}}(B,A)= \text{Mor}^{r}_{\cc}(A,B)
\end{equation}

so that $i_{r,s}:\text{Mor}^{r}_{\cc^{op}}(B,A) \to \text{Mor}^{s}_{\cc^{op}}(B,A)$ is induced by $i_{r,s}:\text{Mor}^{r}_{\cc}(A,B) \to \text{Mor}^{s}_{\cc}(A,B)$. A contravariant functor $F:\cc \to \cc'$ is a functor $F:\cc^{op} \to \cc'$. I.e., if $f\in \text{Mor}^r_\cc(A,B)$ and $F$ is contravariant then $F(f)\in \text{Mor}^r_{\cc'}(F(B),F(A))$, and 
\begin{equation}
    \begin{tikzcd}
        \text{Mor}^s_\cc(A,B)\ar[r,"F_{AB}^s"] & \text{Mor}^s_{\cc'}(F(B),F(A))\\
        \text{Mor}^r_\cc(A,B)\ar[u,"i^\cc_{r,s}"]\ar[r,"F_{AB}^r"] & \text{Mor}_{\cc'}^r(F(B),F(A))\ar[u,"i^{\cc'}_{r,s}"]
    \end{tikzcd}
\end{equation}

Given a persistence category $\cc$, there are two natural $\text{Mod}_k$-enriched categories we denote $\cc_0$ and $\cc_\infty$. $\cc_0$ consists of all objects of $\cc$ but we restrict morphisms to $\text{Hom}_{\cc_0}(A,B)=\text{Mor}^0_\cc(A,B)$. $\cc_\infty$ again consists of all objects, but with morphisms now given by $\text{Hom}_{\cc_\infty}(A,B)= \text{Hom}_\cc(A,B)/\sim$ where if $f\in \text{Mor}^r_\cc(A,B)$ and $g \in \text{Mor}^{r'}_\cc(A,B)$ then $f \sim g$ if there exists $s,s'>0$ with $i_{r,r+s}(f) = i_{r',r'+s'}(g)$. This category can be thought of as the category where we forget the weights of morphisms. Denote by $\eta_r^A\in \text{Mor}^0(A,\ss^{-r}A)$ the morphism $i_{-r,0}\circ (\eta_{0,-r})_A$. Note that $\eta_0^A=1_A$. An object $A\in \cc$ will be called \textbf{$r$-acyclic} if $\eta_r^A=0$. The morphisms $\eta_r^A$ can be seen to satisfy the following properties;

\begin{enumerate}
    \item $\eta_r^{\ss^{-s}A} \circ \eta_s^A = \eta_{r+s}^A$.
    \item If $f:A \to B$ then $\eta_r^B \circ f = \ss^{-r}f \circ \eta_r^A$.
\end{enumerate}
We say two morphisms $f,f'\in \text{Mor}^a_\cc(A,B)$ are \text{$r$-equivalent} and write $f\simeq_r f'$ iff $\eta_r^B \circ f = \eta_r^B \circ f' \in \text{Mor}^a(A,\ss^{-r}B)$. Note that this implies $[f]=[f']\in \text{Hom}_{\cc_\infty}(A,B)$.
A \textbf{triangulated persistence category} (TPC) is a persistence category $\cc$ endowed with an endofunctor $T:\cc_0 \to \cc_0$ with the following properties:
\begin{enumerate}
    \item $(\cc_0,T)$ is a triangulated category in the usual sense, and commutes with the shift functors $\ss^r \circ T = T \circ \ss^r $ for all $r\in \R$.
    \item The morphisms $\eta_r^A\in \text{Hom}_{\cc_0}(A,\ss^{-r}A)$ embed into exact triangles 
    \begin{equation}\begin{tikzcd}
        A \ar[r,"\eta_r^A"]& \ss^{-r}A\ar[r] & C\ar[r] & TA
    \end{tikzcd}\end{equation}
    such that $C$ is $r$-acyclic. 
    \item We have that for all $r\leq 0$: 
    \begin{equation}\text{Mor}^r_\cc(A\oplus B, C) = \text{Mor}^r_\cc(A,C) \oplus \text{Mor}^r_\cc(B,C)\end{equation}
    \begin{equation}\text{Mor}^r_\cc(A, B \oplus C) = \text{Mor}^r_\cc(A,B) \oplus \text{Mor}^r_\cc(A,C)\end{equation}
\end{enumerate}

These axioms allow one to consider a larger class of triangles in $\cc_0$ that carry a `weighting'. A triangle in $\cc_0$ 

\begin{equation}\begin{tikzcd}
    A\ar[r,"u"] & B\ar[r,"v"] & C\ar[r,"w"] & \ss^{-r}TA
\end{tikzcd}\end{equation}
is called a \textbf{strict exact triangle} of weight $r$ if it embeds into a diagram 

\begin{equation}\begin{tikzcd}
   & &\ss^r C\ar[d,"\phi"]\ar[dr,"\ss^rw"]&  \\
    A\ar[r,"u"] & B\ar[dr,"v"]\ar[r,"v'"] & C'\ar[r,"w'"]\ar[d,"f"] & TA\\& & C\ar[r,"w"] & \ss^{-r}TA
\end{tikzcd}\end{equation}

where $f\circ \phi=\eta_r^{\ss^r C}$, and the cone of $f$ is $r$-acyclic. Notice that a strict exact weight zero triangle is an exact triangle in $\cc_0$. Biran, Cornea and Zhang showed that the limit category $\cc_\infty$ of a TPC $\cc$ is canonically triangulated and is in fact the Verdier localisation of $\cc_0$ by the subcategory of $r$-acyclic objects (for all $r\geq 0$). This along with the weighting attached to the larger class of triangles allows one to attach weights to exact triangles in $\cc_\infty$ and define pseudo-metrics on the objects of $\cc$. These weights and metrics are defined as follows: The \textbf{unstable weight} of an exact triangle in the limit category of a TPC,
\begin{equation}\Delta = A \xrightarrow{u} B \xrightarrow{v} C \xrightarrow{w} TA\end{equation} is defined to be
\begin{equation}w_\infty(\Delta) = \inf\{r : A \xrightarrow{u'} \ss^{-r_1}B \xrightarrow{ v'} \ss^{-r_2}C \xrightarrow{w'} \ss^{-r}TA \text{ is strict exact} \} \end{equation}
with $0\leq r_1 \leq r_2 \leq r$ and such that 

\begin{align*}
    u &= [\eta_{-r_1,0} \circ u']\\
    v&= [\eta_{-r_2,0} \circ v' \circ \eta_{0,-r_1}]\\
     w &= [\eta_{-r,0} \circ w' \circ \eta_{0,-r_2}].
\end{align*}
Here $[-]$ denotes the image of the morphism $(-)$ in the limit category $\cc_\infty$. The \textbf{stable weight} of $\Delta$ is then given by 

\begin{equation}w(\Delta) = \inf\{w_\infty(\ss^{s,0,0,s}\Delta)\}\end{equation}

It was shown that this weighting forms a \textbf{triangular weight}. This allows the construction of the pseudo-metrics called \textbf{fragmentation metrics} on the objects of $\cc$, which are defined to be the (pseudo) metric induced by symmetrising the following distance. We first choose a family $\ff\subset \text{\text{Obj}}(\cc)$ and then define the distance from $A$ to $B$ $\delta^\ff(A,B)$ to be

\begin{equation}
\delta^\ff(A,B)= \inf\{\sum_{i=1}^n w_\infty(\Delta_i)\}  
\end{equation}
where $\{\Delta_i\}_{i=1,\hdots,n}$ are exact triangles in $\cc_\infty$ that give an iterated cone decomposition of $A$ from $B$ using $\ff$. That is, we have a sequence of exact triangles
 \begin{equation}\begin{tikzcd}[ampersand replacement =\&]
    \Delta_1 : \& X_1\ar[r] \& 0 \ar[r] \& Y_1\ar[r]  \& TX_1\\
    \Delta_2 : \& X_2 \ar[r] \& Y_1 \ar[r] \& Y_2 \ar[r] \& TX_2\\
    \vdots \& \& \vdots \& \vdots\\
    \Delta_n: \& X_n\ar[r]  \& Y_{n-1} \ar[r] \& A \ar[r] \& TX_n                    
\end{tikzcd}\end{equation} 
where $X_i\in \ff$  $\forall i\neq j$ and $X_j=T^{-1}B$. The symmetrisation gives a pseudo metric

\begin{equation}
    d^\ff(A,B)=\max\{\delta^\ff(A,B),\delta^\ff(B,A)\}
\end{equation}
Clearly this metric can be degenerate, if $A$ and $B$ both belong to $\ff$ then $d^\ff(A,B)=0$. It can also be seen to be infinite often, for example if $\ff=\empty$ there is no hope of building $A$ from $B$ unless $A=B$. These issues are discussed in detail in \cite{BCZ}. The goal of the following sections is to enlarge the metrics $d^\ff$ to metrics on $\text{Split}(\cc_\infty)$.

\section{The persistence category of pairs}\label{sectpcatpair}
The Karoubi completion $\text{Split}(\cc)$ of a category $\cc$, can also be realised as the category of pairs $(A,e)$ where $A\in \C$ and $e:A \to A$ is an idempotent. A morphism $f:(A,e_A) \to (B,e_B)$ is a morphism $f\in \text{Hom}_\cc(A,B)$, with the added requirement that it intertwines the respective idempotents, i.e., $e_B \circ f = f = f \circ e_A$, or equivalently $f = e_B \circ f \circ e_A$. There is an obvious inclusion $\iota: \cc \to \text{Split}(\cc)$, given by $\iota(A)= (A,1_A)$ and $\iota (f) = f$. It is clear that this constructive definition of $\text{Split}(\cc)$ is equivalent to the definition via retracts of representable presheaves. The equivalence sends $(A,e_A)$ to (a choice of) the splitting of $\yy(e_A)$ in $\text{PSh}(\cc)$. The aim of this section is to construct a persistence `category' $\widehat{\text{Split}}_\pp(\cc)$ which is a persistence refinement of $\text{Split}(\cc_\infty)$ (viewed as the category of pairs). Objects of this category will be pairs $(A,e_A)$ where $e_A: A \to \ss^{-r}A$ is a `weighted idempotent'. Similarly to the usual setting, a morphism $f: (A,e_A) \to (B,e_B)$ should intertwine the weighted idempotents $e_B$ and $e_A$ in some way. It quickly becomes clear that there is an issue with this construction, as there is a lack of identity morphisms in the category. This lack of identity morphism may not be a huge complication, as we often need to work with `categories' without identities, often referred to as semi-categories (see \cite{Mit}). Thus before we construct our persistence `category' we first describe how we can generalise persistence categories in the sense of \cite{BCZ} to persistence categories without identities.
\subsection{Persistence semi-categories}
\begin{defn}
    A \textbf{persistence semi-category} $\cc$ consists of a collection of objects, $\text{Obj}(\cc)$, and for every $A,B\in \text{Obj}(\cc)$ there exists a persistence module $\text{Hom}_\cc(A,B): (\R,<) \to \text{Mod}_k$. For every $A,B,C\in \text{Obj}(\cc)$ there exists a family of $k$-module morphisms
    \begin{equation}\circ^{A,B,C}_{r,s}: \text{Hom}_\cc(A,B)(r) \oplus \text{Hom}_\cc(B,C)(s) \to \text{Hom}_\cc(A,C)(r+s) \end{equation}
    such that the following commutes: 

    \begin{equation}\begin{tikzcd}
        \text{Hom}_\cc(A,B)(r) \oplus \text{Hom}_\cc(B,C)(s) \ar[r,"\circ^{A,B,C}_{r,s}"]& \text{Hom}_\cc(A,C)(r+s)\\
        \text{Hom}_\cc(A,B)(r') \oplus \text{Hom}_\cc(B,C)(s') \ar[u,"i_{r',r}^{A,B}\oplus i_{s',s}^{B,C}"]\ar[r,"\circ^{A,B,C}_{r',s'}"]& \text{Hom}_\cc(A,C)(r'+s')\ar[u,"i_{r'+s',r+s}^{A,C}"]
    \end{tikzcd}\end{equation}
    where $i^{j,k}_{t,t'}= \text{Hom}_\cc(j,k)(\iota_{t,t'})$ are the corresponding persistence module structure morphims. We also require the following:
    
    \begin{itemize}
        \item Associativity of compositions, i.e., $\circ^{A,C,D}_{r+s,t} \circ_{\text{Mod}_k} (\circ_{r,s}^{A,B,C} \oplus 1))= \circ_{r,s+t}^{A,B,D} \circ_{\text{Mod}_k} (1\oplus \circ^{B,C,D}_{s,t})$ and this commutes with the persistence module structure morphisms.

        \item For every object $A\in \cc$, there exists an \textbf{$r$-identity} for some large enough $r\geq 0\in \R$. This is a map $\zeta_r^A\in \text{Hom}_\cc(A,A)(r)$ such that for all $f \in \text{Hom}_\cc(A,B)(s)$ 
        \begin{equation}f \circ \zeta_r^A = i_{s,s+r}(f) \end{equation}
        similarly for any $g \in \text{Hom}_\cc(C,A)(s)$
        \begin{equation}\zeta_r^A \circ g = i_{s,s+r}(g)\end{equation} 
        
        \end{itemize}
\end{defn}
\begin{rmk}
    If there exists an $r$-identity on $A$, then there naturally exists an $s$-identity for any $s\geq r$, it is given by $\zeta_s^A = i_{r,s}\circ \zeta_r^A$. With this we define a map $\zeta_\cc:\text{Obj}(\cc) \to \R_{\geq 0}$,
    
    \begin{equation}\zeta_\cc:A \mapsto \inf \{r: \exists \text{ $r$-identity }\zeta_r^A:  A \to A \} \end{equation} 
    We will often denote $\lfloor A \rfloor := \zeta_\cc(A)$.
\end{rmk}

We do not require objects to have an identity morphism. The $r$-identities are a generalisation that act as an `identity with weight $r$'.

\begin{rmk}
    $A$ has an identity if and only if $\zeta(A)= 0$.
\end{rmk}

\begin{defn}
    The \textbf{flattened category} $\cc^\ss$ associated to $\cc$, is the semi-category (without identities in general) whose objects are given by formal objects of the form $\ss^a A$ where $A\in \cc$ and $a\in \R$. The morphisms are defined by 
    \begin{equation}\text{Hom}_{\cc^\ss}(\ss^a A, \ss^{b}B):=\text{Hom}_\cc(A,B)(a-b).\end{equation}
    Composition is defined using composition in $\cc$, i.e., if $f\in \text{Hom}_{\cc^\ss}(\ss^a A, \ss^b B)$ and $g\in \text{Hom}_{\cc^\ss}(\ss^b B, \ss^c C)$ then $g \circ f\in  \text{Hom}_\cc(A,B)((a-b) + (b-c))=\text{Hom}_\cc(A,B)(a-c)=\text{Hom}_{\cc^\ss}(\ss^a A, \ss^c C)$.
\end{defn}
\begin{rmk}
    There is a canonical functor $\ss^r:\cc^\ss \to \cc^\ss$ for every $r\in \R$, given by $\ss^r (\ss^a A)= \ss^{a+r}A$ on objects and by the `identity' on morphisms. By this we mean $\ss^r: \text{Hom}_{\cc^\ss}(\ss^aA , \ss^b B) \to \text{Hom}_{\cc^\ss}(\ss^{a+r}A , \ss^{b+r} B) $ exactly given by $\text{Id}:\text{Hom}_\cc(A,B)(a-b) \to \text{Hom}_\cc (A,B)((a+r)-(b+r))=\text{Hom}_\cc(A,B)(a-b)$.
\end{rmk}

Assume $A\in \cc$ is such that $\lfloor A\rfloor=\hat{a}$. Then for all $r\geq \hat{a}$ there is a corresponding map $\zeta^A_r\in \text{Hom}_{\cc^\ss}(A,\ss^{-r}A)$. There are also maps $\zeta_r^{\ss^{b}A}\in \text{Hom}(\ss^{b}A, \ss^{b-r}A)$, given by $\ss^b\zeta_r^A$ for any $b\in \R$. Assuming that $\zeta_r^A$ and $\zeta_r^B$ exist, then we find the following commutes for any $f\in \text{Hom}_{\cc}(A,B)$
\begin{equation}\begin{tikzcd}
    A \ar[r,"f"]\ar[d,"\zeta_r^A"]& B\ar[d,"\zeta_r^B"]\\
    \ss^{-r} A \ar[r,"\ss^{-r} f"]& \ss^{-r} B
\end{tikzcd}\end{equation}
This follows directly from the commutativity of morphisms and persistence structure morphisms in $\cc$. Finally, we note that composition of $\zeta$'s is additive, i.e., $\zeta_{s}^{\ss^{-r}A} \circ \zeta_r^A = \zeta_{r+s}^A$ (with the obvious requirement that $s,r\geq \zeta_{\cc}(A)$).

\subsection{Calculus of fractions in semi-categories}

We would like to describe the limit category $\cc_\infty$ as a localisation $[W^{-1}]\cc_0$ with respect to a class of morphisms. To do so we need to ensure that this makes sense in a category without identities. In order to do so we need to add an extra criterion for the class $W$ to satisfy. The axioms for $W$ to be a (left) calculus of fractions in a semi-category are as follows:
\begin{enumerate}
    \item $W$ contains any identity morphisms.
    \item $W$ is closed with respect to composition.
    \item Assume $f\in W$ and $x$ is any morphism, then the pull back square $\begin{tikzcd}
                  D\ar[r,"f'",dotted]\ar[d,dotted,"x'"]&C\ar[d,"x"]\\
                 A \ar[r,"f"]& B
            \end{tikzcd}$ exists with $f'\in W$.
            \item If there exists a commutative  diagram $\begin{tikzcd}
            A\ar[r,shift left = 0.75ex, "f"]\ar[r,shift right = 0.75ex, swap,"g"] &  B \ar[r,"t"]&  D
        \end{tikzcd}$ with $t\in W$, then there exists a second commutative diagram $\begin{tikzcd}
             C \ar[r,"s"]&  A\ar[r,shift left = 0.75ex, "f"]\ar[r,shift right = 0.75ex, swap,"g"] &  B 
        \end{tikzcd}$ with $s\in W$. 
        \item For any object $A$, there exists an $A'$ and $w:A' \to A$ in $W$. 
\end{enumerate}

\begin{rmk}
    The only real difference is axiom 5, we will require this for reflectivity of the equivalence relation defined on morphisms. Axiom 1 is a slight restatement of the original axiom 1 for CLFs where we require any identities that exist to be contained in $W$.
\end{rmk}

 The localisation $[W^{-1}]\cc$, will have the same objects as $\cc$, and $\text{Hom}_{[W^{-1}]\cc}(A,B)$ will consist of roof diagrams
\begin{equation}\begin{tikzcd}A &\ar[l,swap,"w"] X \ar[r,"f"] & B\end{tikzcd} \end{equation}
with $w\in W$ and $f\in \text{Hom}_\cc(X,B)$ modulo an equivalence $\sim$. We say $
    \begin{tikzcd}
   A & \ar[l,swap,"w"]\ar[r,"f"]X & B
\end{tikzcd}
$ and $\begin{tikzcd}
   A & \ar[l,swap,"w'"]\ar[r,"f'"]X' & B
\end{tikzcd}$ are equivalent iff there exists a a third roof diagram $\begin{tikzcd}A &\ar[l,swap,"w''"] X''\ar[r,"f''"] & B\end{tikzcd}$, morphisms $\phi$ and $\psi$ and a commutative diagram in $\cc$

\begin{equation}\begin{tikzcd}
    & \ar[dl,swap,"w"]X\ar[dr,"f"] &\\
    A &\ar[l,swap,"w''"]\ar[u,"\phi"] X''\ar[r,"f''"]\ar[d,"\psi"] & B\\
    & \ar[ul,"w'"]X'\ar[ur,swap,"f'"] &
\end{tikzcd}\end{equation}

\begin{prop}
If $W$ is a (left) calculus of fractions in a semi-category $W$ then $\sim$ is an equivalence relation.
\begin{proof}
Note that the proof is essentially the same as in the usual case:

\begin{itemize}
    \item  $\sim $ is reflexive.
    \begin{proof}
    By axiom 5 and axiom 2 we can construct a diagram of the following form
\begin{equation}\begin{tikzcd}
    & \ar[dl,swap,"w"]X\ar[dr,"f"] &\\
    A &\ar[l,swap,"w \circ \phi"]\ar[u,"\phi"] X'\ar[r,"f\circ \phi"]\ar[d,"\phi"] & B\\
    & \ar[ul,"w"]X\ar[ur,swap,"f"] &
\end{tikzcd}\end{equation}
    \end{proof}
    \item $\sim$ is symmetric.
    \begin{proof}
        This is clear by symmetry of the relation diagram.
    \end{proof}
    \item $\sim$ is transitive.
    \begin{proof}
        Assume that we have  $\begin{tikzcd}
           ( A &\ar[l,swap,"u"] X \ar[r,"f"]& B)  \sim  (A & \ar[l,swap,"v"]Y\ar[r,"g"] & B)
        \end{tikzcd}$  and 
        
        $\begin{tikzcd}
             (A & \ar[l,swap,"v"]Y\ar[r,"g"] & B) \sim (A &\ar[l,swap,"w"] Z\ar[r,"h"] & B)
        \end{tikzcd}$.
       In particular there exist diagrams
        
\begin{equation}\begin{tikzcd}
    & \ar[dl,swap,"u"]X\ar[dr,"f"] & &  &\ar[dl,swap,"v"]Y\ar[dr,"g"] &\\
    A &\ar[l,swap,"x"]\ar[u,"\alpha"] P\ar[r,"k"]\ar[d,"\beta"] & B &  A &\ar[l,swap,"y"]\ar[u,"\gamma"] Q\ar[r,"l"]\ar[d,"\delta"] & B\\
    & \ar[ul,"v"]Y\ar[ur,swap,"g"] & & &\ar[ul,"t"]Z\ar[ur,swap,"h"] &
\end{tikzcd}\end{equation}

We `glue' these together along $Y$ to obtain
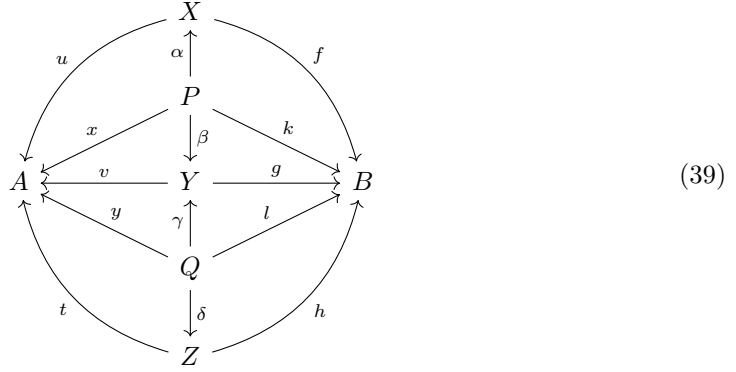
\begin{equation}\begin{tikzcd}
   & &\ar[ddll,bend right = 30,swap,"u"] X\ar[ddrr,bend left = 30, "f"] & &\\
   & & P\ar[u,"\alpha"]\ar[d,"\beta"]\ar[drr,"k"]\ar[dll,swap,"x"] & &\\
   A & & Y \ar[ll,swap,"v"]\ar[rr,"g"]& & B\\
   & & Q\ar[d,"\delta"]\ar[u,"\gamma"]\ar[urr,"l"]\ar[ull,swap,"y"] & & \\
   & & Z \ar[uull,bend left = 30,"t"]\ar[rruu,bend right= 30, swap, "h"]& &
\end{tikzcd}\end{equation}

Using axiom 3 we can pull back $\begin{tikzcd}
     & P\ar[d,"v \circ \beta"]\\
     Q \ar[r,"v \circ \gamma"]& A
\end{tikzcd}$ to $\begin{tikzcd}
   R \ar[r,"\rho"]\ar[d,"\sigma"]  & P\ar[d,"v \circ \beta"]\\
     Q \ar[r,"v \circ \gamma"]& A
\end{tikzcd}$. It follows that $v \circ \beta \circ \rho = v \circ \beta \circ \sigma $, i.e., we have a diagram $\begin{tikzcd}
            R \ar[r,shift left = 0.75ex, "\beta \circ \rho"]\ar[r,shift right = 0.75ex, swap,"\gamma \circ \sigma"] &  Y \ar[r,"v"]&  A
        \end{tikzcd}$. Using axiom 4 we obtain a diagram $\begin{tikzcd}
            T\ar[r,"w"] & R \ar[r,shift left = 0.75ex, "\beta \circ \rho"]\ar[r,shift right = 0.75ex, swap,"\gamma \circ \sigma"] &  Y \ar[r,"v"]&  A
        \end{tikzcd}$. Finally, putting all of this together we obtain a commutative diagram 

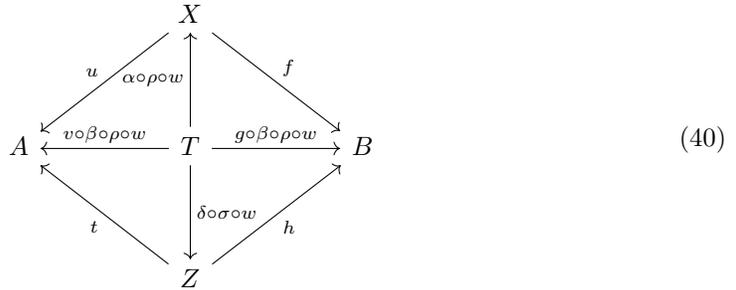
\begin{equation}\begin{tikzcd}
    & & X\ar[ddll,swap,"u"] \ar[ddrr,"f"]& &\\
    & & & & \\
    A& &\ar[uu,"\alpha \circ \rho \circ w"] \ar[ll,swap,"v  \circ \beta \circ \rho \circ w"] \ar[dd,"\delta \circ \sigma \circ w"]T\ar[rr,"g\circ \beta \circ \rho \circ w"] & & B\\
    & & & & \\
    & & \ar[uull,"t"]Z\ar[uurr,swap,"h"] & & 
\end{tikzcd}\end{equation}
Checking this commutes: 
\begin{itemize}
    \item[Top Left:] \begin{align*}
    u \circ \alpha \circ \rho \circ w = & x \circ \rho \circ w\\
    =&  v \circ \beta \circ \rho \circ w
\end{align*}
\item[Bottom Left:]
\begin{align*}
    t \circ \delta \circ \sigma \circ w = & y \circ \sigma \circ w\\
    =& v \circ \gamma \circ \sigma \circ w\\
    =& v \circ \beta \circ \rho \circ w
\end{align*}
\item[Top Right:]
\begin{align*}
    f \circ \alpha \circ \rho \circ w = & k \circ \rho \circ w \\
    =& g \circ \beta \circ \rho \circ w
\end{align*}
\item[Bottom Right:]
\begin{align*}
    h \circ \delta \circ \sigma \circ w =& l \circ \sigma \circ w \\
    =& g \circ \gamma \circ \sigma \circ w\\
    =& g \circ \beta \circ \rho \circ w
\end{align*}
\end{itemize}
Finally we notice that $v \circ \beta \circ \rho \circ w \in W$. This follows from the following; $v \circ \beta = x \in W$ by definition, by axiom 3 we have $\rho\in W$ and by axiom 4 we have $w\in W$, then by axiom 2 the composition is in $W$.
    \end{proof}
\end{itemize}

\end{proof}
\end{prop}

Composition of two morphisms $\begin{tikzcd}
           ( A &\ar[l,swap,"u"] X \ar[r,"f"]& B)\end{tikzcd}$   and $\begin{tikzcd}  (B & \ar[l,swap,"v"]Y\ar[r,"g"] & C)
        \end{tikzcd}$

is defined via axiom 3, we pull back $\begin{tikzcd}
    & Y\ar[d,"v"]\\
    X \ar[r,"f"]& B
\end{tikzcd}$ to obtain a diagram
\begin{equation}\begin{tikzcd}
    & & Z \ar[dl,swap,"v'"]\ar[dll,swap,bend right=30,"u \circ v'" ]\ar[drr,bend left = 30, "g \circ f'"]\ar[dr,"f'"]& &\\
     A &\ar[l,swap,"u"] X \ar[r,"f"]& B& \ar[l,swap,"v"]Y\ar[r,"g"] & C
\end{tikzcd}\end{equation}
and set  $\begin{tikzcd}
             (B & \ar[l,swap,"v"]Y\ar[r,"g"] & C)\circ ( A &\ar[l,swap,"u"] X \ar[r,"f"]& B) := (A & \ar[l,swap,"u \circ v'"]Z\ar[r,"g \circ f'"] & C)
        \end{tikzcd}$.

\begin{prop}
Composition in $[W^{-1}]\cc$is well defined. 
\begin{proof}
Assume $\begin{tikzcd}
   Z'\ar[r,"f''"]\ar[d,"v''"] & Y\ar[d,"v"]\\
    X \ar[r,"f"]& B
\end{tikzcd}$ is another pull back diagram, such that composition gives $\begin{tikzcd}
    (A & Z'\ar[l,swap,"u \circ v''"] \ar[r,"g \circ f''"]& C)
\end{tikzcd}$.

Pull back the diagram $\begin{tikzcd}
    & Z\ar[d," v'"]\\
    Z'\ar[r," v''"] & X
\end{tikzcd}$ to $\begin{tikzcd}
   Z''\ar[r,"x"]\ar[d,"y"] & Z\ar[d,"v'"]\\
    Z'\ar[r," v''"] & X
\end{tikzcd}$. Then $f \circ (v' \circ x) = f \circ (v'' \circ y)$ and hence $v \circ (f' \circ x)= v \circ (f'' \circ y)$. Since $v\in W$ we can construct $\begin{tikzcd}
    K \ar[r,"w"]& Z''\ar[r,shift left = 0.75ex,"f' \circ x "] \ar[r,swap,shift right = 0.75ex,"f'' \circ y"]& X
\end{tikzcd}$ and finally one can check the following diagram commutes and defines the equivalence

\begin{equation}\begin{tikzcd}
   & & Z\ar[ddll,swap,"u \circ v'"]\ar[ddrr,"g \circ f'"] & & \\
    \\
    A & & K \ar[uu,"x \circ w"]\ar[dd,"y \circ w"]\ar[rr,"g \circ f' \circ x\circ w"]\ar[ll,"u \circ v' \circ x \circ w"]& & C\\
    \\
    & & Z'\ar[uull,"u \circ v''"] \ar[uurr,swap,"g \circ f''"]&&
\end{tikzcd}\end{equation}

Finally, we check that if we have  $\begin{tikzcd}
           ( A &\ar[l,swap,"u"] X \ar[r,"f"]& B)  \sim  (A & \ar[l,swap,"v"]X'\ar[r,"g"] & B)
        \end{tikzcd}$  and $\begin{tikzcd}
            B & \ar[l,swap,"w"]Y\ar[r,"h"] & C
        \end{tikzcd}$ then 
        
        $\begin{tikzcd}
           ( A &\ar[l,swap,"u \circ w'"] Z\ar[r," h\circ f' "] & C )\sim ( A &\ar[l,swap,"v \circ t'"] Z'\ar[r,"h \circ g'"] & C )
        \end{tikzcd}$.

\begin{equation}\begin{tikzcd}
     & &  & && & & \\
     \\
     &&\ar[ddll,swap,"u"] X \ar[ddrr,"f"]& &  \ar[ll,swap,color= red,"w'"]Z\ar[ddrr,color= red,"f'"]\ar[ddllll,bend right= 60,swap,"u \circ w'"]\ar[ddrrrr,bend left = 60,"h \circ f'"]& & & &\\
     \\
     A & &\ar[ll,swap,"u \circ \phi"]X''\ar[uu,"\phi"]\ar[dd,"\psi"]\ar[rr,"f \circ \phi"]& & B &&\ar[ll,swap,"w"]Y\ar[rr,"h "] && C\\
     \\
     & & \ar[uull,"v"]X'\ar[uurr,swap,"g"] & &  \ar[ll,color=red,"t'"]Z'\ar[uurr,color = red,"g'"]\ar[uullll,bend left = 60,"v \circ t'"] \ar[uurrrr,bend right = 60,"h \circ g'"] & && \\
     \\
    & & & &  & &  & &
\end{tikzcd}\end{equation}

Now we construct the two pull back squares $\begin{tikzcd}
    W \ar[r,"\alpha"]\ar[d,"\beta"]& X''\ar[d,"\phi"]\\
    Z \ar[r,"w'"]& X
\end{tikzcd}$ and $\begin{tikzcd}
    W' \ar[r,"\gamma"]\ar[d,"\delta"]& X''\ar[d,"\psi"]\\
    Z' \ar[r,"t'"]& X'
\end{tikzcd}$ 
and then further pull back to $\begin{tikzcd}
    W''\ar[r,"\alpha'"]\ar[d,"\gamma'"] & W\ar[d,"\alpha"] \\
    W' \ar[r,"\gamma"]& X''
\end{tikzcd}$. So far we have constructed the following

\begin{equation}\begin{tikzcd}
    X & & \ar[ll,"w'"]Z & & & & \\
    & & W \ar[u,"\beta"]\ar[dll,"\alpha"]& & & &\\
    X''\ar[uu,"\phi"]\ar[dd,"\psi"] & & & & W''\ar[ull,"\alpha'"]\ar[dll,"\gamma'"] \\
    & & W' \ar[d,"\delta"]\ar[ull,"\gamma"]& & & & \\
    X' & & \ar[ll,"t'"]Z' & & & 
\end{tikzcd}\end{equation}
There are two canonical maps into $Y$, and they are coequalised by $w:Y \to B$,

\begin{equation}\begin{tikzcd}
    X & & \ar[ll,"w'"]Z\ar[rrrrdd,bend left = 30,"f'"] & & & & \\
    & & W \ar[u,"\beta"]\ar[dll,"\alpha"]& & & &\\
    X''\ar[uu,"\phi"]\ar[dd,"\psi"] & & & & W''\ar[ull,"\alpha'"]\ar[dll,"\gamma'"] \ar[rr,yshift = 0.75ex,"f' \circ \beta \circ \alpha'"]\ar[rr,yshift= -0.75ex,swap,"g'\circ \delta \circ \gamma'"]& & Y\ar[rr,"w"] & & B\\
    & & W' \ar[d,"\delta"]\ar[ull,"\gamma"]& & & & \\
    X' & & \ar[ll,"t'"]Z' \ar[rrrruu,bend right = 30,"g'"]& & & 
\end{tikzcd}\end{equation}
Indeed we have:
\begin{align*}
    (w \circ f') \circ \beta \circ \alpha' =& f \circ (w' \circ \beta) \circ \alpha'\\
    =& (f \circ \phi) \circ \alpha \circ \alpha' \\
    =& g \circ \psi \circ (\alpha \circ \alpha')\\
    =& g \circ (\psi \circ \gamma) \circ \gamma'\\
    =& (g \circ t') \circ \delta \circ \gamma'\\
    =& w \circ g' \circ \delta \circ \gamma'
\end{align*}
By axiom 4 we construct $\begin{tikzcd}
    K\ar[r,"s"] & W'' \ar[rr,yshift = 0.75ex,"f' \circ \beta \circ \alpha'"]\ar[rr,yshift= -0.75ex,swap,"g'\circ \delta \circ \gamma'"]& &Y \ar[r,"w"]& B
\end{tikzcd}$. Finally one finds the following commutes 

\begin{equation}\begin{tikzcd}
    && Z \ar[ddll,swap,"u \circ w'"]\ar[ddrr,"h \circ f'"]&& \\
    \\
   A  && K\ar[uu," \beta \circ \alpha' \circ s"] \ar[dd,"\delta \circ \gamma' \circ s"]\ar[ll,"u \circ \phi \circ \alpha \circ \alpha' \circ s"]\ar[rr,"h \circ f' \circ \beta \circ \alpha' \circ s"]&& C \\
   \\
    && \ar[uull,"v \circ t'"]Z'\ar[rruu,swap,"h \circ g'"] && 
\end{tikzcd}\end{equation}

Composition by two equivalent roof diagrams on the right follows by a similar argument. 

\end{proof}
\end{prop}

\begin{rmk}
In general, the localisation category will also not have identities.
\end{rmk}

    Returning to $\cc^\ss$, the `flattened' category associated to the persistence semi-category $\cc$. We define a class of morphisms $W$ consisting of weighted isomorphisms.

    \begin{defn}
    We say a morphism $f\in \text{Hom}_{\cc^\ss}(\ss^aA,\ss^bB)$ is a \textbf{weighted isomorphism} of weight $r$, if there exists a $g\in \text{Hom}_{\cc^\ss}(\ss^b B, \ss^{a-r}A)$ such that $g \circ f = \zeta_r^A$ and $f \circ \ss^r g= \zeta_r^{\ss^r B}$. 
      \end{defn}
    \begin{prop}
        The collection of weighted isomorphisms in $\cc^\ss$ forms a calculus of fractions.
    \begin{proof}
    We prove each requirement in turn:
    \begin{enumerate}
        \item Every identity belongs to $W$.
        \begin{proof}
            Recall that $A$ has identity iff $\zeta(A)=0$. In particular $1_A=\zeta_0^A$. Naturally $1_A$ is therefore a weighted isomorphism of weight $0$ since $\zeta_0^A \circ \zeta_0^A=\zeta_0^A$.
        \end{proof}
        \item Composition in $W$ is closed.
        \begin{proof}
            Assume $f\in \text{Hom}_{\cc^\ss}(\ss^a A,\ss^b B)$ is an $r$-isomorphism, and $u\in \text{Hom}_{\cc^\ss}(\ss^b B,\ss^c C)$ is an $s$-isomorphims. There therefore exists $g\in \text{Hom}_{\cc^\ss}(\ss^b B, \ss^{a-r}A)$ such that $g \circ f = \zeta_r^{\ss^aA}$ and $f \circ \ss^r g= \zeta_r^{\ss^{b+r} B}$. As well as a morphism $v\in \text{Hom}_{\cc^\ss}(\ss^{c}C,\ss^{b-s}B)$ satisfying $v \circ u = \zeta_{s}^{\ss^b B}$ and $u \circ \ss^s v = \zeta_s^{\ss^{b+s} C}$. Consider the composition $ (\ss^{-s}g \circ v )\circ (u \circ f)$. By definition, this is given by the composition in $\cc$, $g \circ v \circ u \circ f$. We have
            \begin{align*}
                g \circ v \circ u \circ f=& g \circ \zeta_s^A \circ f\\
                =&g \circ i_{(a-b),(a-b)+s}(f)\\
                =&i_{r,r+s}(g \circ f)\\
                =& i_{r,r+s}(\zeta_r^A)\\
                =&\zeta_{r+s}^A
            \end{align*}
            Hence $ (\ss^{-s}g \circ v )\circ (u \circ f)= \zeta_{r+s}^{\ss^{a}A}\in \text{Hom}_{\cc^\ss}(\ss^a A,\ss^{a-(r+s)}A)$. A similar argument shows that $(u \circ f) \circ \ss^{r+s}( \ss^{-s}g \circ v)= \zeta_{r+s}^{\ss^{c+(r+s)}C}$.
        \end{proof}
        \item Assume $f\in W$ and $x$ is any morphism, then the pull back square $\begin{tikzcd}
                 \ss^d D\ar[r,"f'",dotted]\ar[d,dotted,"x'"]& \ss^cC\ar[d,"x"]\\
                 \ss^aA \ar[r,"f"]& \ss^b B
            \end{tikzcd}$ exists with $f'\in W$.
        \begin{proof}
            Assume $f$ is an $r$-isomorphism with $g\in \text{Hom}_{\cc^\ss}(\ss^b B, \ss^{a-r}A)$ such that $g \circ f = \zeta_r^{\ss^aA}$ and $f \circ \ss^r g= \zeta_r^{\ss^{b+r} B}$. Set $s:=\zeta_\cc(C)$, then one can check the following commutes

            \begin{equation} \begin{tikzcd}\ss^{c+r+s}C\ar[d,swap," \ss^{r} g \circ \ss^{r}x \circ \zeta_{s} "]\ar[r,"\zeta_{r+s}"] & \ss^c C\ar[d,"x"]\\
                 \ss^aA \ar[r,"f"]& \ss^b B
            \end{tikzcd}\end{equation}
        \end{proof}
        \item If there exists a diagram $\begin{tikzcd}
            \ss^a A\ar[r,shift left = 0.75ex, "f"]\ar[r,shift right = 0.75ex, swap,"g"] & \ss^b B \ar[r,"t"]& \ss^d D
        \end{tikzcd}$ with $t\in W$ then there exists another diagram  $\begin{tikzcd}
            \ss^c C \ar[r,"s"]& \ss^a A\ar[r,shift left = 0.75ex, "f"]\ar[r,shift right = 0.75ex, swap,"g"] & \ss^b B 
        \end{tikzcd}$ with $s\in W$.
        \begin{proof}
            Assume $t$ is an $r$-isomorphism and $r\in \text{Hom}_{\cc^\ss}(\ss^d D,\ss^{b-r} B)$ is such that $r \circ t = \zeta_r$ and $t \circ \ss^r r= \zeta_r$.  Take $m= \max \{r, \zeta_\cc(A)\}$ then we have
            \begin{align*}
           & t \circ f = t \circ g\\
            \implies & r \circ t \circ f = r \circ t \circ g\\
               \implies &\zeta_r \circ f = \zeta_r \circ g\\
               \implies & \zeta_m \circ f = \zeta_m \circ g \\
               \implies &\ss^m f \circ \zeta_m = \ss^m g\circ \zeta_m\\
               \implies & f \circ \zeta_m = g \circ \zeta_m
            \end{align*}

Hence we have the following $\begin{tikzcd}
            \ss^{a+m}A \ar[r,"\zeta_m"]&\ss^a A\ar[r,shift left = 0.75ex, "f"]\ar[r,shift right = 0.75ex, swap,"g"] & \ss^bB 
        \end{tikzcd}$.
           
        \end{proof}
        \item For every $\ss^a A\in \cc^\ss$ there exists a weighted isomorphism $u:\ss^{a'}A' \to \ss^a A$. 
        \begin{proof}
            Simply take $\zeta_{a'}:\ss^{a+a'}A \to \ss^aA$ with $a'\geq \zeta_\cc(A)$.
        \end{proof}
    \end{enumerate}

\end{proof}
\end{prop}

Recall that given a persistence category $\cc$, one can consider the limit category $\cc_\infty$, the objects of this are the same as that of $\cc$ but hom-sets are defined as
\begin{equation*}
    \text{Hom}_{\cc_\infty}(A,B):= \lim_{r\to \infty} \text{Hom}_\cc(A,B)(r).
\end{equation*}
It is clear that this definition extends to persistence semi-categories. Moreover one can see that the limit category of a persistence semi-category is a category in the usual sense.
    \begin{lemma}
        The localisation $[W^{-1}]\cc^\ss$ of $\cc^\ss$ with respect to weighted isomorphisms is a category with identities and moreover, is equivalent to the limit category $\cc_\infty$.
        \begin{proof}
        The identity morphism on $A\in [W^{-1}]\cc$ is given by the equivalence class of the roof diagram $\begin{tikzcd}
            \ss^aA & \ar[l,swap,"\zeta_r"]\ss^{a+r}A \ar[r,"\zeta_r"]& \ss^a A
        \end{tikzcd}$ where $r\geq \zeta_{\cc}(A)$. To see this, consider the composition with $\begin{tikzcd}
            \ss^a A & \ar[l,swap, "w"]\ss^xX \ar[r,"f"]& \ss^b B
        \end{tikzcd}$ where $w$ is a $s$-isomorphism. We can take composition to be given by

        \begin{equation}\begin{tikzcd}
           & & \ss^{x+r+s} X\ar[ddll,bend right = 30,swap,"
w \circ \zeta_{r+s}"]\ar[ddrr,bend left = 30,"f \circ \zeta_{r+s}"]\ar[ddl,dotted, "w \circ \zeta_s"]\ar[ddr,dotted,swap,"\zeta_{r+s}"]& & \\
            \\
            \ss^aA & \ar[l,swap,"\zeta_r"]\ss^{a+r}A \ar[r,"\zeta_r"]& \ss^a A& \ar[l,swap, "w"]\ss^xX \ar[r,"f"]& \ss^b B
        \end{tikzcd}\end{equation}
        Then consider 
        \begin{equation}\begin{tikzcd}
        & & \ss^{x+r+s}X\ar[ddll,swap,"w \circ \zeta_{r+s}"]\ar[ddrr,"f \circ \zeta_{r+s}"] && \\
         \\
         \ss^a A & &\ss^{x+r+s+t}\ar[uu,"\zeta_{t}"]\ar[dd,"\zeta_{r+s+t}"] X\ar[rr,"f \circ \zeta_{r+s+t}"]\ar[ll,"w \circ \zeta_{r+s+t}"] & &\ss^b B\\
         \\
        & & \ss^{x}X\ar[uull,"w"]\ar[uurr,swap,"f"] & &
        \end{tikzcd}\end{equation}
        where $t\geq \zeta_\cc(X)$. A similar argument shows $\begin{tikzcd}
            \ss^aA & \ar[l,swap,"\zeta_r"]\ss^{a+r}A \ar[r,"\zeta_r"]& \ss^a A
        \end{tikzcd}$ is also the left inverse. Note in particular that we can take the identity on $\ss^aA$ to be represented by $\begin{tikzcd}
            \ss^aA & \ar[l,swap,"\zeta_{\zeta_\cc(A)}"]\ss^{a+\zeta_\cc(A)}A \ar[r,"\zeta_{\zeta_{\cc}(A)}"]& \ss^a A
        \end{tikzcd}$. We now show the equivalence. We construct a functor $\Xi:\cc_\infty \to W^{-1}\cc^\ss$, which on objects is given by $\Xi(A)=\ss^0 A= A$. On morphisms it is defined as follows; let $\bar{f}\in \text{Hom}_{\cc_\infty}(A,B)$, then assume $f\in \text{Hom}_\cc(A,B)(r)$ is a representative of $\bar{f}$, i.e., $[f]_\infty = \bar{f}$. $f$ defines a map $f\in \text{Hom}_{\cc^\ss}(\ss^r A,B)$. Set $\Xi(f):= (\begin{tikzcd}
            A &\ar[l,swap,"\zeta_{\hat{a}+r}"] \ss^{\hat{a}+r}A \ar[r," f \circ \zeta_{\hat{a}}"]& B
        \end{tikzcd})$
        where $\hat{a}:=\zeta_\cc(A)$. Firstly we check this is well defined. Assume $[f']=[f]\in \text{Hom}_{\cc_\infty}(A,B)$ with $f'\in \text{Hom}_\cc(A,B)(r')$. Then there exist $s,s'\in \R$ with $\zeta_{s}\circ f = \zeta_{s'}\circ f'$ (with $t:=s+r=s'+r'$ large enough). We have $f'\in \text{Hom}_{\cc^\ss}(\ss^{r'}A,B)$ and $\Xi(f')=(\begin{tikzcd}
            A & \ss^{\hat{a}+r'}A \ar[l,swap,"\zeta_{\hat{a}+r'}"]\ar[r,"f' \circ \zeta_{\hat{a}}"]& B
        \end{tikzcd})$.  The following diagram shows the equivalence $\Xi(f)\sim \Xi(f')$

        \begin{equation}\begin{tikzcd}
            & &\ss^{\hat{a}+r}A \ar[ddll,swap,"\zeta_{\hat{a}+r}"]\ar[ddrr,"f\circ \zeta_{\hat{a}}"]& &\\
            \\
           A & &\ss^{\hat{a}+t}A\ar[dd,"\zeta_{s'}"]\ar[uu,"\zeta_s"]\ar[rr,"f \circ \zeta_{\hat{a}+s'}"] \ar[ll,"\zeta_{\hat{a}+t}"]& & B \\
           \\
            & & \ss^{\hat{a}+r'}A\ar[uull,"\zeta_{\hat{a}+r'}"]\ar[uurr,swap,"f' \circ \zeta_{\hat{a}}"]& & 
        \end{tikzcd}\end{equation}
        Next, we show that the functor $\Xi$ is dense; i.e., every object in $W^{-1}\cc^\ss$ is isomorphic to the image of an object in $\cc_\infty$ under $\Xi$. It is clear that this is true iff in $W^{-1}\cc^\ss$ we have $\ss^aA \cong A$ for all $A$ and for all $a\in \R$. To show this, choose $a'$ large enough and consider the composition 

        \begin{equation}\begin{tikzcd}
           & & \ss^{a+2a'}A \ar[dr,dotted,"\zeta_{a'}"]\ar[dl,swap,dotted,"\zeta_{a'}"]\ar[dll,swap,bend right= 30, "\zeta_{a+2a'}"]\ar[drr,bend left= 30, "\zeta_{a+2a'}"]& & \\
           A & \ar[l,swap,"\zeta_{a'+a}"]\ss^{a'+a}\ar[r,"\zeta_{a'}"] A& \ss^{a} A & \ss^{a+a'}A\ar[l,swap,"\zeta_{a'}"] \ar[r,"\zeta_{a'+a}"]& A
        \end{tikzcd}\end{equation}
        This is in equivalent to the identity morphism on $A$. A similar argument shows composition the other way also gives the identity. To complete the proof of the equivalence, we must show that the functor $\Xi$ is fully faithful. First, injectivity. Assume $\Xi(f) = \Xi(f')$, then there exists a commutative diagram in $\cc^\ss$, 

       \begin{equation}\begin{tikzcd}
            & &\ss^{\hat{a}+r}A \ar[ddll,swap,"\zeta_{\hat{a}+r}"]\ar[ddrr,"f\circ \zeta_{\hat{a}}"]& &\\
            \\
           A & &\ar[ll,"w"] \ss^x X\ar[uu,"\phi"]\ar[dd,"\psi"]\ar[rr,"g"] & & B \\
           \\
            & & \ss^{\hat{a}+r'}A\ar[uull,"\zeta_{\hat{a}+r'}"]\ar[uurr,swap,"f' \circ \zeta_{\hat{a}}"]& & 
        \end{tikzcd}\end{equation}
        with $w$ a weighted isomorphism. Assume $w$ is an $s$-isomorphism such that there exists $v\in \text{Hom}_{\cc^\ss}(A,\ss^{x-s}X)$ with $v \circ w = \zeta_s$ and $w\circ \ss^s v= \zeta_s$. We can then replace the diagram with

         \begin{equation}\begin{tikzcd}
            & &\ss^{\hat{a}+r}A \ar[ddll,swap,"\zeta_{\hat{a}+r}"]\ar[ddrr,"f\circ \zeta_{\hat{a}}"]& &\\
            \\
           A & &\ar[ll,"\zeta_{s}"]  \ss^s A\ar[uu,"\phi\circ \ss^s v"]\ar[dd,"\psi\circ \ss^s v"]\ar[rr,"g\circ \ss^s v"] & & B \\
           \\
            & & \ss^{\hat{a}+r'}A\ar[uull,"\zeta_{\hat{a}+r'}"]\ar[uurr,swap,"f' \circ \zeta_{\hat{a}}"]& & 
        \end{tikzcd}\end{equation}
          We can then take a $t$ suitably large enough such that
        \begin{align*}
            &f \circ \zeta_{\hat{a}}\circ \phi \circ \ss^s v = g \circ \ss^s v\\
            \implies & \zeta_t \circ f \circ \zeta_{\hat{a}}\circ \phi \circ \ss^s v = \zeta_t \circ g \circ \ss^s v\\
            \implies & \zeta_{t-r} \circ \ss^r f \circ \zeta_{\hat{a}+r}\circ \phi \circ \ss^s v=  \zeta_{t} \circ g \circ \ss^s v\\
            \implies& \zeta_{t-r}\circ \ss^r f \circ \zeta_s = \zeta_{t}\circ g \circ \ss^s v\\
            \implies & \zeta_{t-r} \circ \ss^r f \circ \zeta_s = \zeta_{t} \circ f' \circ \zeta_{\hat{a}} \circ \psi \circ \ss^s v\\
            \implies &\zeta_{t-r} \circ \ss^r f \circ \zeta_s =\zeta_{t-r'}\circ  \ss^{r'}f' \circ \zeta_{\hat{a}+r'}\circ \psi \circ \ss^s v\\
            \implies & \zeta_{t-r} \circ \ss^r f \circ \zeta_s = \zeta_{t-r'} \circ \ss^{r'} f' \circ \zeta_s
        \end{align*}
        i.e., we have $\zeta_{t-r+s}\circ f = \zeta_{t-r'+s}\circ f'$, giving $[f]_\infty= [f']_\infty$. For surjectivity, let $(\begin{tikzcd}
            A & \ar[l,swap,"w"]\ss^x X \ar[r,"f"]& B
        \end{tikzcd})$ be a morphism in $\text{Hom}_{W^{-1}\cc^\ss}(A,B)$ (note we only need to work with objects of the form $A=\ss^0A$ by the isomorphisms $\ss^a A \cong \ss^0 A$). Assume $w$ is an $s$-isomorphism, such that there exists $v\in \text{Hom}_{\cc^\ss}(A,\ss^{x-s}X)$ with $v \circ w = \zeta_s$ and $w\circ \ss^s v= \zeta_s$. We can construct the following 
\begin{equation}\begin{tikzcd}
    & & \ss^x X\ar[ddll,swap,"w"]\ar[ddrr,"f"]& & \\
    \\
    A& &\ss^{s+\hat{a}}A\ar[uu,"\ss^s v\circ \zeta_{\hat{a}}"]\ar[ll,"\zeta_{\hat{a}+s}"] \ar[dd,"\zeta_{\hat{a}}"]\ar[rr,"f \circ \ss^s v \circ \zeta_{\hat{a}}"]& & B \\
    \\
    & & \ss^{s}A\ar[uull,"\zeta_{s}"]\ar[uurr,swap,"f \circ \ss^sv"]& & 
\end{tikzcd}\end{equation}
        which shows that $(\begin{tikzcd}
            A & \ar[l,swap,"w"]\ss^x X \ar[r,"f"]& B
        \end{tikzcd}) = \Xi(f \circ \ss^s v)$. 
        \end{proof}
      
    \end{lemma}

\newpage

\subsection{Weighted Idempotents and Karoubi Completion}
We now look to understand what the idempotent completion of our persistence semi-category should look like. We will see that every idempotent in the limit category is represented by some weighted idempotent in the `flattened' category. Thus in order to split all idempotents in the limit category we need to define what it meant to split weighted idempotents in the flattened category.

Consider an idempotent morphism $e:A \to A$ in the localisation $W^{-1}\cc^\ss$, i.e., the limit category $\cc_\infty$. Assume $e$ is of the form $\begin{tikzcd}(A & \ar[l,swap,"w"]\ss^xX \ar[r,"f"]& A)\end{tikzcd}$. Assuming $w$ is an $r$-isomorphism, we can also represent $e$ by $\begin{tikzcd}(A & \ar[l,swap,"\zeta_r"]\ss^r A \ar[r,"f'"]& A)\end{tikzcd}$, note $r\geq \zeta_\cc(A)$. This morphism is an idempotent so we have 
\begin{equation}\begin{tikzcd}(A & \ar[l,swap,"\zeta_r"]\ss^r A \ar[r,"f'"]& A) \circ (A & \ar[l,swap,"\zeta_r"]\ss^r A \ar[r,"f'"]& A) \sim (A & \ar[l,swap,"\zeta_r"]\ss^r A \ar[r,"f'"]& A)\end{tikzcd}\end{equation}
We can compose as follows
\begin{equation}\begin{tikzcd}
    &  &\ss^{2r} A\ar[dl,swap,dotted,"\zeta_r"]\ar[dr,dotted,"\ss^r f'"]\ar[dll,swap,bend right = 30,"\zeta_{2r}"]\ar[drr,bend left = 30,"f' \circ \ss^r f' "]  &  & \\
    A  & \ar[l,swap,"\zeta_r"]\ss^r A \ar[r,"f'"] & A &  \ar[l,swap,"\zeta_r"]\ss^r A \ar[r,"f'"]&  A
\end{tikzcd}\end{equation}

and so we have that $\begin{tikzcd}(A & \ar[l,swap,"\zeta_{2r}"]\ss^{2r} A \ar[r,"f' \circ \ss^r f'"]& A) \sim (A & \ar[l,swap,"\zeta_r"]\ss^r A \ar[r,"f'"]& A)\end{tikzcd}$. In particular there exists a diagram in $\cc^\ss$,
\begin{equation}\begin{tikzcd}
            & &\ss^{2r}A \ar[ddll,swap,"\zeta_{2r}"]\ar[ddrr,"f' \circ \ss^r f'"]& &\\
            \\
           A & &\ar[ll,"u"] \ss^y Y\ar[uu,"\phi"]\ar[dd,"\psi"]\ar[rr,"g"] & & A \\
           \\
            & & \ss^r A\ar[uull,"\zeta_{r}"]\ar[uurr,swap,"f'"]& & 
        \end{tikzcd}\end{equation}

Assuming $u$ is an $s$-isomorphism we can replace this with a diagram of the form
\begin{equation}\begin{tikzcd}
            & &\ss^{2r}A \ar[ddll,swap,"\zeta_{2r}"]\ar[ddrr,"f' \circ \ss^r f'"]& &\\
            \\
           A & &\ar[ll,"\zeta_s"] \ss^s A\ar[uu,"\phi'"]\ar[dd,"\psi'"]\ar[rr,"g'"] & & A \\
           \\
            & & \ss^r A\ar[uull,"\zeta_{r}"]\ar[uurr,swap,"f'"]& & 
        \end{tikzcd}\end{equation}

i.e., we have 
\begin{align*}
   & f' \circ \ss^r f' \circ \phi' = f' \circ \psi\\
   \implies & f' \circ \ss^r f' \circ \phi' \circ \zeta_{2r} = f' \circ \psi \circ \zeta_{2r}\\
   \implies & f' \circ \ss^r f' \circ \zeta_s = f' \circ \zeta_r \circ \zeta_s
\end{align*}
We can take $s$ large enough so that we have $(f' \circ \zeta_{s/2}) \circ (\ss^{r+s/2} f' \circ \zeta_{s/2}) = (f' \circ \zeta_{s/2}) \circ \zeta_{r+s/2}$. Set $e_A:= f' \circ \zeta_{s/2}$ then we have 
\begin{equation}e_A \circ \ss^t e_A =  e_A \circ \zeta_t\end{equation}
with $t=r+s/2$. This means that an idempotent $e\in \text{Hom}_{\cc_\infty}(A,A)$ in the limit category can be represented by a morphism $e_A:A \to \ss^{-t}A$ with $e_A \circ \ss^t e_A = \zeta_t \circ e_A$. We call such a morphism a \textbf{weighted idempotent of weight $t$} or simply a $t$-idempotent. 

\begin{rmk}
    A weight zero idempotent, $e_A\in \text{Hom}_{\cc^\ss}(A,A)$ is exactly an idempotent in the usual sense. This would require $\zeta_\cc(A)=0$. Also note that $\ss^ae_A\in \text{Hom}_{\cc^\ss}(\ss^a A,\ss^{a-r} A)$ would also represent the same idempotent as $e_A:A \to \ss^{-r}A$ in $\cc_\infty$ by the canonical identification $\text{Hom}_{\cc^\ss}(A,\ss^{-r}A) \simeq \text{Hom}_{\cc^\ss}(\ss^a A,\ss^{a-r}A)$.
\end{rmk}
\begin{rmk}
    The map $\zeta_r:A \to \ss^{-r}A$ is naturally an $r$-idempotent.
\end{rmk}

\begin{defn}
We say an $r$-idempotent $e:A \to \ss^{-r}A$ \textbf{splits} if there exists a triple $(B,s:B \to A, r:A \to \ss^{-r}B)$ with 
\begin{align}r \circ s =&  \zeta_r \\
\ss^{-r}s \circ r =& e \notag
\end{align}

\begin{equation}\begin{tikzcd}
    A \ar[r,"e"]\ar[dr,"r"]& \ss^{-r}A\\
    B \ar[u,"s"]\ar[r,swap,"\zeta_r"]& \ss^{-r}B\ar[u,swap,"\ss^{-r}s"]
\end{tikzcd}\end{equation}
We define a \textbf{weak splitting} of an $r$-idempotent $e$ to be a splitting of $\zeta_r \circ e$.
 \end{defn}

\begin{defn}
    We will refer to such a $(B,s,r)$ as a \textbf{weighted retract of weight $r$} or simply an $r$-retract. If $B$ is an $r$-retract of $A$ we will write $B<_r A$, i.e., $B<_r A$ if there exists a diagram
    \begin{equation}\begin{tikzcd}
    A \ar[dr,"r"]& \\
    B \ar[u,"s"]\ar[r,swap,"\zeta_r"]& \ss^{-r}B
\end{tikzcd}\end{equation}
\end{defn}

Given a semi-persistence category $\cc$ we constructed $\cc^\ss$, its `flattened' category. We can then localise to form a true category $W^{-1}\cc^\ss$. We have shown that this is equivalent to the limit category $\cc_\infty$. Now consider the Karoubi Envelope $\text{Split}(\cc_\infty) \simeq \text{Split}(W^{-1}\cc^\ss)$. We wish to construct a completion of $\cc^\ss$, which we will denote by $\widehat{\text{Split}}(\cc^\ss)$, such that the corresponding localisation with respect to weighted isomorphisms is (up to equivalence) the Karoubi envelope of the original limit category, i.e., we would like $W^{-1}\widehat{\text{Split}}(\cc^\ss)\simeq \text{Split}(W^{-1}\cc^\ss)$.  

\begin{defn}

We define $\widehat{\text{Split}}(\cc^\ss)$ to be the category with objects $(\ss^aA,\ss^a e_A)$, where $\ss^aA\in \text{Obj}(\cc^\ss)$ and $e_A: A \to \ss^{-r}A$ an $r$-idempotent for some $r\geq 0$. A morphism $f:(\ss^a A,\ss^a e_A) \to (\ss^b B,\ss^b e_B)$ (with $e_A$ an $r$-idempotent and $e_B$ and $s$-idempotent) is a morphism $f\in \text{Hom}_{\cc^\ss}(\ss^aA,\ss^bB)$ such that
\begin{equation}f \circ \ss^{a+r}e_A= f \circ \zeta_r\end{equation}
\begin{equation}\ss^b e_B \circ f = \zeta_s \circ f\end{equation}
in particular we have the following commutative diagram 
\begin{equation}\begin{tikzcd}
    \ss^{a+r}A\ar[r,yshift= 0.75ex,"\ss^{a+r}e_A"]\ar[r,yshift = -0.75ex, swap, "\zeta_r"] & \ss^a A\ar[r,"f"] & \ss^b B \ar[r,yshift= 0.75ex,"\ss^b e_B"]\ar[r,yshift= -0.75ex,swap,"\zeta_s"]& \ss^{b-s}B
\end{tikzcd}\end{equation}

\end{defn}

\begin{rmk}
    Two maps $f,g\in \text{Hom}_{\cc^\ss}(\ss^a A,\ss^b B)$ that induce maps in
    \newline $\text{Hom}_{\widehat{\text{Split}}_\pp(\cc ^\ss)}((\ss^a A,\ss^ae_A),(\ss^b B,\ss^b e_B))$, induce equal morphism if $f\circ \zeta_r = g\circ \zeta_r$ and $\zeta_s \circ f = \zeta_s \circ g$.
\end{rmk}
\begin{rmk}
 We can realise $\widehat{\text{Split}}_\pp(\cc^\ss)$ as the flattened category of a semi-persistence category $\widehat{\text{Split}}_\pp(\cc)$. This has objects that are pairs $(A,e_A)$ with $A\in \cc$ and $e_A\in \text{Hom}_\cc(A,A)(r)$ such that $e_A \circ e_A= i_{r,2r}\circ e_A= \zeta_r \circ e_A$. A morphism $f\in \text{Hom}_{\widehat{\text{Split}}_\pp(\cc)}((A,e_A),(B,e_B))(t)$ is a morphism $f \in \text{Hom}_\cc(A,B)(t)$ such that 
    \begin{align}e_B \circ f =&i_{t,s} \circ f\\
    f \circ e_A =& i_{t,r} \circ f \notag
    \end{align}
    where $e_B\in \text{Hom}_\cc(B,B)(s)$. Then there is a natural identification $\widehat{\text{Split}}_\pp(\cc^\ss) = \widehat{\text{Split}}_\pp(\cc)^\ss$, where $\ss^a(A,e_A)\mapsto (\ss^a A,\ss^ae_A)$.
    \end{rmk}
There is a Yoneda type embedding $\hat{\yy}:\cc^\ss \hookrightarrow \widehat{\text{Split}}_\pp(\cc^\ss)$, given by $\ss^a A \mapsto (\ss^a A,\ss^a \zeta_{\lfloor A \rfloor})$ on objects. On morphisms, if $f\in \text{Hom}_{\cc^\ss}(\ss^a A,\ss^b B)$ then $\hat{\yy}(f):= f:(\ss^a A,\ss^a\zeta_{\lfloor A \rfloor}) \to (\ss^b B,\ss^b \zeta_{\lfloor B \rfloor})$. The embedding can be seen to be full, however it is only faithful up to an equivalence. That is, if $\zeta_{\min\{\lfloor A \rfloor , \lfloor B \rfloor\}}\circ f = \zeta_{\min\{\lfloor A \rfloor , \lfloor B \rfloor\}} \circ g$
then we cannot distinguish $f$ and $g$ as we have $ f \circ \zeta_{\lfloor A \rfloor} = g \circ \zeta_{\lfloor A \rfloor}$ and $\zeta_{\lfloor B \rfloor} \circ f = \zeta_{\lfloor B \rfloor} \circ g$.
\begin{prop}
If $(\ss^aA,\ss^ae_A)$ is an object with $e_A$ an $r$-idempotent then  $\lfloor (\ss^a A,\ss^a e_A) \rfloor\leq  r$. Moreover $\zeta_r^{(\ss^a A,\ss^a e_A)}=e_A$.

\begin{proof}
Assume $f\in \text{Hom}_{\widehat{\text{Split}}_\pp(\cc^\ss)}((\ss^a A, \ss^a e_A),(\ss^b B,\ss^b e_B))$, composition then gives 
\begin{align*}
(\ss^{-r}f \circ \ss^a e_A)\circ \ss^{a+r}e_A =& \ss^{-r}f \circ \ss^a e_A\circ \zeta_r\\ 
=&(\ss^{-r}f \circ \zeta_r) \circ \zeta_r
\end{align*}
and
\begin{align*}
    \ss^{b-r}e_B \circ (\ss^{-r}f \circ \ss^a e_A)=& \zeta_s \circ (\ss^{-r}f \circ \ss^ae_A)\\
    =& \zeta_s \circ (\ss^{-r}f \circ \zeta_r).
\end{align*}
A similar argument shows $\ss^a e_A \circ g = \zeta_r \circ g$ in $\widehat{\text{Split}}_\pp(\cc^\ss)$ for any $g\in \text{Hom}_{\widehat{\text{Split}}_\pp(\cc^\ss)}((\ss^c C,\ss^c e_C),(\ss^a A,\ss^a e_A))$. Therefore if $e_A$ is an $r$-idempotent, we find $\lfloor (\ss^a A,\ss^a e_A)\rfloor\leq r$.
\end{proof}
\end{prop}
\begin{rmk}
    The bound above is not in general an equality. For example $\lfloor (A,\zeta_{r}) \rfloor = \lfloor A \rfloor$ for any $r$, including $r> \lfloor A \rfloor$.  
\end{rmk}

\begin{prop}
If $e_A:A \to \ss^{-r}A$ is an $r$-idempotent, then there is a canonical weak splitting in $\widehat{\text{Split}}_\pp(\cc)$, given by

    \begin{equation}(A,\zeta_{\lfloor A \rfloor}) \xrightarrow{e_A} \ss^{-r}(A,e_A) \xrightarrow{\ss^{-r}e_A} \ss^{-2r}(A,\zeta_{\lfloor A \rfloor})\end{equation}

\begin{proof}
We have $\ss^{-r}e_A \circ e_A = \zeta_r \circ e_A $. $e_A=\zeta_{r}^{(A,e_A)}\in \text{Hom}_{\widehat{\text{Split}}_\pp(\cc^\ss)}((A,e_A),\ss^{-r}(A,e_A))$ giving $\ss^{-r}e_A \circ e_A =\zeta_{2r}^{(A,e_A)}$. Hence we have a commutative diagram in $\widehat{\text{Split}}_\pp(\cc^\ss)$
\begin{equation}\begin{tikzcd}
    (A,\zeta_{\lfloor A \rfloor}) \ar[r,"\zeta_r \circ e_A"]\ar[dr,"e_A"]& \ss^{-2r}(A,\zeta_{\lfloor A \rfloor})\\
    \ss^{r}(A,e_A)\ar[u,"\ss^{r}e_A"]\ar[r,"\zeta_{2r}"] & \ss^{-r}(A,e_A)\ar[u,swap,"\ss^{-r}e_A"]
\end{tikzcd}\end{equation}

\end{proof}
\end{prop}

\subsection{The limit category}

We now show that the limit category of the idempotent completion of our constructed persistence semi-category recovers the usual idempotent completion of the limit of the original persistence semi-category. That is, we show the following result:

\begin{lemma}
    If $\cc$ is a persistence semi-category, then $W^{-1}(\widehat{\text{Split}}_\pp(\cc^{\ss})) $ is equivalent (as a category) to $\text{Split}(W^{-1}\cc^\ss)$.
    \begin{proof}
      To fix notation, let $f\in \text{Hom}_{\cc}(A,B)(r)= \text{Hom}_{\cc^\ss}(\ss^r A,B)$ we will denote $[f]_\infty\in \text{Hom}_{W^{-1}\cc^\ss}(A,B)$ the equivalence class of $\begin{tikzcd}
        A & \ss^{\lfloor A \rfloor +r}\ar[l,swap,"\zeta_{\lfloor A \rfloor +r}"]\ar[r,"f \circ \zeta_{\lfloor A \rfloor}"] & B
      \end{tikzcd}$. We define a functor $\Gamma : W^{-1}\widehat{\text{Split}}_\pp(\cc^\ss) \to \text{Split}(W^{-1}\cc^\ss)$ as follows; on objects it is given by $\Gamma((A,e_A))= (A,[e_A]_\infty)$. On morphisms: if $\begin{tikzcd}
                (A,e_A) & \ar[l,swap,"u"](C,e_C)\ar[r,"f"] & (B,e_B)
            \end{tikzcd}\in \text{Hom}_{W^{-1}\widehat{\text{Split}}_\pp(\cc^\ss)}((A,e_A),(B,e_B))$, then we can find an equivalent morphism of the form $\begin{tikzcd}
                (A,e_A) & \ar[l,swap,"\zeta_t"]\ss^{t}(A,e_A)\ar[r,"f'"] & (B,e_B)
            \end{tikzcd}$
            for some $t$ large enough. $f'$ is a morphism in $\text{Hom}_{\cc^\ss}(\ss^t A,B)$ such that $f' \circ \ss^t e_A= f'\circ \zeta_r$ and $e_B \circ f' = \zeta_s \circ f'$. Hence in $W^{-1}\cc^\ss$ we find $[e_B]_\infty \circ [f']_\infty=[e_B \circ f']_{\infty}= [\zeta_s \circ f']_\infty = [f']_\infty$ and $[f']_\infty \circ [e_A]_\infty=[f' \circ e_A]_\infty = [f' \circ \zeta_r]_\infty = [f']_\infty$. We can therefore define $\Gamma(f):= [f']_\infty$. Note that this does not depend on the choice of replacement $\begin{tikzcd}
                (A,e_A) & \ar[l,swap,"\zeta_t"]\ss^{t}(A,e_A)\ar[r,"f'"] & (B,e_B)
            \end{tikzcd}$. If we replaced the morphism with $\begin{tikzcd}
                (A,e_A) & \ar[l,swap,"\zeta_{k}"]\ss^{k}(A,e_A)\ar[r,"f''"] & (B,e_B)
            \end{tikzcd}$ then $[f'']_\infty = [f']_\infty$. As $\begin{tikzcd}
                ((A,e_A) & \ar[l,swap,"\zeta_t"]\ss^{t}(A,e_A)\ar[r,"f'"] & (B,e_B)) \sim ((A,e_A) & \ar[l,swap,"\zeta_{k}"]\ss^{k}(A,e_A)\ar[r,"f''"] & (B,e_B))
            \end{tikzcd}$
           
        \begin{enumerate}
            \item $\Gamma$ is well defined. If $(A,e_A)\in W^{-1}\widehat{\text{Split}}_\pp(\cc^\ss)$ then $e_A$ is an $r$-idempotent $e_A\in \text{Hom}_{\cc^\ss}(A,\ss^{-r}A)$ for some $r$. $e_A$ therefore represents some idempotent $[e_A]_\infty\in \text{Hom}_{W^{-1}\cc^\ss}(A,A)$. Giving the object $(A,e_A)\in \text{Split}(W^{-1}\cc^\ss)$.
            
            Let $\begin{tikzcd}
                (A,e_A) & \ar[l,swap,"u"](C,e_C)\ar[r,"f"] & (B,e_B)
            \end{tikzcd}\in \text{Hom}_{W^{-1}\widehat{\text{Split}}_\pp(\cc^\ss)}((A,e_A),(B,e_B))$, be a morphism and assume $\begin{tikzcd}
                (A,e_A) & \ar[l,swap,"v"](D,e_{D})\ar[r,"g"] & (B,e_B)
            \end{tikzcd}$ is an equivalent morphism. We can find morphisms $\begin{tikzcd}
                (A,e_A) & \ar[l,swap,"\zeta_t"]\ss^{t}(A,e_A)\ar[r,"f'"] & (B,e_B)
            \end{tikzcd}$ and
             $\begin{tikzcd}
                (A,e_A) & \ar[l,swap,"\zeta_k"]\ss^{k}(A,e_A)\ar[r,"g'"] & (B,e_B)
            \end{tikzcd}$ giving  $\Gamma(\begin{tikzcd}
                (A,e_A) & \ar[l,swap,"u"](C,e_C)\ar[r,"f"] & (B,e_B)
            \end{tikzcd})=[f']_\infty$ and $\Gamma(\begin{tikzcd}
                (A,e_A) & \ar[l,swap,"v"](D,e_{D})\ar[r,"g"] & (B,e_B)
            \end{tikzcd})=[g']_\infty$. Then there is a diagram in $\widehat{\text{Split}}_\pp(\cc^\ss)$,

            \begin{equation}\begin{tikzcd}
                 & \ar[dl,swap,"\zeta_t"]\ss^t(A,e_A)\ar[rd,"f'"] & \\
                (A,e_A) & \ar[l,"\zeta_l"]\ss^{l}(A,e_A)\ar[u,"\phi"]\ar[d,"\psi"] \ar[r,"h"]& (B,e_B)\\
                & \ar[lu,"\zeta_k"]\ss^k(A,e_A)\ar[ru,swap,"g'"] & 
            \end{tikzcd}\end{equation}
            giving $[f']_\infty = [f' \circ \phi]_\infty = [h]_\infty = [g' \circ \psi]_\infty = [g']_\infty$.
            \item $\Gamma$ is essentially surjective. Let $(A,\bar{e}_A)\in \text{Split}(W^{-1}\cc^\ss)$. $\bar{e}_A= (\begin{tikzcd}
                A &\ar[l] C \ar[r]& A
            \end{tikzcd})$ is an idempotent in $W^{-1}\cc^\ss$ and is represented by some $r$-idempotent $e_A\in \text{Hom}_{\cc^\ss}(A,\ss^{-r}A)$. It follows that $(A,\bar{e}_A)\cong (A,[e_A]_\infty)=\Gamma((A,e_A))$.
            \item $\Gamma$ is fully faithful. Let $(A,\bar{e}_A)$ and $(B,\bar{e}_B)$ be objects in $\text{Split}(W^{-1}\cc^\ss)$, we can choose representative idempotent roof diagrams of the form $\bar{e}_A=(\begin{tikzcd}
                A & \ss^pA \ar[l,swap,"\zeta_p"]\ar[r,"\tilde{e}_A"]& A
            \end{tikzcd})$ and $\bar{e}_B= (\begin{tikzcd}
                B & \ss^qB\ar[l,swap,"\zeta_q"]\ar[r,"\tilde{e}_B"]& B)
            \end{tikzcd}$, these are idempotent morphisms in $W^{-1}\cc^\ss$. Let $\begin{tikzcd}
                A & K \ar[l,swap,"u"]\ar[r,"f"]& B
            \end{tikzcd}\in \text{Hom}_{\text{Split}(W^{-1}\cc^\ss)}((A,\bar{e}_A),(B,\bar{e}_B))$  be a well defined morphism, i.e., there are equivalences of roof diagrams
            \begin{equation}\begin{tikzcd}
                ( B & \ss^qB \ar[l,swap,"\zeta_q"]\ar[r,"\tilde{e}_B"]& B) \circ ( A & K \ar[l,swap,"u"]\ar[r,"f"]& B) \sim  ( A & K\ar[l,swap,"u"]\ar[r,"f"]& B)\end{tikzcd}\end{equation}
                and 
                \begin{equation} \begin{tikzcd}( A & K \ar[l,swap,"u"]\ar[r,"f"]& B)\circ(A & \ss^pA \ar[l,swap,"\zeta_p"]\ar[r,"\tilde{e}_A"]& A) \sim ( A & K\ar[l,swap,"u"]\ar[r,"f"]& B)\end{tikzcd}\end{equation}

            We can replace $\begin{tikzcd}
                 A & K \ar[l,swap,"u"]\ar[r,"f"]& B
            \end{tikzcd}$  with$\begin{tikzcd}
                A & \ss^tA \ar[l,swap,"\zeta_t"]\ar[r,"f'"]& B
            \end{tikzcd}$ where $f'$ belongs to $ \text{Hom}_{\cc^\ss}(\ss^tA,B)$. We can therefore find equivalences 
        \begin{equation}\begin{tikzcd}
                ( B & \ss^qB \ar[l,swap,"\zeta_q"]\ar[r,"\tilde{e}_B"]& B) \circ ( A & \ss^tA \ar[l,swap,"\zeta_t"]\ar[r,"f'"]& B) \sim  (  A & \ss^tA \ar[l,swap,"\zeta_t"]\ar[r,"f'"]& B)\end{tikzcd}\end{equation}
                and 
                \begin{equation} \begin{tikzcd}(  A & \ss^tA \ar[l,swap,"\zeta_t"]\ar[r,"f'"]& B)\circ(A & \ss^pA \ar[l,swap,"\zeta_p"]\ar[r,"\tilde{e}_A"]& A) \sim (  A & \ss^tA \ar[l,swap,"\zeta_t"]\ar[r,"f'"]& B)\end{tikzcd}\end{equation}

                Explicitly, the first composition can be defined via the following diagram (taking $p,t$ suitably large)
                \begin{equation}\begin{tikzcd}
                 & & \ss^{t+p}A\ar[dll,swap,bend right = 30,"\zeta_{t+p}"]\ar[drr,bend left = 30,"f' \circ \ss^t\tilde{e}_A"]\ar[dl,swap,dotted,"\zeta_t"]\ar[dr,dotted,"\ss^t\tilde{e}_A"] & &   \\
                   A & \ss^pA \ar[l,swap,"\zeta_p"]\ar[r,"\tilde{e}_A"]& A& \ss^tA \ar[l,swap,"\zeta_t"]\ar[r,"f'"]& B  
                \end{tikzcd}\end{equation}
                The equivalence implies the existence of a diagram 

                \begin{equation}\begin{tikzcd}      
 & &\ss^{t+p}A\ar[dll,swap,"\zeta_{t+p}"]\ar[drr,"f' \circ \ss^t \tilde{e}_A"]&& \\
                    A& & \ss^{t+p+l}A\ar[u,"\phi"]\ar[d,"\psi"]\ar[ll,swap,"\zeta_{t+p+l}"] \ar[rr,"f''"]& &B\\
                   & & \ss^{t}A\ar[urr,swap,"f'"]\ar[ull,"\zeta_t"]& &
                \end{tikzcd}\end{equation}
                $\phi$ and $\psi$ can be seen to belong to $W$. We can then consider 
                \begin{align*}
&\hspace{0.75cm}f' \circ \ss^t\tilde{e}_A \circ \phi \circ \zeta_{t+p} =f' \circ \psi \circ \zeta_{t+p}\\
&\implies f' \circ \ss^{t}\tilde{e}_A \circ \zeta_{t+p+l}=f' \circ \zeta_{t+p+l}\circ \zeta_p 
                \end{align*}
Similarly, the other composition can be given by

                \begin{equation}\begin{tikzcd}
                  & & \ss^{t+q}A\ar[dll,swap,bend right = 30,"\zeta_{t+q}"]\ar[dr,dotted,"\ss^q f'"]\ar[drr,bend left = 30,"\tilde{e}_B \circ \ss^q f'"]\ar[dl,dotted,swap,"\zeta_q"]\\
                  A & \ss^tA \ar[l,swap,"\zeta_t"]\ar[r,"f'"]&  B & \ss^qB \ar[l,swap,"\zeta_q"]\ar[r,"\tilde{e}_B"]& B
                \end{tikzcd}\end{equation}
                and there exists a diagram 
                \begin{equation}\begin{tikzcd}      
 & &\ss^{t+q}A\ar[dll,swap,"\zeta_{t+q}"]\ar[drr,"\tilde{e}_B \circ \ss^q f'"]&& \\
                    A& & \ss^{t+q+l'}A\ar[u,"\phi'"]\ar[d,"\psi'"]\ar[ll,swap,"\zeta_{t+q+l'}"] \ar[rr,"f'''"]& &B\\
                   & & \ss^{t}A\ar[urr,swap,"f'"]\ar[ull,"\zeta_t"]& &
                \end{tikzcd}\end{equation}
                and so $f' \circ \zeta_{t+q+l'}\circ \zeta_q = \tilde{e}_B \circ \ss^qf' \circ \zeta_{t+q+l'}$. Let $m=\max\{t+q+l',t+p+l\}$ and set $\hat{f}=f' \circ \zeta_m$, then $\hat{f}\in \text{Hom}_{\widehat{\text{Split}}_\pp(\cc^\ss)}((\ss^{t+m}A,\ss^{t+m-p}\tilde{e}_A),(B,\ss^{-q}\tilde{e}_B))$. This gives a roof diagram in $W^{-1}\widehat{\text{Split}}_\pp(\cc^\ss),$ $\begin{tikzcd}
                    \ss^{t+m}(A,\ss^{-p}\tilde{e}_A) &\ar[l,swap,"\zeta_p"] \ss^{t+m+p}(A,\ss^{-p}\tilde{e}_A) \ar[r,"\hat{f}\circ \zeta_p "]& (B,\ss^{-q}\tilde{e}_B)
                \end{tikzcd}$. The roof diagram $\begin{tikzcd}(A,\ss^{-p}\tilde{e}_A) &\ss^{n}(A,\ss^{-p}\tilde{e}_A)\ar[l,swap,"\zeta_n"]\ar[r,"\zeta_{n-(t+m)}"] & \ss^{t+m}(A,\ss^{-p}\tilde{e}_A)\end{tikzcd}$ can be seen to be an isomorphism $\ss^{t+m}(A,\ss^{-p}\tilde{e}_A) \cong (A,\ss^{-p}\tilde{e}_A)$. Finally we obtain a roof diagram given by the composition of the two previous diagrams 
                \begin{equation}\begin{tikzcd}
                    & &\ss^{n+p}(A,\ss^{-p}\tilde{e}_A)\ar[dr,dotted,"\zeta_{n-(t+m)}"]\ar[dl,dotted,swap,"\zeta_p"]\ar[drr,bend left = 30, "\hat{f} \circ \zeta_{n+p-(t+m)}"]\ar[dll,swap,bend right =30, "\zeta_{p+n}"] & & \\
                    (A,\ss^{-p}\tilde{e}_A) &\ss^{n}(A,\ss^{-p}\tilde{e}_A)\ar[l,swap,"\zeta_n"]\ar[r,"\zeta_{n-(t+m)}"] & \ss^{t+m}(A,\tilde{e}_A)&\ar[l,swap,"\zeta_p"] \ss^{t+m+p}(A,\ss^{-p}\tilde{e}_A) \ar[r,"\hat{f}\circ \zeta_p "]& (B,\ss^{-q}\tilde{e}_B)
                \end{tikzcd}\end{equation}
                for some large enough $n$. We therefore find $\Gamma((A,\ss^{-p}\tilde{e}_A))=(A,[\ss^{-p}\tilde{e}_A]_\infty)= (A,\bar{e}_A)$, likewise $\Gamma((B,\ss^{-q}\tilde{e}_B))= (B,\bar{e}_B)$, and $\Gamma( \begin{tikzcd}
                  (A,\ss^{-p}\tilde{e}_A) &&\ss^{n+p}(A,\ss^{-p}\tilde{e}_A)\ar[rr, "\hat{f} \circ \zeta_{n+p-(t+m)}"]\ar[ll,swap ,"\zeta_{p+n}"] && (B,\ss^{-q}\tilde{e}_B)
                \end{tikzcd})$ is $\begin{tikzcd}
                    (A,\bar{e}_A) & \ar[l,swap,"\zeta_{p+n}"]\ss^{n+p}(A,\bar{e}_A)\ar[r,"\hat{f}\circ \zeta_k"] & (B,\bar{e}_B)
                \end{tikzcd}$ with $k=n+p-(t+m)$, which is  equivalent to $\begin{tikzcd}(A,\bar{e}_A) & \ss^t (A,e_A)\ar[r,"f'"]\ar[l,swap,"\zeta_t"] & (B,\bar{e}_B)\end{tikzcd}$ as $\hat{f}=f' \circ \zeta_m$.
    Finally we show that $\Gamma$ is faithful. Assume $\Gamma(\begin{tikzcd}(A,e_A)& \ss^{t}\ar[l,swap,"\zeta_t"](A,e_A) \ar[r,"f"]& (B,e_B)\end{tikzcd})= \Gamma(\begin{tikzcd}(A,e_A)& \ar[l,swap,"\zeta_t'"]\ss^{t'}(A,e_A)\ar[r,"f'"] & (B,e_B)\end{tikzcd}) $, in particular 
    \begin{equation}\begin{tikzcd}((A,[e_A]_\infty)& \ss^{t+r}\ar[l,swap,"\zeta_{t+r}"](A,[e_A]_\infty) \ar[r,"f\circ \zeta_r"]& (B,[e_B]_\infty))\sim (A,[e_A]_\infty)& \ar[l,swap,"\zeta_{t'+r}"]\ss^{t'+r}(A,[e_A]_\infty)\ar[r,"f'\circ \zeta_r"] & (B,[e_B]_\infty) \end{tikzcd}\end{equation}
    where $e_A$ is an $r$-idempotent, (note $r\geq \lfloor(A,e_A) \rfloor$). This implies the existence of the following diagram 
    \begin{equation}\begin{tikzcd}
        & & \ss^{t+r}(A,[e_A]_\infty)\ar[dll,swap,"\zeta_{t+r}"]\ar[drr,"f \circ \zeta_{r}"]& & \\
        (A,[e_A]_\infty) & & \ss^{k}(A,[e_A]_\infty)\ar[ll,swap,"\zeta_k"]\ar[rr,"f''"]\ar[u,"\phi"]\ar[d,"\psi"]& & (B,[e_B]_\infty)\\
        & &\ss^{t'+r}(A,[e_A]_\infty)\ar[urr,swap,"f' \circ \zeta_r"]\ar[ull,"\zeta_{t'+r}"] & &
    \end{tikzcd}\end{equation}
    Take $m\geq\max\{t+r,t+r'\}$ to be suitably large, then 
    \begin{align*}
&f \circ \zeta_r \circ \phi \circ \zeta_m =  f' \circ \zeta_r \circ \psi \circ \zeta_m\\
\implies &f \circ \zeta_r \circ \zeta_k \circ \zeta_{m-(t+r)}= f' \circ \zeta_r \circ \zeta_{k} \circ \zeta_{m-(t'+r)}
    \end{align*}
    hence we can construct an equivalence diagram
\begin{equation}\begin{tikzcd}
    & & \ss^t(A,e_A)\ar[dll,swap,"\zeta_t"]\ar[drr,"f"]& & \\
   (A,e_A) & & \ss^{n}(A,e_A)\ar[ll,swap,"\zeta_n"]\ar[u,"\zeta_{n-t}"]\ar[d,"\zeta_{n-t'}"]\ar[rr,"f \circ \zeta_{n-t}"]& & (B,e_B)\\
    & &\ss^{t'}(A,e_A)\ar[ull,"\zeta_{t'}"]\ar[urr,swap,"f'"] & & 
\end{tikzcd}\end{equation}
    for $n$ large enough. 
        \end{enumerate}
    \end{proof}
\end{lemma}

\begin{lemma}
    If every idempotent in $\cc_\infty$ is represented by a weight zero idempotent in $\cc_0$. Then the category $\widehat{\text{Split}}_\pp^0(\cc)$ given by the full subcategory of $\widehat{\text{Split}}_\pp(\cc)$ with objects $(A,e)$ with $e$ a $0$-idempotent, is a persistence category in the usual sense, and moreover its limit category is equivalent to $\text{Split}(\cc_\infty)$.
    \begin{proof}
        We realise that the flattened category $(\widehat{\text{Split}}_\pp^0(\cc))^\ss$ is exactly the full subcategory of $\widehat{\text{Split}}_\pp(\cc^\ss)$ consisting of objects $(\ss^aA,\ss^ae)$ where $e$ is a zero idempotent. Since every idempotent in $\cc_\infty \simeq W^{-1}\cc^\ss$ is represented by a weight zero idempotent, it is clear that $\Gamma$ restricted to $W^{-1}\widehat{\text{Split}}_\pp^0(\cc^\ss)$ is also essentially surjective.
    \end{proof}
\end{lemma}

\section{Persistent presheaves}\label{sectppsheav}\label{sectionpp}
One can identify the Karoubi completion of a category $\cc$ with a full subcategory of the category of presheaves $\text{PSh}(\cc)$ consisting of objects that are retracts of representable presheaves. Thus one can ask if this identification extends more naturally to the persistence setting.

Before we begin we remark that since the limit category of a persistence category $\cc$ is a $\text{Mod}_k$ enriched category, we should be considering presheaves in $\text{Mod}_k$. I.e., we can realise $\text{Split}(\cc_\infty)$ as a subcategory of $\text{PSh}(\cc_\infty):=[\cc^{\text{op}},\text{Mod}_k]$. One would therefore hope to define a persistence refinement of $\text{PSh}(\cc_\infty)$ by considering a subcategory of $[\cc^{\text{op}},\text{Mod}_k]$. However it is not clear how to realise this as a persistence category. Thus instead we should consider presheaves in $\text{Mod}_k^\pp$ the persistence category of persistence modules. 
\begin{rmk}
    The category $\text{Mod}_k^\pp$ corresponds to the category $\cc^{\pp \text{Mod}_k}$ in \cite{BCZ} and $\underline{\text{Mod}}^{\text{P}}_\text{k}$ in \cite{BS}.
\end{rmk}

The limit category of a persitence category has the same objects, hence the limit category of any subcategory of $[\cc^{\text{op}},\text{Mod}_k^\pp]$ will have objects being presheaves in $\text{Mod}_k^\pp$. Thus we should view $\text{Mod}_k$ as a subcategory of $\text{Mod}_k^\pp$.

\begin{prop}
    Given a $\text{Mod}_k$-enriched category $\cc$, its Karoubi completion can be identified with (i.e., is equivalent to) the full subcategory of retracts, of the image of $\hat{\yy}:\cc \to [\cc^{\text{op}},\text{Mod}_k^\pp]$. Where $\hat{\yy}(A)(r)=\text{Hom}_\cc(-,A)$ for all $r\in \R$ and similarly $\hat{\yy}(f:A \to B)(r):=\yy(f):\text{Hom}_\cc(-,A) \to \text{Hom}_\cc(-,B)$.

    \begin{proof}
        There is a fully faithful functor $\Phi:\text{Mod}_k \to \text{Mod}_k^\pp$, defined by setting
        \begin{align*}
            \Phi(V)(r):=&V\\
            i_{r,s}^{\Phi(V)}:=&1_V:\Phi(V)(r) \to \Phi(V)(s)\\
            \Phi(f:V \to W)(r):=&(f:V \to W)
        \end{align*}
        for all $r,s\in \R$. Indeed, faithfulness of $\Phi$ is clear, and fullness follows from the persistence module morphisms being required to commute with the persistence structure maps $i_{r,s}^{\Phi(-)}$ (which are the identity in this case). This functor induces a functor $\Phi_*:[\cc^{\text{op}},\text{Mod}_k] \to [\cc^{\text{op}},\text{Mod}_k^\pp]$ given by
        
        \begin{align*}
            [\Phi_*(F)](A)(r):=&F(A)\\
            [\Phi_*(F)](f:A \to B)(r):=&F(f).
        \end{align*}
        The composition $\hat{\yy}:=\Phi_* \circ \yy:\cc \to [\cc^{\text{op}},\text{Mod}_k^\pp]$ is thus given by 
        \begin{align*}
            \hat{\yy}(A)(r)= \yy(A)\\
            \hat{\yy}(f)(r)=\yy(f)
        \end{align*}
        and is fully faithful. Let $F$ be a retract of $\yy(A)$ in $[\cc^{\text{op}},\text{Mod}_k]$, then $\Phi_*(F)$ is a retract of $\hat{\yy}(A)$ in $[\cc^{\text{op}},\text{Mod}_k^\pp]$. Furthermore any retract of $\hat{\yy}(A)$ must be isomorphic to the image of a retract of $\yy(A)$ under $\Phi_*$. Assume that $\hat{F}$ is a retract of $\hat{\yy}(A)$, given by $s:\hat{F} \to \hat{\yy}(A)$  and $r:\hat{\yy}(A) \to \hat{F}$, such that $r \circ s=1_{\hat{F}}$. The maps $s$ and $r$ consist of a family of persistence module maps $s_B:\hat{F}(B) \to \hat{\yy}(A)(B)$ and $r_B:\hat{\yy}(A) \to \hat{F}$ with $r_B \circ s_B=1_{\hat{F}(B)}$. These maps in turn consist of families of module maps $s_B(r):\hat{F}(B)(r) \to \hat{\yy}(A)(B)(r)$ and $r_B(r):\hat{\yy}(A)(B)(r) \to \hat{F}(B)(r)$ satisfying 
        \begin{align}
            i_{r,s}^{\hat{\yy}(A)(B)}\circ s_B(r) =& s_B(s) \circ i_{r,s}^{\hat{F}(B)}\\
            i_{r,s}^{\hat{F}(B)}\circ r_B(r)=& r_B(s) \circ i_{r,s}^{\hat{\yy}(A)(B)}\notag
        \end{align}
        However by definition $ i_{r,s}^{\hat{\yy}(A)(B)}=1_{\text{Hom}_\cc(B,A)}$, hence the following commutes
        \begin{equation}
            \begin{tikzcd}
                \hat{F}(B)(t)\ar[dr,"s_B(t)"]& &\hat{F}(B)(t)\\
                 \hat{F}(B)(s)\ar[u,"i_{s,t}^{\hat{F}(B)}"]\ar[r,"s_B(s)"]& \yy(A)(B)\ar[ur,"r_B(t)"]\ar[r,"r_B(s)"]\ar[dr,"r_B(r)"] &\hat{F}(B)(s)\ar[u,swap,"i_{s,t}^{\hat{F}(B)}"]\\
                  \hat{F}(B)(r)\ar[ru,"s_B(r)"]\ar[u,"i_{r,s}^{\hat{F}(B)}"]& &\hat{F}(B)(r)\ar[u,swap,"i_{r,s}^{\hat{F}(B)}"]
            \end{tikzcd}
        \end{equation}
        One then sees that 
        \begin{align*}
            r_B(s) \circ s_B(t) \circ i^{\hat{F}(B)}_{s,t}=& r_B(s) \circ s_B(s) \\
            =&1_{\hat{F}(B)(s)}
        \end{align*}
        and 
        \begin{align*}
            i^{\hat{F}(B)}_{s,t} \circ r_B(s) \circ s_B(t)=& r_B(t) \circ s_B(t)\\
            =&1_{\hat{F}(B)(t)}
        \end{align*}
        In particular we see that all of the persistence module structure maps of $\hat{F}(B)$ are isomorphisms. Hence $\hat{F}(B)$ is isomorphic to an object in the image of $\Phi_*$.
    \end{proof}
\end{prop}
From now on we will write $\text{Split}(\cc)$  (for $\cc$ a $\text{Mod}_k$-enriched category) to denote the full subcategory of $[\cc^{\text{op}},\text{Mod}_k^\pp]$ consisting of retracts of the image of $\hat{\yy}$.

\begin{rmk}
    The above proposition allows us to view $\text{Split}(\cc)$ as the idempotent completion of $\cc$.
\end{rmk}

\subsection{The persistence category of persistent presheaves}
In order to attach a persistence structure to our category of presheaves we should not look to presheaves in sets or modules, but to presheaves in persistence modules. Thus, we begin by defining a category of presheaves in persistence modules and explore its natural persistence category structure. 
\begin{defn}
    A \textbf{persistent presheaf} on a persistence category $\cc$ is a persistence functor $F: \cc^{op}\to \text{Mod}^\pp_k$ that respects the persistence structure of $\cc$. Explicitly this means that for all $A\in \cc$ there is a persistence module $F(A):(\R,<) \to \text{Mod}_k$, and for any morphism $f\in \text{Mor}^a(A,B)$ we have a map $F(f)\in \text{Mor}^a(F(B),F(A))$, i.e., a natural transformation of functors $f: F(B) \implies F(A)$ with `shift' $a$. In particular the following diagram commutes
    \begin{equation}\begin{tikzcd}
        F(A)(s+a) &\ar[l,"F(f)(s)"] F(B)(s)\\
        F(A)(r+a) \ar[u,"i_{r+a,s+a}^{F(A)}"]& \ar[l,"F(f)(r)"] F(B)(r)\ar[u,"i_{r,s}^{F(B)}"]
    \end{tikzcd}\end{equation}
where $i_{r,s}^{F(-)} := F(-)(i_{r,s}) $ are the persistence module stucture maps on the persistence module $F(-)$. Furthermore, as $F$ is a persistence functor we have 
\begin{equation}
    \begin{tikzcd}
        \text{Mor}^s_\cc(A,B)\ar[r,"F_{AB}^s"] & \text{Mor}^s_{\text{Mod}^\pp_k}(F(B),F(A))\\
        \text{Mor}^r_\cc(A,B)\ar[u,"i^\cc_{r,s}"]\ar[r,"F_{AB}^r"] & \text{Mor}_{\text{Mod}_k^\pp}^r(F(B),F(A))\ar[u,"i^{\text{Mod}_k^\pp}_{r,s}"]
    \end{tikzcd}
\end{equation}
commutes.
\end{defn}
\begin{defn}
Let $\cc$ be a persistence category, we define the category of \textbf{Persistent Presheaves} on $\cc$, $\text{PSh}_\pp (\cc)$ to be the persistence functor category:
\begin{equation}\text{PSh}_\pp(\cc):=[\cc^{op},\text{Mod}_k^\pp]_\pp.\end{equation} 
 The objects of $\text{PSh}_\pp (\cc)$ are persistent presheaves. A morphism $\nu \in \text{Mor}^s_{\text{PSh}_\ff(\cc)}(F,G)$
is a natural transformation $\nu: F \implies G$ such that for all $A\in \cc$, $\lceil \nu_A \rceil = s$. For any $f\in \text{Mor}^r_{\cc}(A,B)$ and $\nu\in \text{Mor}^s_{\text{PSh}_\ff(\cc)}(F,G)$, we have the following commutative diagrams:

\begin{equation}\begin{tikzcd}
    F(A)\arrow[r,"\nu_A"] & G(A)\\
    F(B)\arrow[u,"F(f)"] \arrow[r,"\nu_B"]& G(B)\arrow[u,swap,"G(f)"]
\end{tikzcd}\end{equation}

\begin{equation}\begin{tikzcd}
     F(A)(t)\arrow[r,"\nu_A(t)"] & G(A)(t+s)\\
    F(B)(t+r)\arrow[u,"F(f)(t+r)"] \arrow[r,"\nu_B(t+r)"]& G(B)(t+s+r)\arrow[u,swap,"G(f)(t+s)"]
\end{tikzcd}\end{equation}

\end{defn}

\begin{prop}
    The category $\text{PSh}_\pp(\cc)$ is a Persistence Category. 
    \begin{proof}
        We need to define the $k$-module structure on $\text{Mor}^r_{\text{PSh}_\pp(\cc)}(F,G)$. We do so by using the persistence category structure of $\text{Mod}^\pp_k$. Let $a,b\in k$ and $\nu,\zeta \in \text{Mor}_{\text{PSh}_\pp(\cc)}^s(F,G)$, we define $a \cdot \nu - b \cdot \zeta \in \text{Mor}^s_{\text{PSh}_\pp(\cc)}(F,G)$ as follows:

        \begin{equation}[a \cdot \nu_A - b \cdot \zeta_A ](t)= a \cdot (\nu_A(t))-b \cdot (\zeta_A (t)) \in \text{Hom}_{\text{Mod}_k}(F(A)(t),G(A)(t+s))\end{equation}
        using the fact $\nu_A(t)$ and $\zeta_A(t)$ are $k$-module morphisms. The persistence structure maps $i_{r,s}:\text{Mor}^r(F,G) \to \text{Mor}^s(F,G)$ for $r\leq s$ are given by 
        \begin{equation}(i_{r,s}\circ \nu )_A = \iota_{r,s}\circ \nu_A \end{equation}
        where $\iota_{r,s}$ is the persistence structure map in $\text{Mod}^\pp_k$, $\iota_{r,s}:\text{Mor}^r_{\text{Mod}_k^\pp}(F(A),G(A)) \to \text{Mor}^s_{\text{Mod}^\pp_k}(F(A),G(A))$.
    \end{proof}
\end{prop}
There exist canonical shift functors on $\text{PSh}_\pp(\cc)$, given by
    \begin{equation}\ss^a F (A) = F ( \ss^{-a}A)\end{equation}
The flipping of the sign will become apparent shortly when we look at a persistent Yoneda embedding. They are chosen so that $\yy$ commutes with $\ss^r$s. The maps $(\eta_{a,b})_F: \ss^aF \to \ss^bF$ are defined by

\begin{align}\label{etaandfunctors}
    ((\eta_{a,b})_F)_A= F((\eta_{-a,-b})_A).
\end{align}

There is a natural fully faithful embedding of any persistence category $\cc$ into $\text{PSh}_\pp(\cc)$, it is given by a persistent enriched version of the Yoneda embedding which follows from the enriched Yoneda Lemma using a monoidal structure on $\text{Mod}^\pp_k$ (see \cite{BM} for discussions on this monoidal structure and \cite{Ke} for the enriched Yoneda lemma) 
\begin{equation}\mathcal{Y}:\cc \hookrightarrow \text{PSh}_\pp(\cc)\end{equation}
\begin{equation}\mathcal{Y}(A) = \text{Hom}_{\cc}(-,A). \end{equation}

Given a morphism $f\in \text{Mor}^r_\cc(A,B)$ we have $\yy(f)\in \text{Mor}^r_{\text{PSh}_\pp(\cc)}(\text{Hom}_\cc(-,A),\text{Hom}_\cc(-,B))$ and is given by
\begin{equation}\yy(f)(C):g\mapsto f \circ g\end{equation}
for any $g\in \text{Hom}_\cc(C,A)$. Note that composition is in $\cc$ and hence respects the persistence structures.
\begin{rmk}
    The persistence Yoneda map is a persistence functor. We easily verify that for any $f\in \text{Mor}^r_\cc(A,B)$ we have that $\yy(i_{r,s}\circ f)= i_{r,s} \circ \yy(f)$. Indeed, for $f\in \text{Mor}^r_\cc(A,B)$ and $X\in \cc$ we have $\yy( i_{r,s}\circ f)(X)$ is the map

    \begin{align*}
        \text{Hom}_\cc(X,A) &\xrightarrow{\yy(i_{r,s}\circ f)(X)} \text{Hom}_\cc(X,B)\\
        (-) &\mapsto i_{r,s} \circ f \circ (-) 
    \end{align*}
which is exactly $i_{r,s}\circ \yy(f)(X)$, hence $i_{r,s}\circ \yy(f) = \yy(i_{r,s}\circ f)$.
\end{rmk}
\begin{prop}
    The persistent Yoneda embedding $\yy: \cc \to \text{PSh}_\pp (\cc)$ commutes with shift functors i.e., $\ss^{-r}\circ \yy \cong \yy \circ \ss^{-r}$. Moreover we have that $\yy(\eta_r^A)=\eta_r^{\yy(A)}$.
\begin{proof}
Recall that (see \cite{BCZ} remark 2.16) there exist isomorphisms
\begin{align}
    \text{Mor}^\alpha(X,Y) &\xrightarrow{(-)\circ (\eta_{0,r})_{Y}} \text{Mor}^{\alpha+r}(X,\ss^r Y)\\
        \text{Mor}^\alpha(X,Y) &\xrightarrow{(\eta_{0,s})_X\circ (-)} \text{Mor}^{\alpha-s}(\ss^sX,Y)
\end{align}
Taking $r=s$ and composing we get a canonical isomorphism 

\begin{equation}
(\eta_{0,r})_Y\circ  (-)\circ (\eta_{0,r})_{X}: \text{Mor}^\alpha(X,Y)    \xrightarrow{(-)\circ (\eta_{0,r})_{Y}} \text{Mor}^{\alpha+r}(X,\ss^r Y) \xrightarrow{(\eta_{0,r})_X\circ (-)} \text{Mor}^{\alpha}(\ss^rX,\ss^rY)
\end{equation}

Now consider $\yy(\ss^{-r}A)=\text{Hom}_\cc(-,\ss^{-r}A)$, the above isomorphism gives that for any $X\in \cc$ and $\alpha\in \R$ we have a canonical isomorphism 
\begin{equation}
    \yy(\ss^{-r}A)(X)(\alpha) \to \yy(A)(\ss^rX)(\alpha)
\end{equation}
furthermore these isomorphisms respect the persistence module structures. By definition of the shift functors $\ss^r$ on $\text{PSh}_\pp(\cc)$ the above can be viewed as a canonical isomorphism 
\begin{equation}
      \yy(\ss^{-r}A)(X)(\alpha) \to \ss^{-r}\yy(A)(X)(\alpha)
\end{equation}
for every $\alpha$ and $X$, hence we have $\yy\circ \ss^{a}= \ss^{a}\circ \yy$. Now consider $\yy(\eta_r^A):\yy(A) \to \yy(\ss^{-r}A)$, we have
\begin{align*}
    \yy(\eta_r^A)&=\yy(i_{-r,0} \circ (\eta_{0,-r})_A)\\
    &=i_{-r,0}\circ \yy((\eta_{0,-r})_A)
\end{align*}
The map $\yy((\eta_{0,-r})_A)(X)$ is given by
\begin{align*}
    \text{Hom}_\cc(X,A) &\xrightarrow{\yy((\eta_{0,-r})_A)} \text{Hom}_\cc(X,\ss^{-r}A)\\
    (-)&\mapsto (\eta_{0,-r})_A \circ (-)
\end{align*}
We are realising $\text{Hom}_\cc(X,\ss^{-r}A)$ as $\ss^{-r}\text{Hom}_\cc(X,A)$ via the isomorphism (as described above)

\begin{equation}
    (\eta_{0,r})_{\ss^{-r}A} \circ (-) \circ (\eta_{0,r})_{X}:\text{Hom}_\cc(X,\ss^{-r}A) \to \text{Hom}_\cc(\ss^{r}X,A).
\end{equation}
But $(\eta_{0,r})_{\ss^{-r}A}=(\eta_{-r,0})_A$, hence the composition
\begin{equation}
      \text{Hom}_\cc(X,A) \xrightarrow{\yy((\eta_{0,-r})_A)} \text{Hom}_\cc(X,\ss^{-r}A) \xrightarrow{} \text{Hom}_\cc(\ss^rX,A)
\end{equation}
gives 
\begin{align*}
     (\eta_{0,r})_{\ss^{-r}A} \circ ((\eta_{0,-r})_A\circ (-)) \circ (\eta_{0,r})_{X}&= (\eta_{-r,0})_{A} \circ (\eta_{0,-r})_A\circ (-) \circ (\eta_{0,r})_{X}\\
     &=(-) \circ (\eta_{0,r})_{X}
\end{align*}

Finally recall equation \ref{etaandfunctors} and the definition of $\ss^a$ on $\text{PSh}_\pp(\cc)$, we see
\begin{equation}
    (\eta_{0,-r})_{\yy(A)}(X):\yy(A)(X) \to \ss^{-r}\yy(A)(X) 
\end{equation}
is given by 
\begin{equation}
    \yy(A)((\eta_{0,r})_X:X \to \ss^rX)
\end{equation}
which is precisely our map $(-) \mapsto (-)\circ (\eta_{0,r})_X$. After applying $i_{-r,0}$ we find 
\begin{equation}
\yy(\eta_r^A)=\eta_r^{\yy(A)}:\yy(A) \to \ss^{-r}\yy(A)
\end{equation}
\end{proof}

\end{prop}
\begin{rmk}
In the above proposition and throughout this paper we identify $\yy \circ \ss^{a} $ with $\ss^{a} \circ \yy$. The are not identically equal but can be identified by canonical isomorphism. One can view this as including this isomorphism into the definition of the Yoneda map.
\end{rmk}

\subsection{Weighted retracts}
Completion of a category to an idempotent complete category can be thought of as simply adding enough retracts of objects to the original category. Now consider a persistence category $\cc$, we are trying to construct a new persistence category $\cc'$ such that $\cc'_\infty \cong \text{Split}(\cc_\infty)$. Assume we have a retract in $\cc_\infty$, then there is no need for it to be the image of a retract in $\cc_0$. This leads to the following definition. 
\begin{defn}
    We define a \textbf{weighted retract} of weight $r$ or simply an \textbf{$r$-retract} of an object $A$, to be an object $B$, such that there exist morphisms $s\in \text{Mor}^0_\cc(B,A)$, $r\in \text{Mor}^0_\cc(A,\ss^{-r}B)$  with $r \circ s = \eta_r^B$
    \begin{equation}\begin{tikzcd}
        A\arrow[dr,"r"]\\
         B \ar[r,swap,"\eta_r^B"]\arrow[u,"s"]& \ss^{-r}B\\
    \end{tikzcd}\end{equation}
    If $B$ is an $r$-retract of $A$ we will write $B<_r A$.
\end{defn}

\begin{rmk}
This definition gives an extension of the usual definition of a retract to account for the extra data in persistence categories. In particular $0$-retracts are usual retracts and a weighted retract will become a retract in the usual sense once we pass to the limit category $\cc_\infty$.
\end{rmk}

\begin{prop}
    If $B<_rA $ and $C<_s B$ then $C<_{r+s}A$.

    \begin{proof}
       Assume the respective retracts give rise to the following diagram
        \begin{equation}\begin{tikzcd} 
        A \ar[dr,"r"]& \\
B\ar[r,"\eta_r^B"]\ar[u,"s"]\arrow[dr,"r'"] & \ss^{-r}B\\
         C \ar[r,swap,"\eta_s^C"]\arrow[u,"s'"]& \ss^{-s}C \\
    \end{tikzcd}\end{equation}
We can then complete this to 
    \begin{equation}\begin{tikzcd} 
        A \ar[dr,"r"]& \\
    B\ar[r,"\eta_r^B"]\ar[u,"s"]\arrow[dr,"r'"] & \ss^{-r}B\ar[dr,"\ss^{-r}r'"]\\
         C \ar[r,swap,"\eta_s^C"]\arrow[u,"s'"]& \ss^{-s}C \ar[r,swap,"\eta_r^{\ss^{-s}C}"]& \ss^{-(r+s)}C\\
    \end{tikzcd}\end{equation}
noting that $\ss^{-r}r' \circ \eta_r^B = \eta_r^{\ss^{-s}C}\circ r'$, we have that 
diagram 
        \begin{equation}\begin{tikzcd}
        A\arrow[dr,"\ss^{-r}r' \circ r"]\\
         C\ar[r,swap,"\eta_{r+s}^C"]\arrow[u,"s\circ s'"]& \ss^{-(r+s)}C\\
         \end{tikzcd}\end{equation}
         is an $(r+s)$-retract diagram.
    \end{proof}
    
\end{prop}
\begin{rmk}
    Any object can be viewed as an $r$-retract of itself via
     \begin{equation}\begin{tikzcd}
        A\arrow[dr,"\eta_r^A"]\\
         A \ar[r,swap,"\eta_r^A"]\arrow[u,equals]& \ss^{-r}A\\
    \end{tikzcd}\end{equation}
   
\end{rmk}

   Recall that if $\cc$ is a TPC and if $f:A \to B$ is an $r$-isomorphism, i.e., there exists an exact triangle in $\cc_0$, $A \xrightarrow{f} B \to K \to TA$, with $K\simeq_r 0$. Then there exists a left and right inverse of $f$; $\varphi$ and $ \phi$ respectively. These are defined by commutative diagrams 
    \begin{equation}\begin{tikzcd}
        B \ar[dr,"\phi"]& &  &A\ar[dr,"f"]&  \\
        A\ar[u,"f"]\ar[r,"\eta_r"] & \ss^{-r}A & & \ss^rB\ar[r,"\eta_r"]\ar[u,"\varphi"]&B
    \end{tikzcd}\end{equation}
    In particular we see that $A<_rB$ and $\ss^rB <_r A$.

\begin{prop}
    Let $s:B \to A$ and $r:A \to \ss^{-r}A$ be morphisms in $\cc_0$ such that $r \circ s = \eta_r^B$ and $\ss^{-r}s \circ r =\eta_r^A$, then $s$ and $r$ are $2r$-isomorphisms.
\begin{proof}
    The proof is a slight modification of the proof of Lemma 2.83 given in \cite{BCZ}. Consider the commutative diagram
    \begin{equation}\begin{tikzcd}
        B \ar[r,"s"]\ar[d,"\eta_r^B"]& A\ar[r]\ar[d,"r"] &\ar[d] K\ar[r] & TB\ar[d,"\eta_r^{TB}"]\\
        \ss^{-r}B \ar[r,equals]\ar[d,equals]& \ss^{-r}B\ar[d,"\ss^{-r}s"] \ar[r]& 0\ar[r]\ar[d] & T\ss^{-r}B\ar[d,equals]\\
        \ss^{-r}B\ar[r,"\ss^{-r}s"] & \ss^{-r}A\ar[r] & \ss^{-r}K\ar[r] & T\ss^{-r}B
    \end{tikzcd}\end{equation}
    where each row is an exact triangle. The first and second columns are $r$-isomorphisms, and hence there exists a third column such that the third column is a $2r$-isomorphism. In particular the map $0:K \to \ss^{-r}K$ is a $2r$-isomorphism. We therefore have that $\eta_{2r}^K=0$, giving the result that $s$ is a $2r$-isomorphism.
 The diagram 
    \begin{equation}\begin{tikzcd}
        A \ar[r,"r"]\ar[d,"\eta_r^A"]& \ss^{-r}B\ar[r]\ar[d,"\ss^{-r}s"] &\ar[d] K'\ar[r] & TA\ar[d,"\eta_r^{TA}"]\\
        \ss^{-r}A \ar[r,equals]\ar[d,equals]& \ss^{-r}A\ar[d,"\ss^{-r}r"] \ar[r]& 0\ar[r]\ar[d] & T\ss^{-r}A\ar[d,equals]\\
        \ss^{-r}A\ar[r,"\ss^{-r}r"] & \ss^{-2r}B\ar[r] & \ss^{-r}K'\ar[r] & T\ss^{-r}A
    \end{tikzcd}\end{equation}
    then shows that $r$ is a $2r$-isomorphism. 
\end{proof} 
\end{prop}

We would like to work within the most general setting of persistence categories, in particular not assume that $\cc$ is not a TPC. Hence we can no longer use the definition of $r$-isomorphism given by \cite{BCZ}. From the above proposition is makes sense to say a morphism $f:A \to B$ in $\cc_0$ is a \textbf{strong $r$-isomorphism} iff there exists a $g:B \to \ss^{-r}A$ such that 
\begin{align*}
    g \circ f &= \eta_{r}^{A}\\
\ss^{-r}f \circ g &= \eta_{r}^B.
\end{align*}

\begin{cor}
    If $A<_r B$ and $\ss^rB<_r A$ then $A$ is strongly $r$-isomorphic to $B$.
\end{cor}

\begin{rmk}
    One can show that if $\cc$ is a TPC, then any strong weighted isomorphism is also a weighted isomorphism, and vice versa. However the weights are not equal. In fact, the best one can show is that a strong weighted isomorphism of weight $r$ is a $2r$-weighted isomorphism, and an $r$-weighted isomorphism is a strong $2r$-weighted isomorphism.
\end{rmk}

\subsection{Weighted idempotents}
Similarly to retracts, there is no need for an idempotent in the limit category $\cc_\infty$ of a persistence category $\cc$, to be the image of an idempotent in $\cc_0$. This means we similarly have to consider a larger class of `weighted idempotents'.
\begin{defn}
    Let $A\in \cc$ be an object of a persistence category. We call a morphism $e\in \text{Mor}^0_\cc(A, \ss^{-r}A)$ a \textbf{weighted idempotent of weight $r$} or simply an \textbf{r-idempotent} if
    \begin{equation}\ss^{-r}e \circ e= \eta_r^{\ss^{-r}A} \circ e.\end{equation}
    Equivalently $e: A \to \ss^{-r}A$ is an $r$-idempotent if
    \begin{equation}\ss^{-r}e \circ e = \ss^{-r}e \circ \eta_r^A.\end{equation}
\end{defn}
This definition extends that of an idempotent in the usual sense, i.e., a $0$-idempotent is a morphism $e\in \text{Mor}^0(A,A)$ such that $\ss^{0}e \circ e = \eta_0^A \circ e$, recalling that $\eta_0^A=1_A$ and $\ss^{0}A=A$ we have $e \circ e = e$. A weighted idempotent will produce a usual idempotent in the limit category $\cc_\infty$.
\begin{rmk}
    The maps $\eta_r^A$ are $r$-idempotents.
\end{rmk}
\begin{prop}
    If $e:A \to \ss^{-r}A$ is an $r$-idempotent then $\eta_r^A-e: A \to \ss^{-r}A$ is also an $r$-idempotent. 
    \begin{proof}
    By explicit calculation we find
        \begin{align*}
            \ss^{-r}(\eta_r^A-e) \circ (\eta_r^A \circ e)=& \eta_r^{\ss^{-r}A} \circ \eta_r^A - \eta_r^{\ss^{-r}A}\circ e - \ss^{-r}e \circ \eta_r^A + \ss^{-r}e \circ e\\
            =&\eta_r^{\ss^{-r}A} \circ \eta_r^A -\eta_r^{\ss^{-r}A} \circ e -\eta_r^{\ss^{-r}A} \circ e + \eta_r^{\ss^{-r}A} \circ e\\
            =&\eta_r^{\ss^{-r}A} \circ \eta_r^A - \eta_r^{\ss^{-r}A} \circ e\\
            =& \eta_r^{\ss^{-r}A}\circ (\eta_r^A-e)
        \end{align*}
    \end{proof}
\end{prop}

\begin{defn}
    We say an $r$-idempotent, $e:A \to \ss^{-r}A$ \textbf{splits} if there exists a triple $(B,s,r)$ with $B\in \text{Obj}(\cc)$, $s\in \text{Mor}^0(B,A)$, and $r\in \text{Mor}^0(A,\ss^{-r}B)$ , such that
    \begin{align}
    r\circ s = \eta_r^B\\
    \ss^{-r}s \circ r = e \notag
    \end{align}
Diagrammatically;
    \begin{equation}\begin{tikzcd}
        A \ar[r,"e"]\ar[rd,"r"]& \ss^{-r}A\\
        B\ar[r,"\eta_r^B"]\ar[u,"s"] & \ss^{-r}B\ar[u,swap,"\ss^{-r}s"]
    \end{tikzcd}\end{equation}
\end{defn}
\begin{rmk}
    The map $\eta_r^A$ is a split $r$-idempotent.
\end{rmk}
This definition again extends the notion of an idempotent splitting to allow for non-zero weighted idempotents. We have that a splitting of a $0$-idempotent is a triple $(B,s,r)$ with $r \circ s = 1_B$ and $s \circ r = e$. We find that splittings of weighted idempotents are not unique up to isomorphism (for non-zero weight), but are only unique up to some weighted isomorphism.

\begin{prop}
    Given an $r$-retract $B<_r A$ defined by 
    \begin{equation}\begin{tikzcd}
        A\arrow[dr,"r"]\\
         B \ar[r,swap,"\eta_r^B"]\arrow[u,"s"]& \ss^{-r}B\\
    \end{tikzcd}\end{equation}
     it induces a split $r$-idempotent $ (\ss^{-r}s \circ r):A \to \ss^{-r}$.
    \begin{proof}
        Set $e= \ss^{-r}s \circ r$, we have a diagram 
        \begin{equation}\begin{tikzcd}
        A\arrow[dr,"r"] \ar[r,dotted,"e"] & \ss^{-r}A\\
         B \ar[r,swap,"\eta_r^B"]\arrow[u,"s"]& \ss^{-r}B\ar[u,swap,"\ss^{-r}s"]\\
    \end{tikzcd}\end{equation}
    and can check that $e$ is in fact an $r$-idempotent morphism 
    \begin{align*}
            \ss^{-r}e \circ e &= \ss^{-r}(\ss^{-r}s \circ r) \circ (\ss^{-r}s \circ r)\\
            &=\ss^{-2r}s \circ \ss^{-r}r \circ \ss^{-r}s \circ r\\
            &=\ss^{-2r}s \circ \ss^{-r}(r \circ s) \circ r\\
            &=\ss^{-2r}s \circ \ss^{-r}\eta^A_r \circ r\\
            &=\ss^{-2r}s \circ \ss^{-r}r \circ \eta_r^B\\
            &=\ss^{-r}(\ss^{-r}s \circ r) \circ \eta_r^B\\
            &=\ss^{-r}e \circ \eta_r^B.
        \end{align*}

    \end{proof}
    \end{prop}

\begin{cor}
    Assume $B<_r A$ is a weighted retract that splits the idempotent $\eta_r^A$, then we have $B \simeq_{r} A$.
    \begin{proof}
        Assume the retraction is given by 
        \begin{equation}\begin{tikzcd}
        A\arrow[dr,"r"] \ar[r,dotted,"\eta_r^A"] & \ss^{-r}A\\
         B \ar[r,swap,"\eta_r^B"]\arrow[u,"s"]& \ss^{-r}B\ar[u,swap,"\ss^{-r}s"]\\
    \end{tikzcd}\end{equation}
    then $r \circ s=\eta_r^B$ and $\ss^{-r}s \circ r=\eta_r^A$, hence $s$ is the desired strong $r$-isomorphism.
    \end{proof}
\end{cor}
There is an obvious relation between weighted idempotent splittings and weighted retracts. If $B<_r A$ then we obtain a split $r$-idempotent on $A$, and if an $r$-idempotent on $A$ splits we get a canonical retract. Note however that in general not every $r$-idempotent splits. One should also be aware that if a weighted retract $B<_r A$ defines an idempotent $e$ on $A$ then this could split via a different retract $B'<_r A$ with $B' \ncong B$ in $\cc_0$. This is a consequence of the following
\begin{lemma}[Uniqueness of $r$-idempotent splittings]

Let $e: A \to \ss^{-r}A$ be an idempotent, then if $e$ splits, this splitting is unique up to strong $2r$-isomorphism.
\begin{proof}
    Assume there exist two splittings of $e$:

     \begin{equation}\begin{tikzcd}
        A \ar[r,"e"]\ar[rd,"r"]& \ss^{-r}A& & A \ar[r,"e"]\ar[dr,"r'"]& \ss^{-r}A\\
        B\ar[r,"\eta_r^B"]\ar[u,"s"] & \ss^{-r}B\ar[u,swap,"\ss^{-r }s"] & & B' \ar[r,"\eta_r^{B'}"]\ar[u,"s'"]& \ss^{-r}B'\ar[u,swap,"\ss^{-r}s'"]
    \end{tikzcd}\end{equation}
Then we have a pair of morphisms $\alpha:= r' \circ s$ and $\beta:= r \circ s'$. We find that
\begin{align*}
    \ss^{-r}\alpha \circ \beta =& \ss^{-r}r' \circ \ss^{-r}s \circ r \circ s'\\
    =& \ss^{-r}r' \circ e \circ s'\\
    =& \eta_r^{\ss^{-r}B'} \circ r' \circ s'\\
    =& \eta_r^{\ss^{-r}B'}  \circ \eta_r^{B'}\\
    =& \eta_{2r}^{B'}
\end{align*}
    A similar argument shows $\ss^{-r}\beta \circ  \alpha = \eta_{2r}^{B}$.
\end{proof}
\end{lemma}

\begin{prop}
    Assume that an $r$-idempotent, $e:A \to \ss^{-r}A$ splits via
     \begin{equation}\begin{tikzcd}
        A \ar[r,"e"]\ar[rd,"r"]& \ss^{-r}A\\
        B\ar[r,"\eta_r^B"]\ar[u,"s"] & \ss^{-r}B\ar[u,swap,"\ss^{-r}s"]
    \end{tikzcd}\end{equation}
    if $f,g\in \text{Hom}(\ss^{-r}B,C)$ are morphisms such that $f \circ r= g \circ r$ then $f \simeq_r g$, i.e., $(f-g)\circ \eta^{B}_r =0$.
\begin{proof}
    The proof is simple, if $f \circ r = g\circ r$ then 
    \begin{align*}
        &(f-g)\circ r =0\\
        \implies &(f-g) \circ r \circ s = 0\\
        \implies & (f-g)\circ \eta_r^B =0
    \end{align*}
\end{proof} 
\end{prop}
\begin{rmk}
    A similar argument shows that if $s \circ f = s \circ g$ then $f\simeq_r g$.
\end{rmk}
\begin{prop}
If there exists an equaliser (in $\cc_0$) to the diagram:
    \begin{equation}\begin{tikzcd}
        A\arrow[r,yshift=1ex,"e"] \arrow[r,yshift=-1ex,swap,"\eta_r^A"]& \ss^{-r}A
    \end{tikzcd}\end{equation}
    where $e$ is an r-idempotent, then $e$ splits.
    
    \begin{proof}
    Let $(B,s)$ be the equaliser of 
           \begin{equation}\begin{tikzcd}
        A\arrow[r,yshift=1ex,"e"] \arrow[r,yshift=-1ex,swap,"\eta_r^A"]& \ss^{-r}A
    \end{tikzcd}\end{equation}
    so that we have a commutative diagram
       \begin{equation}\begin{tikzcd}
       B \arrow[r,"s"]& A\arrow[r,yshift=1ex,"e"] \arrow[r,yshift=-1ex,swap,"\eta_r^A"]& \ss^{-r}A
    \end{tikzcd}\end{equation}
    Then by the universal property of equalisers and the fact that $e$ is an $r$-idempotent, we have that there exists a unique $i: \ss^{r}A \to B$ making the following commute
       \begin{equation}\begin{tikzcd}
       B \arrow[r,"s"]& A\arrow[r,yshift=1ex,"e"] \arrow[r,yshift=-1ex,swap,"\eta_r^A"]& \ss^{-r}A \\
        & \ss^r A \arrow[u,"\ss^r e"]\arrow[lu,dotted,"\exists ! i"]& 
    \end{tikzcd}\end{equation}
    Hence we have $\ss^r e = s \circ i$ and so 
    \begin{equation}e = \ss^{-r}s \circ \ss^{-r}i.\end{equation}
    It then follows that
    \begin{align*}
        e \circ s =& \eta_r^A \circ s\\
        e \circ s =& \ss^{-r}s \circ \eta_r^B\\
        (\ss^{-r}s \circ \ss^{-r}i )\circ s =& \ss^{-r}s \circ \eta_r^B\\
      \ss^{-r}s \circ (\ss^{-r}i \circ s)  =&\ss^{-r}s \circ \eta_r^B.
    \end{align*}
   Because $(B,s)$ is an equaliser, $s$ is a monomorphism, and so it follows that
    \begin{equation}\eta_r^B= \ss^{-r}i \circ s.\end{equation}

    \end{proof}
\end{prop}

\begin{prop}
    If there exists a co-equaliser (in $\cc_0$) to the diagram 
    \begin{equation}\begin{tikzcd}
        A\arrow[r,yshift=1ex,"e"] \arrow[r,yshift=-1ex,swap,"\eta_r^A"]& \ss^{-r}A
    \end{tikzcd}\end{equation}
    where $e$ is an $r$-idempotent, then $e$ splits.
    \begin{proof}
    The proof follows similarly to the proof for equalisers. By assumption there exists a unique $u:C \to \ss^{-2r}A$ a universal diagram
    \begin{equation}\begin{tikzcd}
        A\arrow[r,yshift=1ex,"e"] \arrow[r,yshift=-1ex,swap,"\eta_r^A"]& \ss^{-r}A\arrow[r,"t"]\arrow[d,"\ss^{-r}e"]& C\arrow[dl,dotted,"\exists !u"]\\
        & \ss^{-2r}A & 
    \end{tikzcd}\end{equation}

    giving 
    \begin{equation}e=\ss^{r}u \circ \ss^r t.\end{equation}
    It then follows that 
    \begin{align*}
        t \circ e =& t \circ \eta_r^A\\
       t \circ (\ss^{r}u \circ \ss^r t) =& t \circ \eta_r^A\\
      (t \circ \ss^r u ) \circ \ss^r t =&\eta_r^{\ss^rC} \circ \ss^r t.
    \end{align*}
    $t$ is an epimorphism and hence we have
   \begin{align*}
       t \circ \ss^r u = &\eta_r^{\ss^{r}C}\\
       =&\ss^r \eta_r^C
   \end{align*}
   giving the result
   \begin{equation}\ss^{-r}t \circ u = \eta_r^C.\end{equation}
    \end{proof}
\end{prop}

Note that for $r>0$, an $r$-idempotent splitting does not in general imply the existence of a corresponding equaliser or coequaliser. Indeed as splittings are not unique up to isomorphism, a splitting cannot give a universal objects and morphism equalising/coequalising 
\begin{equation}\begin{tikzcd}
        A\arrow[r,yshift=1ex,"e"] \arrow[r,yshift=-1ex,swap,"\eta_r^A"]& \ss^{-r}A
    \end{tikzcd}\end{equation}

\begin{prop}
    Let $e:A\to \ss^{-r}A$ be a split $r$-idempotent with splitting given by
    \begin{equation}\begin{tikzcd}
        A \ar[r,"e"]\ar[rd,"r"]& \ss^{-r}A\\
        B\ar[r,"\eta_r^B"]\ar[u,"s"] & \ss^{-r}B\ar[u,swap,"\ss^{-r}s"]
    \end{tikzcd}\end{equation}
    Then, $\ss^{-r}r:\ss^{-r}A \to \ss^{-2r}B$ coequalises 
 \begin{equation}\begin{tikzcd}
        A\arrow[r,yshift=1ex,"e"] \arrow[r,yshift=-1ex,swap,"\eta_r^A"]& \ss^{-r}A
    \end{tikzcd}\end{equation}
and moreover, if $\alpha:\ss^{-r}A \to K$ also coequalises, then there exists a $\beta:\ss^{-2r}B \to \ss^{-r}K$ such that $\beta \circ \ss^{-r}r= \eta_r^K \circ \alpha$.
\begin{proof}
Firstly, we have $\ss^{-r}r \circ e = \ss^{-r}r \circ \ss^{-r}s \circ r= \eta^{\ss^{-r}B}_r \circ r = \ss^{-r}r \circ \eta_r^A$. Now set $\beta:= \ss^{-r}\alpha \circ \ss^{-2r}s$. Then we have
    \begin{align*}
        \beta \circ \ss^{-r}r = &\ss^{-r}\alpha \circ \ss^{-2r}s \circ \ss^{-r}r\\
        =&\ss^{-r}\alpha \circ \ss^{-r}e\\
        =& \ss^{-r}(\alpha \circ e)\\
        =& \ss^{-r}(\alpha \circ \eta^{\ss^{-r}A}_r)\\
        =& \eta_r^K \circ \alpha.
    \end{align*}
    The map $\beta$ is not unique, but unique up to $r$-equivalence. Assume there exists $\beta': \ss^{-2r}B \to \ss^{-r}K $ such that $\beta' \circ \ss^{-r}r = \eta_r^K \circ \alpha$, then it follows:

    \begin{align}
        &(\beta ' - \beta) \circ \ss^{-r}r = 0\\
        \implies& \beta \simeq_r \beta ' \notag
    \end{align}

\end{proof}

\end{prop}

\begin{prop}
  With $e:A\to \ss^{-r}A$ split as above, the map $s:B \to A$ equalises 
   \begin{equation}\begin{tikzcd}
        A\arrow[r,yshift=1ex,"e"] \arrow[r,yshift=-1ex,swap,"\eta_r^A"]& \ss^{-r}A
    \end{tikzcd}\end{equation}
    and if $\alpha:C \to A$ also equalises then there exists a $\beta:\ss^{r}C \to B$ unique up to $r$-equivalence such that $s \circ \beta = \alpha \circ \eta_r^{\ss^r C}$.
    \begin{proof}
        Firstly, $e \circ s = \ss^{-r}s \circ \eta_r^B = \eta_r^A \circ s$. Secondly, set $\beta:= \ss^{r}r \circ \ss^r \alpha$, then 
        \begin{align*}
            s\circ \beta = &s \circ \ss^r r \circ \ss^r \alpha \\
            =&\ss^r e \circ \ss^r \alpha \\
            =& \eta_r^{\ss^r A} \circ \ss^r \alpha\\
            =& \alpha \circ \eta_r^{\ss^r C}
        \end{align*}

        Finally, assume there exists $\beta':\ss^r C \to B$ with $s \circ \beta' = \alpha \circ \eta_r^{\ss^rC}$, then,
        \begin{align*}
            &s \circ (\beta - \beta')=0\\
            \implies & \beta \simeq_r \beta'
        \end{align*}
    \end{proof}
\end{prop}

\begin{lemma}
    Assume that both $e:A\to \ss^{-r}A$ and $\eta_r^A-e:A\to \ss^{-r}A$ are split $r$-idempotents, given by respective diagrams
    \begin{equation}\begin{tikzcd}
        A \ar[dr,"r_B"]\ar[r,"e_B"]& \ss^{-r}A& & A\ar[dr,"r_C"]\ar[r,"e_C"] &\ss^{-r}A \\
        B\ar[u,"s_B"]\ar[r,"\eta_r^B"] & \ss^{-r}B\ar[u,swap,"\ss^{-r}s_B"] & &C\ar[u,"s_C"]\ar[r,"\eta_r^C"] & \ss^{-r}C\ar[u,swap,"\ss^{-r}s_C"]
    \end{tikzcd}\end{equation}
    with $e_B:=e$ and $e_C:=\eta_r^A-e$. Then  $B \oplus C\simeq_{2r}A$.

    \begin{proof}
        Let $f:A \to \ss^{-r}B\oplus \ss^{-r}C$ be the map given by $f=\begin{pmatrix}
            r_B \\
            r_C
        \end{pmatrix}$ and $g:B \oplus C \to A$ by $g=\begin{pmatrix}
            s_B &
            s_C
        \end{pmatrix}$. We then have 

        \begin{equation}\ss^{-r}g \circ f = \begin{pmatrix}
            r_B \\
            r_C
        \end{pmatrix}\begin{pmatrix}
            \ss^{-r}s_B &
            \ss^{-r}s_C
        \end{pmatrix} = \ss^{-r}s_B \circ r_B + \ss^{-r}s_C + r_C = e_B + e_C = \eta_r^A  \end{equation}

    and

    \begin{equation}f \circ g = \begin{pmatrix}
            s_B &
            s_C
        \end{pmatrix}  \begin{pmatrix}
            r_B \\
            r_C
        \end{pmatrix} = \begin{pmatrix}
            r_B \circ s_B & r_B \circ s_C\\
            r_C \circ s_B & r_C \circ s_C
        \end{pmatrix}=\begin{pmatrix}
            \eta_r^B& r_B \circ s_C\\
            r_C \circ s_B & \eta_r^C
        \end{pmatrix} \end{equation}

        One can check; 
        \begin{align*}
       r_B \circ s_C \circ \ss^r r_C =& r_B \circ \ss^r e_C \\
            =&r_B \circ (\eta_r^{\ss^{r}A} - \ss^r e)\\
            =&r_B \circ \eta_r^{\ss^{r}A} - r_B \circ \ss^r e\\
            =&\eta_r^B \circ \ss^r r_B - \eta_r^B \circ \ss^r r_B\\
            =&0
        \end{align*}
        hence we have $r_B \circ s_C \simeq_r 0$. Now consider $(\eta_r^{\ss^{-r}(B \oplus C)}\circ f) \circ g$, given by
        \begin{align*}
            \begin{pmatrix}
                \eta_r^{\ss^{-r}B} & 0 \\
                0 & \eta_r^{\ss^{-r}C}        
            \end{pmatrix} \circ \begin{pmatrix}
            \eta_r^B& r_B \circ s_C\\
            r_C \circ s_B & \eta_r^C
        \end{pmatrix}=& \begin{pmatrix}
            \eta_r^{\ss^{-r}B} \circ \eta_r^B&\eta_{r}^{\ss^{-r}B} \circ r_B \circ s_C\\
           \eta_r^{\ss^{-r}C} \circ r_C \circ s_B & \eta_r^{\ss^{-r}C}\circ \eta_r^C
        \end{pmatrix}\\
        =&\begin{pmatrix}
            \eta_{2r}^B& 0\\
            0 & \eta_{2r}^C
        \end{pmatrix}\\
        =&\eta_{2r}^{B \oplus C}
        \end{align*}
We also have, 
\begin{equation}\ss^{-2r}g \circ (\eta_r^{\ss^{-r}(B \oplus C)}\circ f) =\eta_{r}^{\ss^{-r}A} \circ \ss^{-r}g \circ f =\eta_{2r}^A \end{equation}
        hence we have $g:B\oplus C \simeq_{2r}A$. 
    \end{proof}
    
\end{lemma}

\subsection{Weighted idempotent completion}
We now study what idempotent completion in the persistence setting should look like. Similarly to the construction of idempotent completion in the setting of semi-persistence categories we realise that one needs to not just split idempotents but split all weighted idempotents to ensure that the limit category recovers the usual idempotent completion. We will call a persistence category \textbf{weighted idempotent complete} if for any $r\geq 0$, every $r$-idempotent has a splitting, note that this splitting does not need to be unique. 

\begin{lemma}\label{pshidemcomp}
    The category $\text{PSh}_\pp(\cc)$ is weighted idempotent complete. 
    \begin{proof}
        We show that given any $r$-idempotent $e: F \to \ss^{-r}F$ there exists an equaliser to the diagram 
     \begin{equation}\begin{tikzcd}
        F\arrow[r,yshift=1ex,"e"] \arrow[r,yshift=-1ex,swap,"\eta_r^F"]& \ss^{-r}F
    \end{tikzcd}\end{equation}
    and hence a splitting of $e$. Note that if there exists a kernel of $e-\eta_r^F$ then the equaliser exists and so in fact it is enough to show that $\text{Mod}^\pp_k$ has kernels, as we this will imply $\text{PSh}_\pp(\cc)$ also kernels by taking them pointwise in $\text{Mod}^\pp_k$. Let $f\in \text{Hom}_{\text{Mod}_k^\pp}(V,W)(r)$ be a morphism of persistence modules with shift $r$. We have 
    \begin{equation}\text{Ker}\big(f\big)(s)=\text{Ker}(f^s)\end{equation}
  where $f^s:V(s) \to W(s+r)$ is the underlying morphism of $k$-modules. We then must check that there exists a natural persistence module structure on $\text{Ker}(f)$, i.e., there are maps
   \begin{equation}\text{Ker}(f)(t) \xrightarrow{\pi_{t,s}}\text{Ker}(f)(s)\end{equation}
   for all $t\leq s$. Consider the diagram, in which the right hand square commutes, as $f$ is a persistence module morphism:
   \begin{equation}\begin{tikzcd}
       \text{Ker}(f)(t) \arrow[r,"j(t)"]& V(t)\arrow[r,"f^t"]\arrow[d,"i_{t,s}^{V}"] & W(t+r)\arrow[d,"i_{r+t,r+s}^{W}"]\\
        \text{Ker}(f)(s)\arrow[r,"j(s)"] & V(s)\arrow[r,"f^s"] & W(s+r)
   \end{tikzcd}\end{equation}
   By the universal property of the kernel, we have a diagram 
   \begin{equation}\begin{tikzcd}
       &  \text{Ker}(f)(t)\arrow[d,"i_{t,s}^{V} \circ j(t)"]\arrow[dd,bend  left = 90, "0"]\arrow[dl,dotted,swap,"\exists !"]\\
     \text{Ker}(f)(s) \arrow[r,"j(s)"]\arrow[rd,"0"]& V(s)\arrow[d,"f^s"]\\
       & W(s+r)
   \end{tikzcd}\end{equation}
   We set $\pi_{t,s}$ to be this unique map. It follows from the uniqueness that $\pi_{t,s} \circ \pi_{k,t}=\pi_{k,s}$. Finally we see that the persistence module which we constructed satisfies the universal property of a kernel by the universal property levelwise. Hence $\text{Mod}^\pp_k$ has kernels. It follows that $\text{PSh}_\pp(\cc)$ also has kernels by defining them pointwise, i.e., $\text{Ker}(f:F \to G)(A)=\text{Ker}(f(A):F(A) \to G(A))$. Therefore the diagram
   \begin{equation}\begin{tikzcd}
        F\arrow[r,yshift=1ex,"e"] \arrow[r,yshift=-1ex,swap,"\eta_r^F"]& \ss^{-r}F
    \end{tikzcd}\end{equation}
    has an equaliser and hence the $r$-idempotent has a splitting.
 
    \end{proof}
\end{lemma}
\begin{rmk}
    An idempotent in $\text{PSh}_\pp(\cc)$ does not have a unique splitting, however there is a natural choice of splitting given by the equaliser/kernel.
\end{rmk}

\begin{cor}
    If $e:A \to \ss^{-r}A$ is an $r$-idempotent in $\cc$, then $\yy(e): \yy(A) \to \ss^{-r}\yy(A)$ is an $r$-idempotent in $\text{PSh}_\pp(\cc)$.
    \begin{proof}
        This follows directly from functorality of $\yy$ and from $\yy(\eta_r^A)=\eta_r^{\yy(A)}$.
    \end{proof}
\end{cor}

A \textbf{weighted idempotent completion} of a persistence category $\cc$, that is, a pair $(\cc',\iota:\cc \to \cc')$, where $\cc$ is weighted idempotent complete and $\iota$ is fully faithful, recalling that weighted idempotent complete means that every weighted idempotent splits. There could be many such completions. For example, we could choose $(\cc',\iota)$ to be $(\text{PSh}_\pp(\cc),\yy)$, the category of persistent presheaves along with the Yoneda embedding. We know that $\text{PSh}_\pp(\cc)$ is weighted idempotent complete and that $\yy$ is fully faithful. We can view a weighted idempotent completion as a form of enlargement of $\cc$ by adding in enough 
weighted retracts such that every weighted idempotent can split. Ideally then we should consider the smallest such enlargement one can take. This will correspond to a universal idempotent completion. By this we mean that if $(\hat{\cc},\hat{\iota})$ is this smallest such completion, then for any other completion $(\cc',\iota)$ we have a fully faithful embedding $\kappa:\hat{\cc}\to \cc'$ and a commutative diagram
\begin{equation}\begin{tikzcd}
    \cc \ar[dr,"\iota"]\ar[r,"\hat{\iota}"]& \hat{\cc}\ar[d,dotted, "\kappa"]\\
    & \cc'
\end{tikzcd}\end{equation}

\begin{lemma}\label{splitidemcomp}
Let $\cc$ be a Persistence Category, consider the persistent Yoneda embedding $\yy:\cc \hookrightarrow \text{PSh}_\ff(\cc)$, denote by $\text{Split}_\pp(\cc)$ the full subcategory of $\text{PSh}_\pp(\cc)$ on the objects that are $r$-retracts a representable functor $\yy(A)$ for some $A\in \cc$, $r\geq  0$. Then $\text{Split}_\pp(\cc)$ is weighted idempotent complete.

\begin{proof}
    Let $e: F \to \ss^{-r}F$ be an $r$-idempotent for $r\geq 0$, where $F$ is an $s$-retract of some $\yy(A)$. We know $e$ splits in $\text{PSh}_\ff(\cc)$ and so there exists an object $G$ and a pair of morphisms $s\in \text{Mor}^0(G,F)$, $r\in \text{Mor}^0(F,\ss^{-r}G)$ such that $r \circ s = \eta_r^G$ and $\ss^{-r}s \circ r=e$. It follows that $G$ is an $r$-retract of $F$
    \begin{equation}\begin{tikzcd}
         F\arrow[dr,"\ss^{-r}i"]\\
         G \ar[r,swap,"\eta_r^G"]\arrow[u,"s"]& \ss^{-r}G\\
    \end{tikzcd}\end{equation}
  We then have that $G\in \text{Split}_\pp(\cc)$ (as a $(r+s)$-retract of $\yy(A)$) and since $\text{Split}_\pp(\cc)$ is a full subcategory, we have the morphisms $s$ and $i$ give the splitting of $e$ within $\text{Split}_\pp(\cc)$.
\end{proof}
\end{lemma}

Any two splittings of an $r$-idempotent are unique only up to strong $2r$-isomorphism.  Let $i:\cc \hookrightarrow \cc'$ be another completion of $\cc$, then there is not a canonical persistence functor $\kappa: \text{Split}_\pp(\cc) \to \cc'$ satisfying $\kappa \circ \yy = i$. Assume $i(e:A \to \ss^{-r}A)$ splits as $F' \to i(A) \to \ss^{-r}F'$, and $\yy(e)$ as $F \to \yy(A) \to \ss^{-r}F$ we would hope to define $\kappa(F \to \yy(A) \to \ss^{-r}F) = F' \to i(A) \to F'$, however the choice of splitting given by $F$ is not unique but unique only to a strong $2r$ isomorphism. There are therefore many such persistence functors $\kappa: \text{Split}_\pp(\cc) \to \cc'$ and $\kappa \circ \yy(F)$ is only strongly $2r$-isomorphic to $i(F)$. Note however that $\text{Split}_\pp(\cc)_\infty$ will be equivalent to $\cc'_\infty$. It follows that $\text{Split}_\pp(\cc)$ is not universal in the usual sense, but universal up to a form of `persistent equivalence'. 

\begin{rmk}
    One could attempt to choose a splitting of each weighted idempotent, or perhaps work with equivalence classes of splittings. These constructions seem to not behave well, for example a natural choice of splitting (up to actual isomorphism) is to take equalisers of all representable idempotents in $\text{PSh}_\pp(\cc)$. However, consider an $r$-idempotent $e:F \to \ss^{-r}F$ on some non-representable object $F$. Then by assumption $F=\text{eq}(\yy(\hat{e}),\eta_s^{\yy(A)})$ for some $s$-idempotent $\yy(\hat{e}):\yy(A)\to \ss^{-s}\yy(A)$. One would therefore need the splitting $G$ of $e$ to be given by some equaliser of a representable idempotent $\yy(e'):\yy(A) \to \ss^{-(r+s)}\yy(A) $ and such that we have a diagram

    \begin{equation}
        \begin{tikzcd}
            \yy(A)\ar[dr,"r_F"]& & & \\
            F\ar[dr,"r_G"]\ar[u,"s_F"]\ar[r,"\eta_r"]&\ss^{-s}F\ar[dr,"\ss^{-s}r_G"] & & \\
            G\ar[u,"s_G"]\ar[r,"\eta_r"]& \ss^{-r}G\ar[r,"\eta_s"]& \ss^{-(r+s)}G & 
        \end{tikzcd}
    \end{equation}
 i.e., $\yy(e')=\ss^{-(r+s)}(s_F \circ s_G) \circ \ss^{-s}r_G \circ r_F$. There seems to be no obvious way of defining the idempotent $\yy(e')$ without first taking a splitting of $e:F \to \ss^{-r}F$. Assume therefore that we choose some splitting of $e$, given by $G \xrightarrow{s_G}F \xrightarrow{r_G}\ss^{-r}G$, and then construct the idempotent $\yy(e')=\ss^{-(r+s)}(s_F \circ s_G) \circ \ss^{-s}r_G \circ r_F$. We could then look at `replacing' $G$ with $G'$, the equaliser of $\eta_r^{\yy(A)}$ and $\yy(e')$. 

 \begin{equation}
     \begin{tikzcd}
       G' \arrow[r,"s_{G'}"]& \yy(A)\arrow[r,yshift=1ex,"\yy(e)"] \arrow[r,yshift=-1ex,swap,"\eta_{r+s}^{\yy(A)}"]& \ss^{-(r+s)}\yy(A) \\
       & & \\
        & \ss^{r+s}\yy (A) \arrow[uu,swap,"\ss^{r+s} \yy(e')"]\arrow[luu,dotted,"\ss^{r+s}r_{G'}"]& 
    \end{tikzcd}
 \end{equation}
    But there now is no clear way to write $G'$ as a retraction of $F$ splitting the original idempotent $e$.
\end{rmk}

\subsection{The limit category}
The aim of this section is to construct the equivalence of categories $\zz : \text{Split}(\cc_\infty) \to \text{Split}_\pp(\cc)_\infty$. Before we do so we first need to show the following results:
\begin{lemma}\label{localwrtsisos}
The limit category $\cc_\infty$ of a persistence category $\cc$ is equivalent to the localisation $W^{-1}\cc_0$ of $\cc_0$ with respect to the calculus of fractions $W$ given by strongly weighted isomorphisms. 
\end{lemma}

We proved a more general version of this statement for persistence semi-categories in \ref{sectpcatpair}. Note that Biran, Cornea and Zhang proved that in the setting of TPCs that the limit category is equivalent to the localisation with respect to weighted isomorphisms. This however required that $\cc_0$ is triangulated (in order to define weighted isomorphism) hence this is a generalistion to any persistence category. Furthermore note that if $\cc$ is also a TPC then localisation with respect to strongly weighted isomorphisms is equivalent to localisation with respect to weighted isomorphism in the sense of \cite{BCZ}.
\begin{prop}
    Any idempotent in $\cc_\infty$ can be represented by some weighted idempotent in $\cc_0$.
    \begin{proof}
    Assume that $e:A \to A$ is an idempotent in $\cc_\infty = W^{-1}\cc_0$. Then we may assume there exists a commutative diagram:

\begin{equation}\begin{tikzcd}
    A \ar[d,equals]& \ss^rA \ar[l,swap,"\eta_r"]\ar[rr,"\overline{e}"]& &A\ar[d,equals]\\
    A\ar[d,equals] & B \ar[u,"\phi'"]\ar[rr,"f"]\ar[l,swap,"u"]\ar[d,"\varphi'"]& &A\ar[d,equals]\\
    A & \ss^{2r}A\ar[l,swap,"\eta_{2r}"]\ar[rr,"\overline{e}\circ \overline{e}"] & &A
\end{tikzcd}\end{equation}
where $\overline{e}$ represents $e$ and we assume $u$ is some strong $2s$-isomorphism. There exists a morphism $v:A\to \ss^{-2s}B$ such that $v\circ u = \eta_{2s}$ and $u \circ \ss^{2s}  v=\eta_{2s}$ (note for simplicity we omit the object from the notation of the $\eta$s). We therefore can replace the diagram with
\begin{equation}\begin{tikzcd}
    A \ar[d,equals]& \ss^rA \ar[l,swap,"\eta_r"]\ar[rr,"\overline{e}"]& &A\ar[d,equals]\\
    A\ar[d,equals] & \ss^{2s} A \ar[u,"\phi"]\ar[rr,"f \circ \ss^{2s} v"]\ar[l,swap,"\eta_{2s}"]\ar[d,"\varphi"]& &A\ar[d,equals]\\
    A & \ss^{2r}A\ar[l,swap,"\eta_{2r}"]\ar[rr,"\overline{e}\circ \overline{e}"] & &A
\end{tikzcd}\end{equation}
We then see that $\eta_{2s} = \eta_r \circ \phi$ and $\eta_{2s} = \eta_{2r} \circ \varphi$. It follows that $\phi \simeq_r \eta_{2s-r}$ and $\varphi \simeq_{2r} \eta_{2s-2r}$. Giving, $f \circ \ss^{2s} v \circ \eta_{2r} = \overline{e}\circ \phi \circ \eta_{2r} = \overline{e} \circ \eta_{r+2s}$, and $f \circ \ss^{2s} v \circ \eta_{2r} = \overline{e} \circ \overline{e} \circ \varphi \circ \eta_{2r}= \overline{e} \circ \overline{e} \circ \eta_{2s} $, i.e., the following commutes: 
\begin{equation}\begin{tikzcd}
    A \ar[d,equals]& \ss^{2r}A \ar[l,swap,"\eta_{2r}"]\ar[rr,"\overline{e} \circ \eta_r"]& &A\ar[d,equals]\\
    A\ar[d,equals] & \ss^{2s+2r} A \ar[u,"\eta_{2s}"]\ar[rr,"f \circ \ss^{2s} v \circ \eta_{2r}"]\ar[l,swap,"\eta_{2s+2r}"]\ar[d,"\eta_{2s}"]& &A\ar[d,equals]\\
    A & \ss^{2r}A\ar[l,swap,"\eta_{2r}"]\ar[rr,"\overline{e}\circ \overline{e}"] & &A
\end{tikzcd}\end{equation}

Consider now the map $\overline{e} \circ \eta_s$, it is equivalent in $\cc_\infty$ to $\overline{e}$ via the diagram
\begin{equation}\begin{tikzcd}
    A \ar[d,equals]& \ss^rA \ar[l,swap,"\eta_r"]\ar[rr,"\overline{e}"]& &A\ar[d,equals]\\
    A\ar[d,equals] & \ss^{r+s} A \ar[u,"\eta_s"]\ar[rr,"\overline{e}\circ \eta_s"]\ar[l,swap,"\eta_{r+s}"]\ar[d,equals]& &A\ar[d,equals]\\
    A & \ss^{r+s}A\ar[l,swap,"\eta_{r+s}"]\ar[rr,"\overline{e}\circ \eta_s"] & &A
\end{tikzcd}\end{equation}

and satisfies, from above,
\begin{align*}
(\overline{e}\circ \eta_s) \circ \ss^{(r+s)}(\overline{e}\circ \eta_s) = &    \overline{e}\circ \eta_s \circ \ss^{(r+s)}\overline{e}\circ \eta_s\\
=& \overline{e} \circ \ss^{r}\overline{e} \circ \eta_{2s}\\
=&\overline{e} \circ \eta_r \circ \eta_{2s}\\
=& (\overline{e} \circ \eta_{s})\circ \eta_{r+s}
\end{align*}
therefore $e$ in $\cc_\infty$ is represented by the $(r+s)$-idempotent $\overline{e}\circ \eta_s$.
    \end{proof}
\end{prop}
\begin{cor}
    Any retract in $\cc_\infty$ can be represented as a weighted retract in $\cc_0$.
\end{cor}

\begin{cor}
    The category $\text{Split}_\pp(\cc)_\infty$ is idempotent complete.
    \begin{proof}
          Given an idempotent $e:A \to A$ in $\text{Split}_\pp(\cc)_\infty$, it can be written as the image of some weighted idempotent $\bar{e}:A \to \ss^{-r}A$. This weighted idempotent splits in $\text{Split}_\pp(\cc)$, assume this splitting is given by 
          
          \begin{equation}
              \begin{tikzcd}
                  A \ar[dr,"r"]& \\
                  B \ar[u,"s"]\ar[r,"\eta_r"]& \ss^{-r}B
              \end{tikzcd}
          \end{equation}
          Then we have that the composition of roof diagrams

          \begin{equation}
          \begin{tikzcd}
             & & \ar[dl,swap,"\eta_r"]\ss^rB\ar[dr,"\ss^rs"]\ar[lld,bend right,swap, "\eta_r"]\ar[drr, bend left, "\eta_r"] & &\\
             B & \ar[l,equals]B \ar[r,"s"]& A & \ar[l,swap,"\eta_r"]\ss^rA\ar[r,"\ss^rr"] & B
          \end{tikzcd}
\end{equation}
         is seen to be (equivalent to) the identity on $B$. Furthermore the composition

           \begin{equation}
          \begin{tikzcd}
             & &  \ss^rA\ar[dl,equals]\ar[dr,"\ss^rr"]\ar[dll,bend right, swap,"\eta_r"]\ar[drr,bend left, "\bar{e}"]& &\\
             A & \ar[l,swap,"\eta_r"]\ss^rA \ar[r,"\ss^r r"]& B & \ar[l,swap,equals]B\ar[r,"s"] & A
          \end{tikzcd}
\end{equation}
          defines a splitting of the idempotent $e: A \to A$ in $\text{Split}_\pp(\cc)_\infty$. 
    \end{proof}
\end{cor}
\begin{theorem}\label{inftytri}
    There is an equivalence of categories $\tilde{\zz}:\text{Split}(\cc_\infty) 
 \to \text{Split}_\pp(\cc)_\infty $.
\begin{proof}
Consider the map $\zz:\cc_\infty \to \text{Split}_\pp(\cc)_\infty$, given by $\zz(A)=\hat{\yy}(A)$ on objects $A\in \text{Obj}(\cc_\infty)=\text{Obj}(\cc)$. Recall that $\hat{\yy}(A)(r)=\yy(A)$. On morphisms $\tilde{\zz}$ is defined as follows: let $f \in \text{Hom}_{\cc_\infty}(A,B)$ be represented by $\bar{f}:\ss^rA \to B$ in $\cc_0$, then define $\zz(f) = [\yy(\bar{f})]_{\text{Split}_\pp(\cc)_\infty}$, where $[-]_{\text{Split}_\pp(\cc)_\infty}$ means image in $\text{Split}_\pp(\cc)_\infty$. This functor is well defined: assume $f$ and $f'$ are equivalent representatives of some morphism in $\text{Hom}_{\cc_\infty}(A,B)$ then we have a commutative diagram  

\begin{equation}\begin{tikzcd}
    A\ar[d,equals] & \ss^rA\ar[r,"f"]\ar[l,swap,"\eta_r"] & B\ar[d,equals] \\
 A \ar[d,equals] & C\ar[u,"\phi"]\ar[d,"\varphi"]\ar[r,"f''"]\ar[l,swap,"u"] & B\ar[d,equals] \\
 A & \ss^sA\ar[r,"f'"]\ar[l,swap,"\eta_s"] & B
\end{tikzcd}\end{equation}
in $\cc_0$. Giving a commutative diagram 
\begin{equation}\begin{tikzcd}
    \yy(A)\ar[d,equals] & \ss^r\yy(A)\ar[r,"\yy(f)"]\ar[l,swap,"\eta_r"] & \yy(B)\ar[d,equals] \\
 \yy(A) \ar[d,equals] & \yy(C)\ar[u,"\yy(\phi)"]\ar[d,"\yy(\varphi)"]\ar[r,"\yy(f'')"]\ar[l,swap,"\yy(u)"] &\yy( B)\ar[d,equals] \\
\yy( A )& \ss^s\yy(A)\ar[r,"\yy(f')"]\ar[l,swap,"\eta_s"] & \yy(B)
\end{tikzcd}\end{equation}
therefore $[\yy(f)]_{\text{Split}_\pp(\cc)_\infty}=[\yy(f')]_{\text{Split}_\pp(\cc)_\infty}$, i.e., $\zz(f)=\zz(f')$. We show that $\zz$ is fully faithful. For fullness,
let $f\in \text{Hom}_{\text{Split}_\pp(\cc)_\infty}(\yy(A),\yy(B))$, then let $\bar{f}\in \text{Hom}_{\text{Split}_\pp(\cc)_0}(\ss^r \yy(A),\yy(B))$ represent $f$, then $\bar{f}=\yy(f')$ for some $f'\in \text{Hom}_{\cc_0}(\ss^rA,B)$ because $\yy$ is fully faithful. Then $f'$ is a representative of some $\tilde{f}\in \text{Hom}_{\cc_\infty}(\ss^r A,B)$, hence $\zz(\tilde{f})=f$.
 For faithfulness, assume $\zz(f)=\zz(g)$ in $\text{Hom}_{\text{Split}_\pp(\cc)_\infty}(\yy(A),\yy(B))$, we show $f=g$ in $\cc_\infty$. We have a commutative diagram

 \begin{equation}\begin{tikzcd}
    \yy(A)\ar[d,equals] & \ss^r\yy(A)\ar[r,"\zz(f)'"]\ar[l,swap,"\eta_r"] & \yy(B)\ar[d,equals] \\
 \yy(A) \ar[d,equals] & \yy(C)\ar[u,"\yy(\phi)"]\ar[d,"\yy(\varphi)"]\ar[r,"\yy(f'')"]\ar[l,swap,"\yy(u)"] &\yy( B)\ar[d,equals] \\
\yy( A )& \ss^s\yy(A)\ar[r,"\zz(g)'"]\ar[l,swap,"\eta_s"] & \yy(B)
\end{tikzcd}\end{equation}
where $\zz(f)'$ and $\zz(g)'$ represent $\zz(f)$ and $\zz(g)$ respectively. We can write $\zz(f)'= \yy(\tilde{f})$ and $\zz(g)'=\yy(\tilde{g})$. Hence we have a commutative diagram in $\cc_0$
\begin{equation}\begin{tikzcd}
    A\ar[d,equals] & \ss^rA\ar[r,"\tilde{f}"]\ar[l,swap,"\eta_r"] & B\ar[d,equals] \\
 A \ar[d,equals] & C\ar[u,"\phi"]\ar[d,"\varphi"]\ar[r,"f''"]\ar[l,swap,"u"] & B\ar[d,equals] \\
 A & \ss^sA\ar[r,"\tilde{g}"]\ar[l,swap,"\eta_s"] & B
\end{tikzcd}\end{equation}

Since $\zz$ is well defined we have $f=[\tilde{f}]_{\cc_\infty}$ and $g=[\tilde{g}]_{\cc_\infty}$. By the commutativity of the above diagram $[\tilde{f}]_{\cc_\infty}=[\tilde{g}]_{\cc_\infty}$, hence $f=g$. So far we have shown $\zz: \cc_\infty \to \text{Split}_\pp(\cc)_\infty$ is fully faithful and moreover as $\text{Split}_\pp(\cc)_\infty$ is idempotent complete, $(\text{Split}_\pp(\cc)_\infty, \zz)$ gives an idempotent completion of $\cc_\infty$. By the universal property of Karoubi extensions, $\zz$ extends uniquely (up to natural isomorphism) to a functor $\tilde{\zz}:\text{Split}(\cc_\infty) \to \text{Split}_\pp(\cc)_\infty$. Explicitly, we may define $\tilde{\zz}$ as follows: firstly $\tilde{\zz}|_{\yy(\cc_\infty)}=\zz$, i.e., $\tilde{\zz}(\yy(A))=\zz(A)$. Now assume $e: A \to A$ splits as $F \xrightarrow{s} \yy(A) \xrightarrow{r} F$ in $\text{Split}(\cc_\infty)$, with $F\notin \yy(\cc_\infty)$ then define

\begin{equation}\tilde{\zz}(F)= \tilde{F}\end{equation}
where $\tilde{F} \xrightarrow{\tilde{s}} \zz(A) \xrightarrow{\tilde{r}} \tilde{F}$ is a choice of splitting of $\zz(e):\zz(A) \to \zz(A)$ in $\text{Split}_\pp(\cc)_\infty$. On morphisms, let $f\in \text{Hom}_{\text{Split}(\cc_\infty)} (F,G)$ then $\zz(f)\in \text{Hom}_{\text{Split}_\pp(\cc)_\infty}(\zz(F) ,\zz(G))$, where $G\xrightarrow{s'}B\xrightarrow{r'}G$ is a splitting of $e':B\to B$ and $\tilde{\zz}(G)=\tilde{G}$ in which $\tilde{G} \xrightarrow{\tilde{s}'}\zz(B) \xrightarrow{\tilde{r}'}\tilde{G}$ is a choice splitting of $\zz(e')$. Then $\zz(f)$ is defined by

\begin{equation}
    \Tilde{\zz}(f)= \tilde{r}' \circ \zz(s' \circ f \circ r) \circ \tilde{s}
\end{equation}

diagrammatically,

\begin{equation}\begin{tikzcd}
    A\ar[d,"r"] & B\ar[d,"r'"] && \tilde{F}\ar[d,red,"\tilde{s}"] &\ar[d,"\tilde{s}'"] \tilde{G}\\
    F\ar[r,red,"f"] \ar[d,"s"]& G\ar[d,"s'"] & \xrightarrow{\tilde{\zz}}& \zz(A)\ar[r,red,"\zz(s'\circ f \circ r)"]\ar[d,"\tilde{r}"] & \zz(B)\ar[d,red,"\tilde{r}'"]\\
     A & B && \tilde{F} & \tilde{G}
\end{tikzcd}\end{equation}
Note that if $F,G\in \yy(\cc_\infty)$ then 
\begin{align*}
\tilde{\zz}(f)& = \tilde{r}'\circ \zz(s')\circ \zz(f) \circ \zz(r) \circ \tilde{s}\\
&= \tilde{r}' \circ \tilde{s}' \circ \zz(f) \circ \tilde(r) \circ \tilde{s}\\
&= \zz(f)
\end{align*}
hence this is a well defined extension. We check that $\tilde{\zz}$ is fully faithful. Fullness: Assume $\tilde{f}\in \text{Hom}_{\text{Split}_\pp(\cc)_\infty}(\tilde{F},\tilde{G})$ where $\tilde{F} \xrightarrow{\tilde{s}} \zz(A) \xrightarrow{\tilde{r}} \tilde{F} $ and $\tilde{G} \xrightarrow{\tilde{s}'} \zz(B) \xrightarrow{\tilde{r}'} \tilde{G} $ are the respective retracts. Then $\tilde{s}'\circ \tilde{f} \circ \tilde{r}\in \text{Hom}_{\text{Split}_\pp(\cc)_\infty}(\zz(A),\zz(B))$ and so can be written as $\zz(f)=\tilde{s}'\circ \tilde{f}\circ \tilde{r}$. We then have

\begin{align*}
    \tilde{\zz}(s' \circ f \circ r)=&\tilde{r}'\circ \zz(s' \circ r' \circ f \circ s \circ r) \circ \tilde{s}\\
    =&\tilde{r}' \circ \zz( e' \circ f \circ e) \circ \tilde{s}\\
    =&\tilde{r}' \circ \zz(e') \circ \zz(f) \circ \zz(e) \circ \tilde{s}\\
    =&\tilde{r}' \circ \tilde{s}' \circ \tilde{r}' \circ \zz(f) \circ \tilde{s} \circ \tilde{r} \circ \tilde{s}\\
    =& \tilde{r}' \circ \zz(f) \circ \tilde{s}\\
    =& \tilde{r}' \circ \tilde{s}' \circ \tilde{f} \circ \tilde{r} \circ \tilde{s}\\
    =& \tilde{f}.
\end{align*}

For faithfulness: Assume $\tilde{\zz}(f)=\tilde{\zz}(g)$, then:

\begin{align*}
    \tilde{r}'\circ \zz(s' \circ f\circ r)\circ \tilde{s} =& \tilde{r}'\circ \zz(s' \circ g\circ r)\circ \tilde{s}\\
    \implies \tilde{s}' \circ \tilde{r}'\circ \zz(s' \circ f\circ r)\circ \tilde{s} \circ \tilde{r}=&\tilde{s}' \circ \tilde{r}'\circ \zz(s' \circ g\circ r)\circ \tilde{s} \circ \tilde{r}\\
    \implies \tilde{e}' \circ \zz(s' \circ f \circ r) \circ \tilde{e}=& \tilde{e}' \circ \zz(s' \circ g \circ r) \circ \tilde{e}\\
    \implies \zz(e') \circ \zz(s' \circ f \circ r) \zz(e)= & \zz(e') \circ \zz(s' \circ g \circ r) \circ \zz(e)\\
    \implies \zz(e' \circ s' \circ f \circ r \circ e)=& \zz(e' \circ s' \circ g\circ r \circ e)\\
    \implies e' \circ s' \circ f \circ r \circ e =& e' \circ s' \circ g \circ r \circ e\\
    \implies s' \circ f \circ r =& s' \circ g \circ r\\
    \implies r' \circ s' \circ f \circ r \circ s =&  r' \circ s' \circ g \circ r \circ s \\
    \implies f=&g
\end{align*}
Finally, we find that $\tilde{\zz}$ is essentially surjective, indeed any object of $\text{Split}_\pp(\cc)_\infty$ is a retract of some representable object. Thus, it defines an idempotent on said representable object. By fullness, we can write this as the image of an idempotent in $\text{Split}(\cc_\infty)$. Choose some splitting of this in $\text{Split}(\cc_\infty)$, then by uniquness of idempotent splitting the image of this splitting under $\tilde{\zz}$ is isomorphic to our original object.
\end{proof}
\end{theorem}

Define $\text{Split}^0_\pp(\cc)$ to be the full subcategory of $\text{Split}_\pp(\cc)$ consisting of objects that are weight zero retracts. 
\begin{rmk}
    $\text{Split}^0_\pp(\cc)_0$ is equivalent to the Karoubi completion of $\cc_0$. 
\end{rmk}
\begin{theorem}
     Let $\cc$ be a persistence category, if every idempotent of $\cc_\infty$ is represented by an idempotent of weight zero in $\cc_0$ then $\text{Split}(\cc_\infty)$ is equivalent to $\text{Split}_\pp^0(\cc)_\infty$.
     \begin{proof}
         The functor $\tilde{\zz}$ as before restricts to an equivalence onto $\text{Split}_\pp^0(\cc)_\infty \subset \text{Split}_\pp(\cc)_\infty$. Indeed, if $F\in \text{Split}_\pp^0(\cc)_\infty  $ then $F<\yy(A)$ for some $A$, and the maps $F \xrightarrow{[s]} \yy(A) \xrightarrow{[r]} F$ with $[r] \circ [s] =1_F$ defining $F$, can be represented by maps $r:\yy(A) \to F$ and $s:F \to \yy(A)$ in $\cc_0$, with $r \circ s=1_F$. The composition $\yy(e):=s \circ r$ is a $0$-idempotent on $\yy(A)$. Take a splitting of $\yy([e])\in \text{Hom}_{\text{Split}(\cc_\infty)}(\yy(A),\yy(A))$, say $F' \xrightarrow{s'} \yy(A)\xrightarrow{r'} F'$.  Then $\tilde{\zz}(F') \xrightarrow{\tilde{\zz}(s')} \yy(A) \xrightarrow{\tilde{\zz}(r')} \tilde{\zz}(F')$ defines a splitting of the idempotent $\tilde{\zz}(s') \circ \tilde{\zz}(r')=\tilde{\zz}(s' \circ r')$. By definition of $\tilde{\zz}$, $\tilde{\zz}(s' \circ r')=[s] \circ [r] $, hence $\zz(F')$ is a splitting of $[s]\circ [r]$ and so is isomorphic to $F$. This shows that the $\tilde{\zz}$ is essentially surjective and hence an equivalence to $\text{Split}_\pp^0(\cc)_\infty$.
     \end{proof}
\end{theorem}

\section{Triangulation}\label{secttri}
\subsection{TPC structure of Karoubi completions}
Given a TPC refinement $\cc$ of a triangulated category $\cc_\infty$, one can ask whether the persistence categories $\text{Split}_\pp(\cc)$ or $\widehat{\text{Split}}_\pp(\cc)$ are still TPCs. This is key for one to extend fragmentation metrics TPCs to their idempotent completion. We will see however that in general the idempotent completions we constructed need not carry the structure of TPCs. We do however show the following positive result, that gives a criteria for TPC to complete to an idempotent complete TPC:

\begin{theorem}\label{TPCcriteria}
    Given a TPC $\cc$, if every idempotent $e\in \text{Hom}_{\cc_\infty}(A,A)$ is represented by some weight zero idempotent $e'\in \text{Hom}_{\cc_0}(A,A)$ in $\cc_0$ then $\text{Split}_\pp^0(\cc)$ is a TPC.
   
\end{theorem}
\begin{proof}
First we have that $\text{Split}_\pp^0(\cc)_0$ is equivalent to $\text{Split}(\cc_0)$. Which, by assumption that $\cc_0$ is triangulated, is also triangulated with exact triangles being triangles that are retracts of the image of an exact triangle in $\cc_0$ under $\yy$. Next we show every morphism $A\xrightarrow{\eta_r}\ss^{-r}A$ embeds into an exact triangle in $\text{Split}_\pp^0(\cc)_0$ with $r$-acyclic cone. If $A=\yy(A')$ then there exists such a triangle given by the image of a triangle 
        
        $\begin{tikzcd}
            A'\ar[r,"\eta_r"] & \ss^{-r}A'\ar[r,"i"] & K' \ar[r,"j"]& TA'
        \end{tikzcd}$
        with $K'\simeq_r 0$. Assume $A$ is not representable, then $A<_0 \yy(A')$ for some $A'$. Consider the commutative diagram

        \begin{equation}
            \begin{tikzcd}
                A \ar[r,"\eta_r"]\ar[d,"s"]& \ss^{-r}A\ar[d,"\ss^{-r}s"]& &\\
                \yy(A')\ar[d,"r"]\ar[r,"\eta_r"] & \ss^{-r}\yy(A')\ar[d,"\ss^{-r}r"]\ar[r,"\yy(i)"]& \yy(K')\ar[r,"\yy(j)"] & T\yy(A')&\\
                A \ar[r,"\eta_r"]& \ss^{-r}A& &
            \end{tikzcd}
        \end{equation}
        using the $0$-idempotent $s \circ r:\yy(A') \to \yy(A')$ induced by the retraction $A<_0 \yy(A')$, and the fact $\cc_0$ is triangulated, we can construct a $0$-idempotent on $\yy(K)$. Denote this by $e_K$, then $e_K$ satisfies $e_K \circ \yy(i)=\yy(i) \circ \ss^{-r}(s \circ r)$ and $\yy(j) \circ e_K = T(s \circ r) \circ \yy(j)$  . This idempotent splits in $\text{Split}_\pp^0(\cc)_0$, assume by $K$. Thus consider
        \begin{equation}
            \begin{tikzcd}
                A \ar[r,"\eta_r"]\ar[d,"s"]& \ss^{-r}A\ar[d,"\ss^{-r}s"]\ar[r,dotted,"r_K \circ \yy(i) \circ \ss^{-r}s"]& \ar[r,dotted,"Tr\circ \yy(j) \circ s_K"]K\ar[d,"s_K"]&TA\ar[d,"Ts"]\\
                \yy(A')\ar[d,"r"]\ar[r,"\eta_r"] & \ss^{-r}\yy(A')\ar[d,"\ss^{-r}r"]\ar[r,"\yy(i)"]& \yy(K')\ar[d,"r_K"]\ar[r,"\yy(j)"] & T\yy(A')\ar[d,"Tr"]&\\
                A \ar[r,"\eta_r"]& \ss^{-r}A\ar[r,dotted,"r_K \circ \yy(i) \circ \ss^{-r}s"]& \ar[r,dotted,"Tr\circ \yy(j) \circ s_K"] K& TA
            \end{tikzcd}
        \end{equation}
        One checks that each square is commutative. Note that the top row is a retract of the central exact triangle and so is an exact triangle (see \cite{BSc}). We now check that $K\simeq_r 0$, we have that the following diagram commutes
        \begin{equation}
            \begin{tikzcd}
                K\ar[r,"\eta_r^K"]\ar[d,"s_K"] & \ss^{-r}K\ar[d,"\ss^{-r}s_K"]\\
                \yy(K')\ar[d,"r_K"]\ar[r,"\eta_r^{\yy(K')}"] &\ar[d,"\ss^{-r}r_K"] \ss^{-r}\yy(K')\\
                K \ar[r,"\eta_r^K"]& \ss^{-r}K
            \end{tikzcd}
        \end{equation}
        hence $\eta_r^K = \eta_r^K \circ r_K \circ s_K = \ss^{-r}r_K \circ \eta_r^{\yy(K')}\circ s_K=0$ as $\eta_r^{K'}=0$. Thus we have verified the axioms for $\text{Split}_\pp^0(\cc)$ to be a TPC.

    \end{proof}

We remark that under the assumption of all idempotents of $\cc_\infty$ being represented by weight zero idempotents in $\cc_0$, the persistence category  $\widehat{\text{Split}}(\cc)$ also exhibits a TPC refinement of $\text{Split}(\cc_\infty)$. In fact one can see that  $\widehat{\text{Split}}_\pp^0(\cc)_0$ and $\text{Split}_\pp^0(\cc)_0$ are equivalent as triangulated categories.

\begin{rmk}
    If there exists an idempotent in $\cc_\infty$ not represented by a weight zero idempotent in $\cc_0$, then the constructions $\text{Split}_\pp(\cc)$ and $\widehat{\text{Split}}_\pp(\cc)$ are not TPC refinements of $\text{Split}(\cc_\infty)$. 
It is clear that under the assumption, $\widehat{\text{Split}}_\pp(\cc)_0$ is not a category, but only a semi category and it is not clear what it means to be triangulated in this case. Viewing $\text{Split}(\cc_\infty)$ as the full subcategory of retracts of representable objects in $\text{PSh}(\cc)$, then the construction of $\text{Split}_\pp(\cc)$ consists of too many objects for its associated zero level category to be triangulated. One realises this issue when trying to extend the triangulated functor of $\cc_0$ to a triangulated functor on $\text{Split}_\pp(\cc)_0$. In particular if one were to define $TF$ for some non-representable $F<_r \yy(A)$ then one would want $TF$ to be the splitting of the $r$-idempotent $Te_F:T\yy(A) \to T\yy(A)$ where $e_F$ is the idempotent inducing the splitting $F<_r \yy(A)$. However as non-zero weighted idempotent splittings are not unique up to isomorphism but only up to weighted isomorphism we realise that $TF$ is not well defined.
\end{rmk}

One reason for wanting to define a persistence refinement of the Karoubi completion of a persistence category, is to be able to extend the triangular weights (in the case that the category is a TPC) on $\cc_\infty $ to $\text{Split}(\cc_\infty)$. Recall (\cite{BCZ}) the unstable weight of an exact triangle in the limit category of a TPC,
\begin{equation}\Delta = A \xrightarrow{u} B \xrightarrow{v} C \xrightarrow{w} TA\end{equation} is given by
\begin{equation}w_\infty(\Delta) = \inf\{r : A \xrightarrow{u'} \ss^{-r_1}B \xrightarrow{ v'} \ss^{-r_2}C \xrightarrow{w'} \ss^{-r}TA \text{ is strict exact} \} \end{equation}
with $0\leq r_1 \leq r_2 \leq r$ and such that 

\begin{align*}
    u &= [\eta_{-r_1,0} \circ u']\\
    v&= [\eta_{-r_2,0} \circ v' \circ \eta_{0,-r_1}]\\
     w &= [\eta_{-r,0} \circ w' \circ \eta_{0,-r_2}].
\end{align*}

The stable weight of $\Delta$ is then given by 

\begin{equation}w(\Delta) = \inf\{w_\infty(\ss^{s,0,0,s}\Delta)\}.\end{equation}
 We would like to define a weight to an exact triangle in $\text{Split}(\cc_\infty)$
\begin{equation}\Delta = F \xrightarrow{u} G \xrightarrow{v} H \xrightarrow{w} TF\end{equation}

Furthermore one would like this weighting to be an extension of the weighting on $\cc_\infty
$.

\begin{lemma}
If every idempotent in $\cc_\infty$ is represented by a weight zero idempotent, then the triangular weights associated to $\text{Split}_\pp^0(\cc)_\infty$ using the TPC structure of $\text{Split}_\pp^0(\cc)_\infty$ extend the triangular weights of $\cc_\infty$ to $\text{Split}(\cc_\infty)$.
\begin{proof}
    Let $w_\infty$ be the weight associated to $\cc_\infty$ and $w_\infty^\pi$ be the weight associated to $\text{Split}_\pp^0(\cc)_\infty$. Then let $\Delta$ be a triangle in $\cc_\infty$, and $\yy(\Delta)$ be the associated triangle in $\text{Split}_\pp^0(\cc)_\infty$. Assume that $\bar{\Delta}$ is a strict exact triangle in $\cc_0$ representing $\Delta$, then $\yy(\bar{\Delta})$ is a strict exact triangle representing $\yy(\Delta)$. In particular we see 
    \begin{equation}
        w_\infty(\Delta) \geq w^\pi_\infty( \yy( \Delta))
    \end{equation}
\end{proof}
\end{lemma}

\begin{rmk}
    The construction of the triangular weights associated to (the limit category of) a TPC relies on the assumption that the zero level category is being triangulated. This assumption is used to prove the weighted octahedral axiom. Without the zero level being triangulated it is not at all clear if one can construct triangular weights. This is the case for idempotents not being represented by weight zero idempotents. That is, the zero level categories of $\widehat{\text{Split}}_\pp(\cc)$ and $\text{Split}_\pp(\cc)$ are not triangulated and so it is not clear if one can extend the triangular weights of $\cc_\infty$ to $\text{Split}(\cc_\infty)$.
\end{rmk}

\section{Comments and Examples}\label{sectcomm}

\subsection{Triangulation in $\text{Split}_\pp(\cc)_0$}
For proofs of results stated in this subsection see the appendix. One would hope that the persistence idempotent completion $\text{Split}_\pp(\cc)$ of a TPC, is itself a TPC. As mentioned previously, it becomes apparent very quickly that this will not be the case when we try to extend the functor $T: \cc_0 \to \cc_0$ to a functor $T:\text{Split}_\pp(\cc)_0 \to \text{Split}_\pp(\cc)_0$.
Assume $F$ is given by an $r$-retract of $\yy(A)$

    \begin{equation}\begin{tikzcd}
        \yy(A)\ar[dr,"r_F"] & \\
        F\ar[u,"s_F"]\ar[r,"\eta_r^F"] & \ss^{-r}F
    \end{tikzcd}\end{equation}
    The retract defines an $r$-idempotent $\yy(e_F): \yy(A) \to \ss^{-r}\yy(A)$, and in turn an $r$-idempotent $e_F: A \to \ss^{-r}A$ in the original category $\cc$. Apply the functor $T$ on $\cc$ to obtain an $r$-idempotent $Te_F: TA \to \ss^{-r}TA$ and thus obtain $\yy(Te_F) : \yy(TA) \to \ss^{-r}\yy(TA)$ a new $r$-idempotent in $\text{Split}_\pp(\cc)$. This must split, say via the retract
       \begin{equation}\begin{tikzcd}
        \yy(TA)\ar[dr,"r_F'"] & \\
        F'\ar[u,"s_F'"]\ar[r,"\eta_r^{F'}"] & \ss^{-r}F'
    \end{tikzcd}\end{equation}
    Ideally we would define $TF:= F'$. However, $r$-idempotents only split up to (strong) $2r$-isomorphism. Thus we need to make a choice of splitting of $\yy(Te_F)$ to define $TF$. This issue however is not so bad, any two choices of $TF$ will be isomorphic in the limit category $\text{Split}_\pp(\cc)_\infty$. If one attempts to explore what kind of structure $\text{Split}_\pp(\cc)_0$ has, one will find that there is some class of `triangles'. In fact any morphism $F \to G$ can be completed to a triangle 

\begin{equation}
       \Delta:= \begin{tikzcd}
            F\ar[r] & G\ar[r] & H\ar[r] & \ss^{-r}TF.
        \end{tikzcd}
    \end{equation}

This triangle can be shown to be an $s$-retract of some representable strict exact triangle of weight also $r$. Here both $r$ and $s$ depend on the size of the retraction defining both $F$ and $G$. One could hope that this class of triangles (of the form $\Delta$ with a weight retraction onto a representable strict exact triangle) would allow us to define weights on the limit category $\text{Split}_\pp(\cc)_\infty$ in a similar way to the usual TPC construction. One can show that any exact triangle in the limit category can be represented by such a triangle. However, proving a weighted octahedral axiom seems unlikely. This is due to the following result 
 \begin{prop}
        Let $\cc$ be a TPC and $\begin{tikzcd}
             A\ar[r,"u"] & B\ar[r,"v"] & C\ar[r,"w"] & TA
        \end{tikzcd}$ be an exact triangle in $\cc_0$. Assume $e_A$ is an $r$-idempotent on $A$ and $e_B$ is an $s$-idempotent on $B$ such that the following commutes

        \begin{equation}\begin{tikzcd}
            A\ar[d,"\eta_s \circ e_A"]\ar[r,"u"] & B\ar[r,"v"] \ar[d,"\eta_r \circ e_B"]& C\ar[rr,"w"] & &TA\ar[d,"T(\eta_{s}\circ e_A)"]\\
            \ss^{-(r+s)}A\ar[r,"\ss^{-(r+s)u}"] & \ss^{-(r+s)}B\ar[r,"\ss^{-(r+s)}v"] & \ss^{-(r+s)}C\ar[rr,"\ss^{-(r+s)}w"] & &\ss^{-(r+s)}TA
        \end{tikzcd}\end{equation}
        Then there exists a $3(r+s)$-idempotent $e_C:C \to \ss^{-3(r+s)}C$, such that

          \begin{equation}\begin{tikzcd}
            A\ar[d,"\eta_{2(r+s)}\circ \eta_s \circ e_A"]\ar[r,"u"] & B\ar[r,"v"] \ar[d,"\eta_{2(r+s)}\circ \eta_r \circ e_B"]& C\ar[rr,"w"]\ar[d,dotted,"e_C"] & &TA\ar[d,"T(\eta_{2(r+s)}\eta_{s}\circ e_A)"]\\
            \ss^{-3(r+s)}A\ar[r,"\ss^{-3(r+s)u}"] & \ss^{-3(r+s)}B\ar[r,"\ss^{-3(r+s)}v"] & \ss^{-3(r+s)}C\ar[rr,"\ss^{-3(r+s)}w"] & &\ss^{-3(r+s)}TA
        \end{tikzcd}\end{equation}
        commutes.  
\end{prop}

To prove the regular octahedral axiom for the Karoubi completion of a triangulated category, one needs to construct a retract (alternatively idempotent) on a mapping cone using the retraction (idempotents) of the first two objects. It is known that there is a natural way to do this and is detailed in \cite{BSc}. However, when we attempt to prove a weighted version of this we appear to only be able to construct a retract (idempotent) of weight three times the sum of the weights of the idempotents on the first two objects. This causes an issue in attempting to control weights of triangles in the weighted octahedral axiom.

\subsection{Examples}

\subsubsection{Non-weight zero representable idempotents}
We construct a persistence category $\cc$ whose limit category has an idempotent not represented by weight zero idempotent in $\cc_0$. Define $\cc$, to be the persistence category with one object $*$ and with morphisms given by 
\begin{equation}
    \text{Hom}_\cc(*,*)(r):=\begin{cases}
       0 & r< 0\\
         k\langle a_r \rangle & r\in [0,1)\\
        k\langle a_r \rangle\oplus k\langle e_r \rangle & r\geq 1
    \end{cases}
\end{equation}
The persistence module structure maps are given by 
\begin{align}
    i_{s,t}(a_s)=&a_t\\
    i_{s,t}(e_s)=&e_t \notag
\end{align}
Composition is given by
\begin{align*}
    a_r \circ a_s:=& a_{r+s}\\
    e_r \circ e_s:=& e_{r+s}\\
    a_r \circ e_s:=& e_{r+s}\\
    e_r \circ a_s:=& e_{r+s}
\end{align*}
We have that the identity is given by $a_0$, and the limit category consists of one object $*$ with $\text{Hom}_{\cc_\infty}(*,*)=k\langle a_\infty \rangle \oplus k\langle e_\infty \rangle$. Where $e_\infty$ is an idempotent. Furthermore, $e_\infty$ is not represented by a weight zero idempotent in $\cc$.

\subsubsection{Filtered chain complexes}
The category of filtered chain complexes (over a field) is defined to be the category with objects being pairs $((C_*,d),\ff)$ where $(C_*,d)$ is a chain complex and $\ff$ is a filtration of $(C_*,d)$, i.e. a family of sub-complexes $\ff:=\{(C_*^{\leq r},d^{\leq r})\}_{r\in \R}$. The filtration is such that for all $s\leq r$ we have as sub-complexes
\begin{equation}
    (C_*^{\leq s},d^{\leq s}) \subset (C_*^\leq{ r},d^{\leq r}) 
\end{equation}
and $d^{\leq r}$ is the restriction of $d$ to $C_*^{\leq r}.$ The hom-sets are chain maps that respect this filtration, i.e restrict to chain maps on each filtration level. We allow these maps to also carry a shift, so a map $f:((C_*,d),\ff_C) \to( (C'_*,d'),\ff')$ has shift $s$ if it consists of a family of chain maps $f^{\leq r}:(C_*^{\leq r},d^{\leq r}) \to ((C_*')^{\leq r+s},(d')^{\leq r+s})$. The hom-sets naturally form filtered chain complexes with filtration given by this shifting. Passing to the homotopy category of chain complexes levelwise one obtains a persistence category. This persistence category is naturally weighted idempotent complete, i.e. every weighted idempotent splits. This follows from realising the category has kernels, formed by taking level-wise kernels maps of chain map complexes. Recall that chain complexes (over $k-$Mod) form an Abelian category. 

In general one can study chain complexes of some additive category $\mathcal{A}$. This is the first step in defining the derived category for $\mathcal{A}$. Thus, perhaps $\mathcal{A}$ comes with additional structure making it a filtered (or persistence) category. Then one would look to filtered chain complexes over $\mathcal{A}$. In general this need not be weighted idempotent complete, being so would depend on the structure of $\mathcal{A}$. I.e. if $\mathcal{A}$ is Abelian then chain complexes over $\mathcal{A}$ is also Abelian.

\subsubsection{Filtered dg-categories}
The above example of filtered chain complexes (up to homology) is a specific example of a persistence category arising from a filtered dg-category. A filtered dg-category is simply a category whose hom-sets form filtered chain complexes. One can pass to its corresponding homotopy category, to obtain a persistence category. Given any dg-category one can construct a filtered dg-category by filtering the objects and morphisms similarly to that of filtered chain complexes. In general a dg-category need not be 
Abelian and need not be idempotent complete. Thus, in general we cannot split weighted idempotents in the corresponding persistence category via taking kernels. However, if the original dg-category is Abelian then the persistence category will be weighted idempotent complete. 

 Note that even if the original dg-category $\cc$ is idempotent complete, its filtered category $\ff \cc$ might not be weighted idempotent complete.  Indeed a map $e\in \text{Hom}_{\ff \cc}((A,\ff),(A,\ff))(r)$ can be regarded as a family of maps $\{e(t):A(t) \to A(t+r)\}_{t\in \R}$, $e$ being an $r$-idempotent means that $e(t+r) \circ e(t) = i_{t+r,t+2r} \circ e(t)$. If $r=0$ then naturally $e$ splits via idempotent completeness of $\cc$. However, if $e>0$ then splitting depends on the existence of $\text{Ker}(e(t)-i_{t,t+r})$ for all $t\in \R$.

\subsubsection{Persistence refinements of idempotent complete categories}
Assume that we have a category $\cc$ that admits a persistence refinement $\cc'$ (i.e. $\cc'_\infty=\cc$). Furthermore, assume $\cc$ is idempotent complete. Then $\cc'$ may not be weighted idempotent complete, but every weighted idempotent `weakly splits'. By this we mean that perhaps $e:A \to \ss^{-r}A$ does not split, but for some $s$ large enough, $\eta_s \circ e:A \to \ss^{-(r+s)}A$ does split. In particular, we realise that $\text{Split}_\pp(\cc')\neq \cc'$ even if $\cc'_\infty$ is idempotent complete. The completed category forces all weighted idempotents to split and not just weakly split.

\section{Appendix}\label{sectappen}
We leave for the appendix a series of calculations and results concerning the triangulated structure of $\text{Split}_\pp(\cc)_0$. Recall that this is not a triangulated category in general. Thus, we look to understand what kind of structure it might form, in particular, what the notion of exact/strict exact triangle could be in this setting. 

 \begin{rmk}
        If $F<_r \yy(A)$ then $\ss^{-r}F<_r \ss^{-r}\yy(A)$, where the $s_{\ss^{-r}F}=\ss^{-r}s_F$ and $r_{\ss^{-r}F}=\ss^{-r}r_F$. This gives that the induced idempotent on $\ss^{-r}\yy(A)$ is given by $e_{\ss^{-r}F}=\ss^{-r}e_F$. In turn this means that $\yy(Te_{\ss^{-r}F})=\yy(T\ss^{-r}e_F)=\yy(\ss^{-r}Te_F)=\ss^{-r}\yy(Te_F)$. In particular this gives us that $T\ss^{-r}F$ can be taken to be $\ss^{-r}TF$.
    \end{rmk}

    \begin{rmk}\label{diagconstructing}
    Let $F,G\in \text{PSh}_{\pp}(\cc)$ be $r,s$-retracts of $\yy(A),\yy(B)$ respectively and $u:F \to G$ a morphism. We can produce diagrams

      \begin{equation}\begin{tikzcd}
          \yy(A)\arrow[d,"r_F"] \ar[r,"\tilde{u}",dotted]& \ss^{-r}\yy(B)\\
          \ss^{-r}F \ar[r,"\ss^{-r}u"]& \ss^{-r}G\ar[u,"\ss^{-r}s_G"]
      \end{tikzcd} \hspace{1cm}\tilde{u}:=\ss^{-r}s_G \circ \ss^{-r}u \circ r_F\end{equation}

      \begin{equation}\begin{tikzcd}
        F \ar[d,"s_F"]\ar[r,"\bar{u}",dotted]& \ss^{-(r+s)}G\\
        \yy(A)\ar[r,"\tilde{u}"] & \ss^{-r}\yy(B)\ar[u,"\ss^{-r}r_G"]
      \end{tikzcd}\hspace{1cm} \bar{u}:=\ss^{-r}r_G \circ \tilde{u}\circ s_F\end{equation}

  Such that 
      \begin{align*}
          \bar{u}=&\ss^{-r}r_G \circ \ss^{-r}s_G\circ \ss^{-r}u\circ r_F \circ s_F\\
          =&\eta_s \circ \ss^{-r}u \circ \eta_r\\
          =& \eta_{r+s}\circ u
    \end{align*}
    `Gluing' these diagrams together we obtain a commutative diagram 
    \begin{equation}\begin{tikzcd}
        F\ar[r,"u"]\ar[d,"s_F"] & G\ar[d,"\ss^{-r}s_G\circ \eta_r"] \\
        \yy(A) \ar[r,"\tilde{u}"]\ar[d,"\ss^{-s}r_F\circ \eta_s"]& \ss^{-r}\yy(B)\ar[d,"\ss^{-r}r_G"]\\
        \ss^{-(r+s)}F\ar[r,swap,"\ss^{-(r+s)}u"] & \ss^{-(r+s)}G
    \end{tikzcd}\end{equation}
  where the vertical compositions give $\eta_{(r+s)}$. Hence any morphism $u$ can be extended to give a morphism of weighted retract diagrams. Similarly take any morphism $v:\yy(A) \to \yy(B) $, then it also induces a morphism of weighted retract diagrams:

\begin{equation}\begin{tikzcd}
    F\ar[rr,"\tilde{v}"]\ar[d,"s_F"] & &\ss^{-s}G\ar[d,"\eta_r \circ \ss^{-s}s_G"]\\
    \yy(A)\ar[rr,"\ss^{-r}e_G \circ \ss^{-r}v \circ e_F"] \ar[d,"\eta_s \circ r_F"]& &\ss^{-(r+s)}\yy(B)\ar[d,"\ss^{-(r+s)}r_G"]\\
    \ss^{-(r+s)}F \ar[rr,"\ss^{-(r+s)}\tilde{v}"]& &\ss^{-(r+s)}\ss^{-s}G
\end{tikzcd}\end{equation}
     where $\tilde{v}=r_G \circ v \circ s_F$. 
\end{rmk}
 We can then look using this to define $T$ on morphisms, i.e., $T(u:F \to G)$. We first extend $u$ to a morphism of weighted retract diagrams

\begin{equation}
    \begin{tikzcd}
        F \ar[r,"u"]& G
    \end{tikzcd}\rightsquigarrow \begin{tikzcd}
        F\ar[r,"u"]\ar[d,"s_F"] & G\ar[d,"\ss^{-r}s_G\circ \eta_r"] \\
        \yy(A) \ar[r,"\tilde{u}"]\ar[d,"\ss^{-s}r_F\circ \eta_s"]& \ss^{-r}\yy(B)\ar[d,"\ss^{-r}r_G"]\\
        \ss^{-(r+s)}F\ar[r,swap,"\ss^{-(r+s)}u"] & \ss^{-(r+s)}G
    \end{tikzcd}
\end{equation}

Then apply $T$ to $\tilde{u}:A \to \ss^{-r}A$ to obtain a map $\yy(T\tilde{u}):\yy(TA) \to \ss^{-r}\yy(TA)$. We then take 
\begin{equation}
    T(u)=r_{TG} \circ \yy(T\tilde{u}) \circ s_{TF} \in \text{Hom}_{\text{Split}_\pp(\cc)_0}(TF, \ss^{-(r+s)}TG)
\end{equation}
\begin{rmk}
    The target of the morphism $Tu$ is not $TG$ but a $TG$ shifted. Notice that the shift depends on the weighting of $F$ and $G$ as weighted retracts of representable objects. It is clear that $T$ does not define an endofunctor in the usual sense on $\text{Split}_\pp(\cc)_0$, but some form of generalisation of functor that includes a shift. Do note however that $T$ restricts to a well defined functor on representable objects.
\end{rmk}

\begin{rmk}\label{rtfvstrf}
    We note that $s_{TF}\neq Ts_F$ and $r_{TF}\neq Tr_F$ but are $r$-equivalent. Consider the diagram
    \begin{equation}\begin{tikzcd}
        F\ar[r,"s_F"]\ar[d,"s_F"] & \yy(A)\ar[d,"\eta_r"] \\
        \yy(A) \ar[r,"\tilde{s_F}=e_F"]\ar[d,"r_F"]& \ss^{-r}\yy(A)\ar[d,equals]\\
        \ss^{-r}F\ar[r,swap,"\ss^{-r}s_F"] & \ss^{-r}\yy(A)
    \end{tikzcd}\end{equation}
    where we find $\tilde{s_F}=1_{\yy(A)}\circ \ss^{-r}s_F \circ r_F= e_F$. Hence $Ts_F=r_{T\yy(A)} \circ \yy(Te_F) \circ s_{TF}=T\yy(e_F) \circ s_{TF}$. By definition $T\yy(e_F)=\ss^{-r}s_{TF}\circ r_{TF}$ and so
    \begin{align*}Ts_F=&T\yy(e_F) \circ s_{TF}\\
    =&\ss^{-r}s_{TF}\circ r_{TF} \circ s_{TF}\\
    =&\eta^{TF}_r \circ s_{TF}
    \end{align*}. Similarly if we consider 
     \begin{equation}\begin{tikzcd}
        \yy(A)\ar[r,"r_F"]\ar[d,equals] & \ss^{-r}F\ar[d,"\ss^{-r}s_F"] \\
        \yy(A) \ar[r,"\tilde{r_F}=e_F"]\ar[d,"\eta_r"]& \ss^{-r}\yy(A)\ar[d,"\ss^{-r}r_F"]\\
        \ss^{-r}\yy(A)\ar[r,swap,"\ss^{-r}r_F"] & \ss^{-2r}F
    \end{tikzcd}\end{equation}
    with $\tilde{r}_F=\ss^{-r}s_F \circ r_F \circ 1_{\yy(A)}= e_F$. Hence we have 
    \begin{align*}
        Tr_F=&r_{T\ss^{-r}F} \circ \yy(Te_F) \circ s_{T\yy(A)}\\
        =& r_{\ss^{-r}TF} \circ \ss^{-r}s_{TF}\circ r_{TF}\\
        =& \ss^{-r}r_{TF}\circ \ss^{-r}s_{TF}\circ r_{TF}\\
        =& \eta_r^{\ss^{-r}TF}\circ r_{TF}
        \end{align*}
        It follows that 
        \begin{equation}
            \begin{tikzcd}
               \ss^{-r}\ar[dr,"Tr_F"] T\yy(A)\\
                TF\ar[u,"Ts_F"] \ar[r,"\eta_{3r}"]& \ss^{-3r}TF
            \end{tikzcd}
        \end{equation}
        defines a $3r$-retract.
\end{rmk}

\begin{lemma}\label{idemext}
      Let $\cc$ be a TPC, $\begin{tikzcd}
             A\ar[r,"u"] & B\ar[r,"v"] & C\ar[r,"w"] & TA
        \end{tikzcd}$ be an exact triangle in $\cc_0$ and $e_A,e_B$ be $r$-idempotents on $A$ and $B$ respectively. If the following commutes

        \begin{equation}\begin{tikzcd}
            A\ar[d,"e_A"]\ar[r,"u"] & B\ar[r,"v"] \ar[d," e_B"]& C\ar[rr,"w"] & &TA\ar[d,"T( e_A)"]\\
            \ss^{-r}A\ar[r,"\ss^{-r}u"] & \ss^{-r}B\ar[r,"\ss^{-r}v"] & \ss^{-r}C\ar[rr,"\ss^{-r}w"] & &\ss^{-r}TA
        \end{tikzcd}\end{equation}
        then there exists a $3r$-idempotent $e_C:C \to \ss^{-3r}C$, such that

          \begin{equation}\begin{tikzcd}
            A\ar[d,"\eta_{2r}\circ e_A"]\ar[r,"u"] & B\ar[r,"v"] \ar[d,"\eta_{2r} \circ e_B"]& C\ar[rr,"w"]\ar[d,dotted,"e_C"] & &TA\ar[d,"T(\eta_{2r}\circ e_A)"]\\
            \ss^{-3r}A\ar[r,"\ss^{-3r} u"] & \ss^{-3r}B\ar[r,"\ss^{-3r}v"] & \ss^{-3r}C\ar[rr,"\ss^{-3r}w"] & &\ss^{-3r}TA
        \end{tikzcd}\end{equation}
        commutes.

        \begin{proof}
             The proof follows similarly to that of Lemma 1.13 of \cite{BSc}. Firstly, since $\cc_0$ is triangulated there exists a map $k: C \to \ss^{-r}C$ which completes the diagram. Consider now the diagram
              \begin{equation}\begin{tikzcd}
            A\ar[d," e_A"]\ar[r,"u"] & B\ar[r,"v"] \ar[d,"e_B"]& C\ar[rr,"w"]\ar[d,dotted,"k"] & &TA\ar[d,"Te_A"]\\
            \ss^{-r}A\ar[r,"\ss^{-r}u"]\ar[d,"\ss^{-r}e_A"] & \ss^{-r}B\ar[d,"\ss^{-r} e_B"]\ar[r,"\ss^{-r}v"] & \ss^{-r}C\ar[d,dotted,"\ss^{-r}k"]\ar[rr,"\ss^{-r}w"] & &\ss^{-r}TA\ar[d,"\ss^{-r}Te_A"]\\
             \ss^{-2r}A\ar[r,"\ss^{-2r}u"] & \ss^{-2r}B\ar[r,"\ss^{-2r}v"] & \ss^{-2r}C\ar[rr,"\ss^{-2r}w"] & &\ss^{-2r}TA
        \end{tikzcd}\end{equation}
         The morphism $\ss^{-r}k \circ k$ then satisfies 
            \begin{align*}
                \ss^{-r}k \circ k \circ v = &\ss^{-2r}v \circ \ss^{-r}e_B\circ e_B\\
                =&\ss^{-2r}v \circ \eta_r \circ e_B\\
                =& \eta_r \circ \ss^{-r}v \circ e_B\\
                =& \eta_r \circ k \circ v
            \end{align*}
            Similarly $\ss^{-2r}w\circ \ss^{-r}k \circ k= \ss^{-2r}w \circ \eta_r \circ k$. Hence setting $z:=\ss^{-r}k \circ k- \eta_r \circ k$, we have $z \circ v =0$ and $\ss^{-2r}w \circ z =0 $. Since the original triangle is exact, we know $w \circ v =0$, and we  can factor $z$ as $z= z' \circ w$. Hence $\ss^{-2r}z \circ z = \ss^{2r}z' \circ (\ss^{-2r}w \circ z) \circ w =0$. Next note that 
            \begin{align*}
    \ss^{-r}z \circ k =& (\ss^{-2r}k \circ \ss^{-r}k - \eta_r \circ \ss^{-r}k) \circ k\\
    =&\ss^{-2r}k \circ (\ss^{-r}k \circ k - \eta_{r} \circ k)\\
    =& \ss^{-2r}k \circ z 
            \end{align*}
and define $e_C:= \eta_{2r}\circ k + \eta_r \circ z - 2 \ss^{-2r}k \circ z$. One then finds that (using commutativity of $\ss^{-2r}k$ and $\ss^{-r}z$ and that $\ss^{-2r}z\circ z=0$)
\begin{align*}
    \ss^{-3r}e_C \circ e_C=&\ss^{-3r}(\eta_{2r}\circ k + \eta_r \circ z - 2 \ss^{-2r}k \circ z) \circ (\eta_{2r}\circ k + \eta_r \circ z - 2 \ss^{-2r}k \circ z)\\
    =&\eta_{2r}\circ \ss^{-3r} k \circ (\eta_{2r} \circ k + \eta_r \circ z - 2\ss^{-2r}k \circ z)\\
    &+\eta_r \circ \ss^{-3r}z \circ (\eta_{2r}\circ k)\\
    &-2\ss^{-5r}k \circ \ss^{-3r}z\circ (\eta_{2r} \circ k)\\
    =&\eta_{2r} \circ \eta_{2r} \circ \ss^{-r}k \circ k + \eta_r \circ \ss^{-2r}k \circ z - 2\ss^{-2r}(\ss^{-r}k \circ k) \circ z\\
     &+\eta_r \circ \ss^{-3r}z \circ (\eta_{2r}\circ k)\\
    &-2\ss^{-5r}k \circ \ss^{-3r}z\circ (\eta_{2r} \circ k)\\
    =&\eta_{2r} \circ \eta_{2r} \circ (z+\eta_r \circ k) + \eta_r \circ \ss^{-2r}k \circ z - 2\ss^{-2r}(z+ \eta_r \circ k) \circ z\\
     &+\eta_r \circ \ss^{-3r}z \circ (\eta_{2r}\circ k)\\
    &-2\eta_{2r}\circ \ss^{-3r}k \circ \ss^{-r}z\circ k\\
    &=\eta_{3r}\circ \eta_{r}\circ z + \eta_{3r}\circ \eta_{2r}\circ k -2 \eta_{2r}\circ \ss^{-3r}k \circ \ss^{-r}z \circ k\\
     &=\eta_{3r}\circ \eta_{r}\circ z + \eta_{3r}\circ \eta_{2r}\circ k -2 \eta_{2r}\circ \ss^{-2r}z \circ \ss^{-r}k \circ k\\
      &=\eta_{3r}\circ \eta_{r}\circ z + \eta_{3r}\circ \eta_{2r}\circ k -2 \eta_{2r}\circ \ss^{-2r}z \circ \eta_{r}\circ k\\
      &=\eta_{3r}\circ (\eta_{r}\circ z + \circ \eta_{2r}\circ k -2 \ss^{-r}z \circ \circ k)\\
      &=\eta_{3r}\circ e_C
      \end{align*}

            Hence $e_C$ is a $3r$-idempotent on $C$. Moreover we have 
            \begin{align*}
e_C \circ v=& [\eta_{2r}\circ k +\eta_{r}\circ z -2 (\ss^{-2r}k \circ z) ]\circ v\\
=& \eta_{2r}\circ k \circ v \\
=&\eta_{2r}\circ \ss^{-r}v \circ e_B\\
=&\ss^{-3r}v \circ \eta_{2r}\circ e_B\\
\end{align*}
Similarly we have $T(\eta_{2r}\circ e_A)\circ w = \ss^{-3r}w \circ e_C$ and thus 
          \begin{equation}\begin{tikzcd}
            A\ar[d,"\eta_{2r}\circ e_A"]\ar[r,"u"] & B\ar[r,"v"] \ar[d,"\eta_{2r} \circ e_B"]& C\ar[rr,"w"]\ar[d,dotted,"e_C"] & &TA\ar[d,"T(\eta_{2r}\circ e_A)"]\\
            \ss^{-3r}A\ar[r,"\ss^{-3r} u"] & \ss^{-3r}B\ar[r,"\ss^{-3r}v"] & \ss^{-3r}C\ar[rr,"\ss^{-3r}w"] & &\ss^{-3r}TA
        \end{tikzcd}\end{equation}
        commutes.
        \end{proof}
\end{lemma}

    \begin{cor}\label{idemextcor}
        Let $\cc$ be a TPC and $\begin{tikzcd}
             A\ar[r,"u"] & B\ar[r,"v"] & C\ar[r,"w"] & TA
        \end{tikzcd}$ be an exact triangle in $\cc_0$. Assume $e_A$ is an $r$-idempotent on $A$ and $e_B$ is an $s$-idempotent on $B$ such that the following commutes

        \begin{equation}\begin{tikzcd}
            A\ar[d,"\eta_s \circ e_A"]\ar[r,"u"] & B\ar[r,"v"] \ar[d,"\eta_r \circ e_B"]& C\ar[rr,"w"] & &TA\ar[d,"T(\eta_{s}\circ e_A)"]\\
            \ss^{-(r+s)}A\ar[r,"\ss^{-(r+s)u}"] & \ss^{-(r+s)}B\ar[r,"\ss^{-(r+s)}v"] & \ss^{-(r+s)}C\ar[rr,"\ss^{-(r+s)}w"] & &\ss^{-(r+s)}TA
        \end{tikzcd}\end{equation}
        Then there exists a $3(r+s)$-idempotent $e_C:C \to \ss^{-3(r+s)}C$, such that

          \begin{equation}\begin{tikzcd}
            A\ar[d,"\eta_{2(r+s)}\circ \eta_s \circ e_A"]\ar[r,"u"] & B\ar[r,"v"] \ar[d,"\eta_{2(r+s)}\circ \eta_r \circ e_B"]& C\ar[rr,"w"]\ar[d,dotted,"e_C"] & &TA\ar[d,"T(\eta_{2(r+s)}\eta_{s}\circ e_A)"]\\
            \ss^{-3(r+s)}A\ar[r,"\ss^{-3(r+s)u}"] & \ss^{-3(r+s)}B\ar[r,"\ss^{-3(r+s)}v"] & \ss^{-3(r+s)}C\ar[rr,"\ss^{-3(r+s)}w"] & &\ss^{-3(r+s)}TA
        \end{tikzcd}\end{equation}
        commutes.  
\begin{proof}
    Use Lemma \ref{idemext} on the $(r+s)$-idempotens $\eta_s \circ e_A$ and $\eta_r \circ e_B$.
\end{proof}
    \end{cor}

\begin{cor}
     Let $\cc$ be a TPC and $\begin{tikzcd}
             A\ar[r,"u"] & B\ar[r,"v"] & C\ar[r,"w"] & TA
        \end{tikzcd}$ be an exact triangle in $\cc_0$. Assume $e_A$ is an $r$-idempotent on $A$ and $e_B$ is an $s$-idempotent on $B$ such that the following commutes

        \begin{equation}\begin{tikzcd}
            A\ar[d,"\eta_{m-r} \circ e_A"]\ar[r,"u"] & B\ar[r,"v"] \ar[d,"\eta_{m-s} \circ e_B"]& C\ar[rr,"w"] & &TA\ar[d,"T(\eta_{s}\circ e_A)"]\\
            \ss^{-m}A\ar[r,"\ss^{-m}u"] & \ss^{-m}B\ar[r,"\ss^{-m}v"] & \ss^{-m}C\ar[rr,"\ss^{-m}w"] & &\ss^{-m}TA
        \end{tikzcd}\end{equation}
where $m:=\max\{r,s\}$. Then there exists a $3m$-idempotent $e_C:C \to \ss^{-3m}C$, such that

          \begin{equation}\begin{tikzcd}
            A\ar[d,"\eta_{2m}\circ \eta_{m-r} \circ e_A"]\ar[r,"u"] & B\ar[r,"v"] \ar[d,"\eta_{2m}\circ \eta_{m-s} \circ e_B"]& C\ar[rr,"w"]\ar[d,dotted,"e_C"] & &TA\ar[d,"T(\eta_{2m}\circ \eta_{m-r}\circ e_A)"]\\
            \ss^{-3m}A\ar[r,"\ss^{-3m}u"] & \ss^{-3m}B\ar[r,"\ss^{-3m}v"] & \ss^{-3m}C\ar[rr,"\ss^{-3m}w"] & &\ss^{-3m}TA
        \end{tikzcd}\end{equation}
        commutes.  
\end{cor}

Take a morphism $u: F \to G$ where, $F$ and $G$ are given by weighted retract diagrams 
\begin{equation}\begin{tikzcd}
          \yy(A) \ar[dr,"r_F"]& & & \yy(B)\ar[dr,"r_G"]&\\
          F\ar[u,"s_F"]\ar[r,"\eta_r"] & \ss^{-r}F & & G\ar[u,"s_G"]\ar[r,"\eta_s"] & \ss^{-s}G
      \end{tikzcd}\end{equation}
we first complete to the diagram 
\begin{equation}\begin{tikzcd}
        F\ar[r,"u"]\ar[d,"s_F"] & G\ar[d,"\ss^{-r}s_G\circ \eta_r"] \\
        \yy(A) \ar[r,"\tilde{u}"]\ar[d,"\ss^{-s}r_F\circ \eta_s"]& \ss^{-r}\yy(B)\ar[d,"\ss^{-r}r_G"]\\
        \ss^{-(r+s)}F\ar[r,swap,"\ss^{-(r+s)}u"] & \ss^{-(r+s)}G
    \end{tikzcd}\end{equation}
where $\tilde{u}= \ss^{-r}s_G \circ \ss^{-r}u \circ r_F$. Completing $\tilde{u}$ to an exact triangle using the triangulated structure of $\cc_0$ we get a diagram
\begin{equation}\begin{tikzcd}
        F\ar[r,"u"]\ar[d,"s_F"] & G\ar[d,"\ss^{-r}s_G\circ \eta_r"] & &\\
        \yy(A) \ar[r,"\tilde{u}"]\ar[d,"\ss^{-s}r_F\circ \eta_s"]& \ss^{-r}\yy(B)\ar[d,"\ss^{-r}r_G"]\ar[r,"v"]&\yy(C) \ar[r,"w"]&T\yy(A)\\
        \ss^{-(r+s)}F\ar[r,swap,"\ss^{-(r+s)}u"] & \ss^{-(r+s)}G& &
    \end{tikzcd}\end{equation}

Using Lemma \ref{idemext} on the diagram 

\begin{equation}\begin{tikzcd}
    \yy(A) \ar[r,"\tilde{u}"]\ar[d,"\eta_s \circ e_F"]& \ss^{-r}\yy(B)\ar[r,"v"]\ar[d,"\ss^{-r}(\eta_r \circ e_G)"] & \yy(C)\ar[r,"w"] & T\yy(A)\ar[d,"T(\eta_s \circ e_F)"]\\
    \ss^{-(r+s)}\yy(A)\ar[r,"   \ss^{-(r+s)}\tilde{u}"] & \ss^{-(2r+s)}\yy(B)\ar[r,"   \ss^{-(r+s)}v"]  & \ss^{-(r+s)}\yy(C) \ar[r,"   \ss^{-(r+s)}w"]& T\ss^{-(r+s)}\yy(A)
\end{tikzcd}\end{equation}

we obtain a $3(r+s)$-idempotent $e_C: \yy(C) \to \ss^{-3(r+s)}\yy(C)$, that makes the following commute,

\begin{equation}\label{ecidem}\begin{tikzcd}
    \yy(A) \ar[r,"\tilde{u}"]\ar[d,"\eta_{2(r+s)}\circ \eta_s \circ e_F"]& \ss^{-r}\yy(B)\ar[r,"v"]\ar[d,"\eta_{2(r+s)}\circ \ss^{-r}(\eta_r \circ e_G)"] & \yy(C)\ar[r,"w"] \ar[d,dotted,"e_C"]& T\yy(A)\ar[d,"\eta_{2(r+s)}\circ T(\eta_s \circ e_F)"]\\
    \ss^{-3(r+s)}\yy(A)\ar[r,"   \ss^{-3(r+s)}\tilde{u}"] & \ss^{-3(r+s)}\ss^{-r}\yy(B)\ar[r,"   \ss^{-3(r+s)}v"]  & \ss^{-3(r+s)}\yy(C) \ar[r,"   \ss^{-3(r+s)}w"]& T\ss^{-3(r+s)}\yy(A)
\end{tikzcd}\end{equation}

$e_C$ then splits in $\text{Split}\pp(\cc)_0$, assume as $H\xrightarrow{s_H} \yy(C) \xrightarrow{r_H}\ss^{-3(r+s)}H$, and so consider;

\begin{equation}\begin{tikzcd}
        F\ar[r,"u"]\ar[d,"s_F"] & G\ar[d,"\ss^{-r}s_G\circ \eta_r"] & H\ar[d,"s_H"]&\\
        \yy(A) \ar[r,"\tilde{u}"]\ar[d,"\eta_{2(r+s)} \circ \ss^{-s}r_F\circ \eta_s"]& \ss^{-r}\yy(B)\ar[d,"\eta_{2(r+s)}\circ \ss^{-r}r_G"]\ar[r,"v"]&\yy(C)\ar[d,"r_H"] \ar[r,"w"]&T\yy(A)\\
        \ss^{-3(r+s)}F\ar[r,swap,"\ss^{-3(r+s)}u"] & \ss^{-3(r+s)}G& \ss^{-3(r+s)} H&
    \end{tikzcd}\end{equation}

We cannot construct a map from $\text{Hom}_{\cc_0}(G,H)$ making the two squares with a single missing edge commute. However we do obtain a map $G \to \ss^{-3(r+s)}H$ given by $r_H \circ v \circ \ss^{-r}s_G \circ \eta_r$. Set $\tilde{v}=r_H \circ v \circ \ss^{-r}s_G \circ \eta_r$ then one can check that the following commutes

\begin{equation}\begin{tikzcd}
        F\ar[r,"u"]\ar[d,"s_F"] & G\ar[d,"\ss^{-r}s_G\circ \eta_r"]\ar[r,"\tilde{v}"] & \ss^{-3(r+s)}H\ar[d,"\ss^{-3(r+s)}s_H"]& &\\
        \yy(A) \ar[r,"\tilde{u}"]\ar[d,"\eta_{2(r+s)} \circ \ss^{-s}r_F\circ \eta_s"]& \ss^{-r}\yy(B)\ar[d,"\eta_{2(r+s)}\circ \ss^{-r}r_G"]\ar[r,"\eta_{3(r+s)}\circ v"]&\ss^{-3(r+s)}\yy(C)\ar[d,"\ss^{-3(r+s)}r_H"] \ar[rr,"\ss^{-3(r+s)}w"]&&\ss^{-3(r+s)}T\yy(A)\\
        \ss^{-3(r+s)}F\ar[r,swap,"\ss^{-3(r+s)}u"] & \ss^{-3(r+s)}G\ar[r,"\ss^{-3(r+s)}\tilde{v}"]&  \ss^{-6(r+s)} H& &
    \end{tikzcd}\end{equation}

For the top square we have 
\begin{align*}
  \ss^{-3(r+s)}s_H \circ \tilde{v}  =&\ss^{-3(r+s)}s_H \circ r_H \circ v \circ \ss^{-r}s_G \circ \eta_r\\
  =&e_C \circ v \circ \ss^{-r}s_G \circ \eta_r\\
  =&\ss^{-3(r+s)}v \circ \eta_{2(r+s)}\circ \eta_{2r} \circ e_G \circ s_G\\
  =&\ss^{-3(r+s)}v \circ \eta_{2(r+s)}\circ \eta_{2r}\circ \eta_s \circ s_G\\
  =&\eta_{3(r+s)}\circ v \circ \ss^{-r}s_G \circ \eta_r.
\end{align*}
And for the bottom square we have
\begin{align*}
    \ss^{-3(r+s)}\tilde{v} \circ \eta_{2(r+s)}\circ \ss^{-r}r_G=& \eta_{2(r+s)}\circ \ss^{-(r+s)}[r_H \circ v \circ \ss^{-r}s_G \circ \eta_r]\circ \ss^{-r}r_G\\
    =&\eta_{2(r+s)} \circ \ss^{-(r+s)}r_H \circ \ss^{-(r+s)}v\circ \eta_r \circ \ss^{-r}e_G\\
    =&\ss^{-3(r+s)}r_H \circ \ss^{-3(r+s)}v \circ \eta_{2(r+s)}\circ \eta_r \circ \ss^{-r}e_G\\
    =&\ss^{-3(r+s)}r_H \circ e_C \circ v\\
    =&\eta_{3(r+s)}\circ r_H \circ v\\
    =& \ss^{-3(r+s)}\circ \eta_{3(r+s)}\circ v
\end{align*}

Next we use the construction in remark \ref{diagconstructing} and let $\tilde{w}=\ss^{-3(r+s)}(\eta_{2(r+s)}\circ \ss^{-s}r_{TF} \circ \eta_s \circ w \circ s_H)$ and then complete the diagram to

\begin{equation}\begin{tikzcd}
        F\ar[r,"u"]\ar[d,"s_F"] & G\ar[d,"\ss^{-r}s_G\circ \eta_r"]\ar[r,"\tilde{v}"] & \ss^{-3(r+s)}H\ar[d,"\ss^{-3(r+s)}s_H"]\ar[rrr,"\tilde{w}"]& & & \ss^{-6(r+s)}TF\ar[d,"\ss^{-6(r+s)}s_{TF}"]\\
        \yy(A) \ar[r,"\tilde{u}"]\ar[d,"\eta_{2(r+s)} \circ \ss^{-s}r_F\circ \eta_s"]& \ss^{-r}\yy(B)\ar[d,"\eta_{2(r+s)}\circ \ss^{-r}r_G"]\ar[r,"\eta_{3(r+s)}\circ v"]&\ss^{-3(r+s)}\yy(C)\ar[d,"\ss^{-3(r+s)}r_H"] \ar[rrr,"\eta_{3(r+s)}\circ \ss^{-3(r+s)}w"]& & &\ss^{-6(r+s)}T\yy(A)\ar[d,"\ss^{-6(r+s)}(\eta_{2(r+s)}\circ \ss^{-s}r_{TF} \circ \eta_s)"]\\
        \ss^{-3(r+s)}F\ar[r,swap,"\ss^{-3(r+s)}u"] & \ss^{-3(r+s)}G\ar[r,"\ss^{-3(r+s)}\tilde{v}"]&  \ss^{-6(r+s)} H\ar[rrr,"\ss^{-3(r+s)}\tilde{w}"]& &&\ss^{-9(r+s)}TF
    \end{tikzcd}\end{equation}\label{triangle}

Checking the last two squares commute; the top right square gives

\begin{align*}
    \ss^{-6(r+s)}s_{TF} \circ \tilde{w}=& \ss^{-6(r+s)}s_{TF} \circ \ss^{-3(r+s)}[\eta_{2(r+s)}\circ \ss^{-s}r_{TF} \circ \eta_s \circ w \circ s_H]\\
    =&\ss^{-3(r+s)}[\ss^{-3(r+s)}s_{TF} \circ \eta_{2(r+s)}\circ \ss^{-s}r_{TF} \circ \eta_s \circ w \circ s_H]\\
    =&\ss^{-3(r+s)}[\eta_{2(r+s)}\circ\ss^{-(r+s)}s_{TF} \circ  \ss^{-s}r_{TF} \circ \eta_s \circ w \circ s_H]\\
    =&\ss^{-3(r+s)}[\eta_{2(r+s)}\circ\ss^{-s}e_{TF} \circ \eta_s \circ w \circ s_H]\\
   =&\ss^{-3(r+s)}[ \eta_{2(r+s)}\circ T(\eta_s \circ e_F)  \circ w \circ s_H]\\
    =&\ss^{-3(r+s)}[ \ss^{-3(r+s)}w \circ e_C \circ s_H]\\
    =&\ss^{-6(r+s)}w \circ \eta_{3(r+s)}\circ \ss^{-3(r+s)}s_H\\
    =&\eta_{3(r+s)}\circ \ss^{-3(r+s)}w \circ \ss^{-3(r+s)}s_H
\end{align*}
and the bottom right square gives

\begin{align*}
    \ss^{-3(r+s)}\tilde{w} \circ \ss^{-3(r+s)}r_H=&\ss^{-3(r+s)}\big(\ss^{-3(r+s)}[(\eta_{2(r+s)}\circ \ss^{-s}r_{TF} \circ \eta_s) \circ w \circ s_H]\circ r_H\big) \\
    =& \ss^{-3(r+s)}\big(\ss^{-3(r+s)}[(\eta_{2(r+s)}\circ \ss^{-s}r_{TF} \circ \eta_s) \circ w] \circ e_C\big) \\
    =& \ss^{-3(r+s)}\big(\ss^{-3(r+s)}[(\eta_{2(r+s)}\circ \ss^{-s}r_{TF} \circ \eta_s)] \circ \eta_{2(r+s)}\circ T(\eta_s \circ e_F) \circ w\big) \\
    =&\eta_{4(r+s)}\circ \ss^{-3(r+s)}\big(\ss^{-(r+s)} (\eta_s \circ r_{TF} ) \circ T(\eta_s \circ e_F) \circ w\big)\\
    =&\eta_{4(r+s)}\circ \ss^{-3(r+s)}\big(\ss^{-(r+s)} (\eta_s \circ r_{TF} ) \circ \eta_s \circ Te_F) \circ w\big)\\
    =&\eta_{4(r+s)}\circ \ss^{-3(r+s)}\big(\eta_s \circ \ss^{-r} (\eta_s \circ r_{TF} ) \circ Te_F) \circ w\big)\\  =&\eta_{4(r+s)}\circ \ss^{-3(r+s)}\big(\eta_{2s} \circ \ss^{-r}r_{TF}  \circ Te_F \circ w\big)\\
=&\eta_{4(r+s)}\circ \ss^{-3(r+s)}\big(\eta_{2s} \circ \ss^{-r}r_{TF} \circ \eta_r \circ w\big)\\
=&\eta_{4(r+s)}\circ \ss^{-3(r+s)}\big(\eta_{r+s} \circ \eta_s \circ r_{TF}  \circ w\big)\\
     =&\eta_{5(r+s)} \circ \eta_s \circ \ss^{-3(r+s)}r_{TF} \circ \ss^{-3(r+s)}w\\
     =& \eta_{2(r+s)}\circ \ss^{-6(r+s)}(\eta_{s}\circ r_{TF})\circ \eta_{3(r+s)}\circ \ss^{-3(r+s)}w\\
     =& \ss^{-6(r+s)}(\eta_{2(r+s)}\circ \ss^{-s}r_{TF} \circ \eta_s)\circ \eta_{3(r+s)}\circ \ss^{-3(r+s)}w
\end{align*}
It is important to note that the diagram is not strictly a morphism of triangles as $Tr_F \neq r_{TF}$ and $Ts_F\neq s_{TF}$.

 \begin{rmk}
        If $r=s=0$ then in the above diagram we recover the completion of $u$ to an exact triangle in $\text{Split}(\cc_0)$ the usual idempotent completion of the zero level category.
    \end{rmk}

\begin{rmk}
    The central triangle \begin{equation}\begin{tikzcd}
         \yy(A) \ar[r,"\tilde{u}"]& \ss^{-r}\yy(B)\ar[r,"\eta_{3(r+s)}\circ v"]&\ss^{-3(r+s)}\yy(C) \ar[rrr,"\eta_{3(r+s)}\circ \ss^{-3(r+s)}w"]& & &\ss^{-6(r+s)}T\yy(A)
    \end{tikzcd}\end{equation}
    is the image of a strict exact triangle of weight $6(r+s)$ under the embedding $\yy$. Indeed one can check the following commutes

\begin{equation}\begin{tikzcd}
        & & \ss^{3(r+s)}\yy(C)\ar[d,"\eta_{3(r+s)}"]\ar[drrr,"\ss^{6(r+s)}(\eta_{3(r+s)}\circ \ss^{-3(r+s)}w)"]&\\
         \yy(A) \ar[r,"\tilde{u}"]& \ss^{-r}\yy(B)\ar[dr,swap,"\eta_{3(r+s)}\circ v"]\ar[r,"v"]&\yy(C)\ar[rrr,"w"]\ar[d,"\eta_{3(r+s)}"] & &&T\yy(A)\\
         & &\ss^{-3(r+s)}\yy(C) \ar[rrr,"\eta_{3(r+s)}\circ \ss^{-3(r+s)}w"]& & &\ss^{-6(r+s)}T\yy(A)
    \end{tikzcd}\end{equation}
    
\end{rmk}

    Assume we have a triangle 
    \begin{equation}
       \Delta:= \begin{tikzcd}
            F\ar[r] & G\ar[r] & H\ar[r] & \ss^{-r}TF
        \end{tikzcd}
    \end{equation}
    then we call $\Delta$ an \textbf{exact $(r,s)$-triangle} in $\text{Split}_\pp(\cc)_0$ if it is an $s$-retract of a representable strict exact triangle of weight $r$. Explicitly we require that there exists a diagram 
    \begin{equation}
        \begin{tikzcd}
            F\ar[r]\ar[d,"s"] & G\ar[r] \ar[d]& H\ar[r]\ar[d] & \ss^{-r}TF\ar[d,"s'"]\\
             \yy(A)\ar[d,"r"]\ar[r]&\yy(B)\ar[d]\ar[r]\ar[d] &\yy(C)\ar[d]\ar[r] &\ss^{-r}T\yy(A)\ar[d,"r'"] \\
              \ss^{-s}F\ar[r] & \ss^{-s}G\ar[r] & \ss^{-s}H\ar[r] & \ss^{-(r+s)}TF
        \end{tikzcd}
    \end{equation}
    such that composition along vertical arrows gives $\eta_s$. Furthermore we do not require $r'=Tr$ and $s'=Ts$ but only that $\ss^{-s}s'\circ r'= T(\ss^{-s}s \circ r)$. We have already seen that any morphism $u:F \to G$ can be completed to an exact $(a,b)$-triangle, with $a=6(r+s)$ and $b=3(r+s)$ where $F<_r \yy(A)$ and $G<_s \yy(B)$ are the sizes of the retractions defining $F$ and $G$.

    \begin{rmk}
    Any strict exact triangle of weight $r$ in $\cc_0$, can be viewed as an exact $(r,0)$-triangle in $\text{Split}_\pp(\cc)_0$, via 
       \begin{equation}\begin{tikzcd}
           \yy(A)\ar[d,equals]\ar[r,"\bar{u}"] & \yy(B)\ar[d,equals]\ar[r,"\bar{v}"] & \yy(C)\ar[d,equals]\ar[r,"\bar{w}"] & \ss^{-r}T\yy(A)\ar[d,equals]\\
          \yy(A)\ar[d,equals]\ar[r,"\bar{u}"] & \yy(B)\ar[d,equals]\ar[r,"\bar{v}"] & \yy(C)\ar[d,equals]\ar[r,"\bar{w}"] & \ss^{-r}T\yy(A)\ar[d,equals]\\
          \yy(A)\ar[r,"\bar{u}"] & \yy(B)\ar[r,"\bar{v}"] & \yy(C)\ar[r,"\bar{w}"] & \ss^{-r}T\yy(A)
    \end{tikzcd}\end{equation}
    \end{rmk}
\begin{rmk}
    The identity $1_F: F \to F$ for any $F<_r\yy(A)$ completes to an exact $(0,r)$-triangle: $F \xrightarrow{=}F \to 0 \to TF$ via:

\begin{equation}\begin{tikzcd}
    F\ar[r,equals]\ar[d,"s_F"] & F\ar[r]\ar[d,"s_F"] & 0\ar[r]\ar[d] & TF\ar[d,"s_{TF}"]\\
    \yy(A)\ar[d,"r_F"] \ar[r,equals] & \yy(A)\ar[r]\ar[d,"r_F"] & 0\ar[r]\ar[d] & T\yy(A)\ar[d,"r_{TF}"]\\
    \ss^{-r}F\ar[r,equals] & \ss^{-r}F\ar[r] & 0\ar[r] & \ss^{-r}TF
\end{tikzcd}\end{equation}

\end{rmk}

\begin{rmk}
    The map $\eta_t: F \to \ss^{-t}F$ where $F<_r \yy(A)$, then we can complete this to an exact $(t,r)$-triangle, $F \xrightarrow{\eta_t}\ss^{-t}F\to 0 \to \ss^{-t}TF$ via
 \begin{equation}\begin{tikzcd}
    F\ar[r,"\eta_t"]\ar[d,"s_F"] & \ss^{-t}F\ar[r]\ar[d,"\ss^{-t}s_F"] & 0\ar[r]\ar[d] & \ss^{-t}TF\ar[d,"\ss^{-t}s_{TF}"]\\
    \yy(A)\ar[d,"r_F"] \ar[r,"\eta_t"] & \ss^{-t}\yy(A)\ar[r]\ar[d,"\ss^{-t}r_F"] & 0\ar[r]\ar[d] & \ss^{-t}T\yy(A)\ar[d,"\ss^{-t}r_{TF}"]\\
    \ss^{-r}F\ar[r,"\eta_t"] & \ss^{-r}F\ar[r] & 0\ar[r] & \ss^{-(r+t)}TF
\end{tikzcd}\end{equation}

\end{rmk}

 It is clear that $\text{Split}_\pp(\cc)_0$ is not triangulated. We can only complete morphisms to retracts of strict exact triangles with weight dependent on the size of the retractions defining the objects. Thus it seems to be that we need to relax the assumption that the zero level category be triangulated in the definition of a TPC. Furthermore it points towards including sizes on objects and not just morphisms.

\begin{lemma}
    Every exact triangle in $\text{Split}_\pp(\cc)_\infty$ is represented by some $(r,s)$-triangle.
\end{lemma}

\begin{proof}
Assume that we have an exact triangle defined via the diagram

\begin{equation}
            \begin{tikzcd}
            F\ar[r,"u"]\ar[d,"s_F"] & G\ar[r,"v"] \ar[d,"s_G"]& H\ar[r,"w"]\ar[d,"s_H"] & TF\ar[d,"s_{TF}"]\\
             \yy(A)\ar[d,"r_F"]\ar[r,"\bar{u}"]&\yy(B)\ar[d,"r_G"]\ar[r,"\bar{v}"] &\yy(C)\ar[d,"r_H"]\ar[r,"\bar{w}"] &T\yy(A)\ar[d,"r_{TF}"] \\
              F\ar[r,"u"] & G\ar[r,"v"] & H\ar[r,"w"] & TF
        \end{tikzcd}
\end{equation}
 This diagram is the image of a diagram in $\text{Split}_\pp(\cc)$. Denote $u'\in \text{Hom}_{\text{Split}_\pp(\cc)}(F,G)$ to be a representative of $u$, and denote similarly representatives of the other morphisms ($-'$ represents $-$), we then have

\begin{align*}
    \lceil u' \rceil + \lceil s_G' \rceil = \lceil \bar{u}' \rceil +\lceil s_F' \rceil 
\end{align*}

and similarly for all other squares, where $\lceil - \rceil$ denotes the persistence module level (shift) of the morphism. Also note we must have

\begin{equation}\lceil r_F' \rceil + \lceil s'_F \rceil = \lceil r_G' \rceil + \lceil s_G' \rceil\end{equation}
and likewise for $H$ and $TF$. We start by shifting the central triangle to obtain (the image of) a strict exact triangle in $\text{Split}_\pp(\cc)_0$, 

\begin{equation}\begin{tikzcd}
    \yy(A) \ar[r,"\bar{u}''"]& \ss^{-\lceil \bar{ u}' \rceil}\yy(B)\ar[r,"\bar{v}''"] & \ss^{-(\lceil \bar{u}' \rceil+ \lceil \bar{v}' \rceil)}\yy(C) \ar[r,"\bar{w}''"]& \ss^{-(\lceil \bar{u}' \rceil +\lceil \bar{v}'\rceil + \lceil \bar{w}' \rceil)}T\yy(A)
\end{tikzcd}\end{equation}

where
\begin{equation}\bar{u}''=  \eta_{0,-\lceil \bar{u}' \rceil}\circ \bar{u}'\end{equation}
\begin{equation}\bar{v}'' = \eta_{0,-(\lceil \bar{v}' \rceil + \lceil \bar{u}' \rceil)}\circ \bar{v}'\circ  \eta_{-\lceil \bar{u}' \rceil,0}\end{equation} 
and similarly for $\bar{w}''$. This strict exact triangle represents the central triangle in $\text{Split}(\cc_\infty)$.  Next, we can choose representatives of $s_-$ and $r_-$ such that $r'_- \circ s'_- = 1'_- $, i.e., $[r_F' \circ s_F']_{\text{Split}_\pp(\cc)_\infty}=1_F$. After shifting we can then represent the retractions as 
\begin{equation}\begin{tikzcd}
    F\ar[d,"s_F''"]\\
    \ss^{-\lceil s_F' \rceil }\yy(A)\ar[d,"r_F''"]\\
    \ss^{- (\lceil s_F' \rceil + \lceil r_F' \rceil)}F
\end{tikzcd}\end{equation}
a diagram in $\text{Split}_\pp(\cc)_0$. The composition represents $1_F$ and so must be $k$-equivalent to $\eta_{\lceil s'_F \rceil + \lceil r_F' \rceil }$ for some $k\geq 0$. We can therefore assume 

\begin{equation}\begin{tikzcd}
    F\ar[d,"\tilde{s}_F"]\\
    \ss^{-\lceil s_F' \rceil }\yy(A)\ar[d,"\tilde{r}_F"]\\
    \ss^{-p_F}F
\end{tikzcd}\end{equation}
represents the retract in $\text{Split}_\pp(\cc)_0$, where $p_F=\inf\{\lceil s'_F \rceil + \lceil r_F' \rceil +k: \eta_{k}\circ r''_F \circ s''_F= \eta_{k+ \lceil s_F'' \rceil + \lceil r_F'' \rceil} \}$, $\tilde{s}_F=s_F''$ and $\tilde{r}_F= \eta_{k}\circ r_F''$. Using this we obtain a commutative diagram
\begin{equation}\begin{tikzcd}
F \ar[r,"u''"]\ar[d,"s_F''"]& \ss^{-\lceil u' \rceil}G \ar[d,"s_G''"]\ar[r,"v''"]& \ss^{-(\lceil u' \rceil + \lceil v' \rceil)}H \ar[d,"s_H''"]\ar[r,"w''"]& \ss^{-(\lceil u' \rceil +\lceil v' \rceil +\lceil w' \rceil)}TF\ar[d,"(s_{TF})''"]\\
    \ss^{-\lceil s_F' \rceil}\yy(A) \ar[r,"\tilde{u}''"]& \ss^{-(\lceil s_F' \rceil)+\lceil \bar{ u}' \rceil)}\yy(B)\ar[r,"\tilde{v}''"] & \ss^{-(\lceil s_F' \rceil +\lceil \bar{u}' \rceil+ \lceil \bar{v}' \rceil)}\yy(C) \ar[r,"\tilde{w}''"]& \ss^{-(\lceil s_F' \rceil +\lceil \bar{u}' \rceil +\lceil \bar{v}'\rceil + \lceil \bar{w}' \rceil)}T\yy(A)\\
    & & & 
\end{tikzcd}\end{equation}
where 
\begin{align*}
    u''&:= \eta_{0,-\lceil u' \rceil}\circ u'\\
    s_F''&:=\eta_{0,-\lceil s_F' \rceil}\circ s_F'\\
    s_G''&:=\eta_{0,-(\lceil \bar{u}' \rceil + \lceil s_F' \rceil)} \circ s_G' \circ \eta_{-\lceil u' \rceil,0}\\
    \tilde{u}''&:= \eta_{0,-(\lceil \bar{ u}' \rceil+ \lceil s_F' \rceil)} \circ \bar{u}' \circ \eta_{-\lceil s_F' \rceil,0}\\
    v''&:= \eta_{0,-(\lceil u' \rceil + \lceil v' \rceil)}\circ v' \circ \eta_{-\lceil u' \rceil,0}\\
    s_H''&:= \eta_{0,-(\lceil s_F' \rceil +\lceil \bar{u}' \rceil+ \lceil \bar{v}' \rceil)} \circ  s_H'\circ \eta_{-(\lceil u' \rceil + \lceil v' \rceil),0}\\
    \tilde{v}''&:=\eta_{0,-(\lceil s_F' \rceil +\lceil \bar{u}' \rceil+ \lceil \bar{v}' \rceil)}\circ \bar{v}' \circ  \eta_{-(\lceil \bar{ u}' \rceil+ \lceil s_F' \rceil),0}\\
    w''&:= \eta_{0,-(\lceil u' \rceil +\lceil v' \rceil +\lceil w' \rceil)}\circ w'\circ \eta_{-(\lceil u' \rceil + \lceil v' \rceil),0}\\
    \tilde{w}''&:=\eta_{0,-(\lceil s_F' \rceil +\lceil \bar{u}' \rceil +\lceil \bar{v}'\rceil + \lceil \bar{w}' \rceil)}\circ \bar{w}' \circ \eta_{-(\lceil s_F' \rceil +\lceil \bar{u}' \rceil+ \lceil \bar{v}' \rceil),0}\\
    (s_{TF})''&:= \eta_{0,-(\lceil s_F' \rceil +\lceil \bar{u}' \rceil +\lceil \bar{v}'\rceil + \lceil \bar{w}' \rceil)}\circ T s'_F \circ \eta_{-(\lceil u' \rceil +\lceil v' \rceil +\lceil w' \rceil),0}
\end{align*}

Now take, $k$ to be 
\begin{equation}k:=\max\{p_F-(\lceil s_F' \rceil + \lceil r_F' \rceil),p_G-(\lceil s_G' \rceil + \lceil r_G' \rceil),p_H-(\lceil s_H' \rceil + \lceil r_H' \rceil)\}\end{equation}
and set $p:= k+ \lceil r_F' \rceil + \lceil s_F' \rceil$.
Then we can extend the above diagram to
\begin{equation}\begin{tikzcd}
F \ar[r,"u''"]\ar[d,"s_F''"]& \ss^{-\lceil u' \rceil}G \ar[d,"s_G''"]\ar[r,"v''"]& \ss^{-(\lceil u' \rceil + \lceil v' \rceil)}H \ar[d,"s_H''"]\ar[r,"w''"]& \ss^{-(\lceil u' \rceil +\lceil v' \rceil +\lceil w' \rceil)}TF\ar[d,"(s_{TF})''"]\\
    \ss^{-\lceil s_F' \rceil}\yy(A)\ar[d,"r''_F"] \ar[r,"\tilde{u}''"]& \ss^{-(\lceil s_F' \rceil+\lceil \bar{ u}' \rceil)}\yy(B)\ar[d,"r_G''"]\ar[r,"\tilde{v}''"] & \ss^{-(\lceil s_F' \rceil +\lceil \bar{u}' \rceil+ \lceil \bar{v}' \rceil)}\yy(C)\ar[d,"r_H''"] \ar[r,"\tilde{w}''"]& \ss^{-(\lceil s_F' \rceil +\lceil \bar{u}' \rceil +\lceil \bar{v}'\rceil + \lceil \bar{w}' \rceil)}T\yy(A)\ar[d,"(Tr_F)''"]\\
   \ss^{-p}F \ar[r,"\ss^{-p}u''"]&\ss^{-(p+ \lceil u' \rceil)}G \ar[r,"\ss^{-p}v''"]&\ss^{-(p+\lceil u' \rceil + \lceil v' \rceil)}H \ar[r,"\ss^{-p}w''"]&\ss^{-(p+\lceil u' \rceil +\lceil v' \rceil +\lceil w' \rceil)}TF 
\end{tikzcd}\end{equation}
where
\begin{align*}\label{tricomp}
    r''_F&:=\eta_k \circ \eta_{0,-(\lceil s_F' \rceil + \lceil r_F' \rceil)}\circ r_F' \circ \eta_{-\lceil s_F' \rceil,0}\\
    r_G'' &:= \eta_k \circ \eta_{0,-(\lceil r_G' \rceil + \lceil s_G'\rceil+\lceil u' \rceil)} \circ r_G' \circ \eta_{-(\lceil s_F' \rceil+\lceil \bar{ u}' \rceil),0}\\
    r_H''&:=\eta_k \circ \eta_{0,-(\lceil r_H' \rceil + \lceil s_H'\rceil+\lceil u' \rceil+ \lceil v' \rceil)}\circ r_H' \circ \eta_{-(\lceil s_F' \rceil +\lceil \bar{u}' \rceil+ \lceil \bar{v}' \rceil),0}\\
    (Tr_F)''&:=\eta_k \circ \eta_{0,-(\lceil s_{TF}' \rceil + \lceil Tr_F' \rceil +\lceil u' \rceil +\lceil v' \rceil +\lceil w' \rceil)}\circ Tr_F' \circ \eta_{-(\lceil s_F' \rceil +\lceil \bar{u}' \rceil +\lceil \bar{v}'\rceil + \lceil \bar{w}' \rceil),0}
\end{align*}
 And so we have shown that for any exact triangle $\Delta= F \to G \to H \to TF$ in $\text{Split}(\cc_\infty)$, there exists a triangle of the form 
\begin{equation}\begin{tikzcd}
    F \ar[r,"u''"]& \ss^{-\lceil u' \rceil}G \ar[r,"v''"]& \ss^{-(\lceil u' \rceil + \lceil v' \rceil)}H \ar[r,"w''"]& \ss^{-(\lceil u' \rceil +\lceil v' \rceil +\lceil w' \rceil)}TF
\end{tikzcd}\end{equation}
which is a $p$-retract of the image strict exact triangle also of weight $(\lceil u' \rceil +\lceil v' \rceil +\lceil w' \rceil)$ . (Note since $\lceil s'_F \rceil = \lceil s_{TF}' \rceil$ we have $\lceil u' \rceil +\lceil v' \rceil +\lceil w' \rceil = \lceil \bar{u}' \rceil +\lceil \bar{v}' \rceil +\lceil \bar{w}' \rceil$ and we can also replace the central triangle $\Delta$ by $\Delta':=\ss^{-\lceil s_F \rceil}\Delta$).
\end{proof}


\newpage

\begin{thebibliography}{}

\bibitem[Ab]{Ab}

M. Abouzaid, \textit{A geometric criterion for generating the Fukaya category}, Publ.math.IHES 112, (2010),  191–240. https://doi.org/10.1007/s10240-010-0028-5.

\bibitem[AS]{AS}
D. Auroux,  I. Smith,\textit{Fukaya categories of surfaces, spherical objects and mapping class groups}, Forum of Mathematics, Sigma,  (2021),  9:e26. doi:10.1017/fms.2021.21.



\bibitem[Bo]{Bo}
F. Borceux,  \textit{Handbook of Categorical Algebra: Basic category theory}, Cambridge University Press, (1994).

\bibitem[BCZ]{BCZ}
 P. Biran, O. Cornea, J. Zhang, \textit{Triangulation, Persistence, and Fukaya categories},  (2023)  
https://doi.org/10.48550/arXiv.2304.01785.

\bibitem[BCS]{BCS} 	
    P. Biran, O. Cornea and E. Shelukhin, \textit{Lagrangian shadows and triangulated categories}.
Societe Mathematique de France, Asterique 426, (2021).
\bibitem[BD]{BD}
F. Borceux,  D. Dejean, \textit{Cauchy completion in category theory}. Cahiers de Topologie et Géométrie Différentielle Catégoriques 27 (1986),  133-146.




\bibitem[BK]{BK}
A. Bondal, M. Kapranov, \textit{Enhanced Triangulated Categories}. Math. USSR Sb.  (1991), 70 93.


 
 \bibitem[BM]{BM}
P. Bubenik, N. Milićević, \textit{Homological algebra for persistence modules}, Foundations of Com-
putational Mathematics, (2019), 1–46.

\bibitem[BSc]{BSc}
 P. Balmer, M. Schlichting,  \textit{Idempotent completion of triangulated categories}, Journal of Algebra, 236(2), (2001), 819–834. https://doi.org/10.1006/jabr.2000.8529 .

\bibitem[BS]{BS}
P. Bubenik, J. A. Scott, \textit{Categorification of persistent homology}, (2012),  https://arxiv.org/abs/1205.3669.


\bibitem[Ca]{Ca}
G. Carlsson, \textit{Topology and data}. Bull. Amer. Math. Soc. (N.S.), 46(2), (2009) , 255–308.
\bibitem[ELZ]{ELZ}
H. Edelsbrunner, D. Letscher, A. Zomorodian, \textit{Topological persistence and simplification}, In Proceedings 41st Annual Symposium on Foundations of Computer Science, IEEE, (2000),  454–463. 




 \bibitem[GZ]{GZ}
 P. Gabriel, M. Zisman, \textit{Calculus of fractions and homotopy theory}. Ergebnisse der Mathematik und ihrer
Grenzgebiete, Band 35, Springer, Berlin, (1967).

\bibitem[Ke]{Ke}
G. Kelly, \textit{Basic concepts of enriched category theory}. Reprints in Theory and Applications of Categories, No. 10,  (2005).


\bibitem[Ko]{Ko}
M. Kontsevich, \textit{Homological algebra of mirror symmetry}. In Proceedings of the International Congress of Mathematicians, Birkhäuser, (1994),
120–139.  https://arxiv.org/abs/
alg-geom/9411018v1.



\bibitem[Lu]{Lu}
J. Lurie,  \textit{Higher Topos Theory}. Annals of Mathematics Studies 170, Princeton Univ. Press, (2009).

\bibitem[Mi]{Mi}
J. Milne, \textit{Motives — Grothendieck’s Dream} Expository Article, (2012),
https://www.jmilne.org/math/xnotes/MOT.pdf.

\bibitem[Mit]{Mit}
B. Mitchell, \textit{The dominion of Isbell}. Trans. Amer. Math. Soc. 167, (1972), 319-331
https://doi.org/10.1090/S0002-9947-1972-0294441-0.
\bibitem[Ne]{Ne}
A. Neeman,  \textit{Triangulated categories}. Annals of Mathematics Studies, Princeton University Press,  (2001).

\bibitem[PZ]{PZ}
A. Polishchuk, E. Zaslow, \textit{ Categorical mirror symmetry: the elliptic
curve}. Adv. Theor. Math. Phys. 4, (2000),  1187–1207. 
https://arxiv.org/abs/math/9801119v3.
 
 \bibitem[PRSZ]{PRSZ}
 L. Polterovich, D. Rosen, K. Samvelyan, J. Zhang, \textit{Topological Persistence in Geometry and Analysis}. 	University Lecture Series, 74. American Mathematical Society, Providence, RI, (2020). 

 

 
 
 
 \bibitem[Sch]{Sch}
 O. M. Schnürer, \textit{Homotopy categories and idempotent completeness, weight structures and weight complex functors}.  (2011), https://arxiv.org/abs/1107.1227.


 \bibitem[Ta]{Ta}
G. Tabuada,
\textit{Theorie homotopique des DG-categories}.  (2007),
https://doi.org/10.48550/arXiv.0710.4303.




 \bibitem[Ve]{Ve}
J-L. Verdier, \textit{Des catégories dérivées des catégories abéliennes}. Astérisque, no. 239, (1996), 269 p. 

 \bibitem[We]{We}
C. Weibel, \textit{An Introduction to Homological Algebra}. Cambridge University Press, (1994). 









\bibitem[ZC]{ZC}
A. Zomorodian, G. Carlsson, \textit{Computing persistent homology}. Discrete Comput. Geom.,
33(2), (2005),  249–274, https://doi.org/10.1007/s00454-004-1146-y.












\end{thebibliography}
\end{document}